\documentclass[a4paper,12pt,reqno]{amsart}

\usepackage[normalem]{ulem}

\usepackage{fullpage}
\usepackage{amssymb}
\usepackage{amsmath}
\usepackage{mathtools}
\usepackage{enumitem}
\usepackage{mathrsfs}
\usepackage[all]{xy}
\usepackage{mathtools}
\usepackage{hyperref}
\usepackage{url}
\usepackage{bbm}
\usepackage{longtable}
\usepackage{dsfont}
\usepackage{upgreek}
\usepackage{lmodern}
\usepackage[OT2, T1]{fontenc}

\DeclareSymbolFont{cyrletters}{OT2}{wncyr}{m}{n}
\usepackage[british]{babel}

\usepackage[margin=1in]{geometry}
\numberwithin{equation}{section} \numberwithin{figure}{section}

\DeclareMathOperator{\Pic}{Pic} 
\DeclareMathOperator{\Gal}{Gal}

\DeclareMathOperator{\Spec}{Spec}
 \DeclareMathOperator{\rank}{rank}
 \DeclareMathOperator{\re}{Re}
   
\DeclareMathOperator{\vol}{vol} 
 \DeclareMathOperator{\Val}{Val}
\DeclareMathOperator{\Br}{Br} 

\DeclareMathOperator{\inv}{inv} \DeclareMathOperator{\res}{\partial}
 
 \DeclareMathOperator{\Sub}{Br_{\mathrm{Sub}}}
 \DeclareMathOperator{\Ind}{Ind}

 \DeclareMathOperator{\Norm}{N}

\DeclareMathOperator{\HH}{H}

\let\L\relax
\DeclareMathOperator{\L}{L} 
 
\DeclareMathOperator{\Eff}{\Lambda_{eff} }
\DeclareMathOperator{\BrV}{Br_{vert}} 
\DeclareMathOperator*{\Osum}{\sum{}^*}

\DeclareSymbolFont{cyrletters}{OT2}{wncyr}{m}{n}
\DeclareMathSymbol{\Sha}{\mathalpha}{cyrletters}{"58}
\DeclareMathSymbol{\Be}{\mathalpha}{cyrletters}{"42}

\renewcommand{\O}{\mathcal{O}}
\newcommand{\OO}{\mathcal{O}}

\newcommand{\one}{\mathbbm{1}}
\newcommand{\x}{\mathbf{x}}
\newcommand{\y}{\mathbf{y}}

\newcommand\F{\mathbb{F}}
\renewcommand\P{\mathbb{P}}
\newcommand\Z{\mathbb{Z}}
\newcommand\N{\mathbb{N}}
\newcommand\Q{\mathbb{Q}}
\newcommand\R{\mathbb{R}}
\newcommand\C{\mathbb{C}}
\newcommand\G{\mathbb{G}}

\newcommand\Gm{\G_\mathrm{m}}

\newcommand{\Adele}{\mathbf{A}}

\newcommand{\br}{\mathscr{B}}

\newcommand{\A}{\mathbb{A}}

\newcommand\ZZ{\mathbb{Z}}

\newcommand\QQ{\mathbb{Q}}
\newcommand\RR{\mathbb{R}}
\newcommand\CC{\mathbb{C}}

\renewcommand{\mod}[1]{ \ \left(\textnormal{mod}\ #1\right)}
\newcommand{\md}[1]{  \left(\textnormal{mod}\ #1\right)}
 
\renewcommand{\l}{\left}

\renewcommand{\r}{\right}
\renewcommand{\b}{\mathbf} 

\renewcommand{\c}{\mathcal} 
\renewcommand{\gcd}{\textrm{gcd}} 
\renewcommand{\leq}{\leqslant}
\renewcommand{\geq}{\geqslant}
\renewcommand{\#}{\sharp}

\newcommand{\fp}{\mathfrak{p}}

\newtheorem{lemma}{Lemma}
\newtheorem{conjecture}[lemma]{Conjecture}
\newtheorem{theorem}[lemma]{Theorem}
\newtheorem{proposition}[lemma]{Proposition}
\newtheorem{corollary}[lemma]{Corollary}

\theoremstyle{definition}
\newtheorem{example}[lemma]{Example}

\newtheorem{definition}[lemma]{Definition}
\newtheorem{remark}[lemma]{Remark}

\newtheorem*{ack}{Acknowledgements}

\usepackage[usenames,dvipsnames]{color}

\newcommand{\dan}[1]{{\color{blue} \sf $\clubsuit\clubsuit\clubsuit$ Dan: [#1]}}

\numberwithin{lemma}{section}

\newcommand{\beq}[2]
{
\begin{equation}
\label{#1}
{#2}
\end{equation}
}

\begin{document}

\title[Leading constant]{The leading constant for rational points in families}

\begin{abstract}
	We prove asymptotics for Serre's problem on the number of diagonal planar conics with a rational point and use this to put forward a new conjecture on counting the number of varieties in a family which are everywhere locally soluble.
  \end{abstract}

\author{Daniel Loughran} 
  \address{
  Department of Mathematical Sciences \\
University of Bath \\
Claverton Down \\
Bath \\
BA2 7AY \\
UK}
\urladdr{https://sites.google.com/site/danielloughran}

\author{Nick Rome} 
\address{Graz University of Technology\\ Institute of Analysis and Number Theory\\
Kopernikusgasse 24/II \\8010 Graz\\ Austria}
\email{rome@tugraz.at}

\author{Efthymios Sofos} 
\address{
Department of Mathematics\\
University of Glasgow  \\ G12~8QQ \\ UK}
\email{efthymios.sofos@glasgow.ac.uk}




\maketitle

\thispagestyle{empty}

\tableofcontents

\section{Introduction}

The  fundamental problem in Diophantine Geometry is to determine solubility of Diophantine equations over the rationals.
Given the difficulty of this problem, one can instead ask this for ``random'' equations
as done by  
Bhargava for  random ternary cubic curves~\cite{Bha14}
or
Browning--Le Boudec--Sawin~\cite{boudec}
in their 
 resolution of the 
Poonen--Voloch conjecture \cite{PV04}.
A key step in  these questions 
 is to understand the probability 
of a random equation being everywhere locally soluble, i.e.~soluble over $\R$ and $\Q_p$ for all $p$. The naive guess would be to consider the product over all places 
of the probabilities that a random equation in the family is locally soluble; for example, as in the work of 
 Bhargava--Cremona--Fisher--Jones--Keating for quadrics~\cite{MR3544620}. However, this guess
can fail even for simple-looking families.
As an example,
define 
\beq{def:nalphax}{
 N(B)
= \#\bigg\{\b t \in (\Z\setminus\{0\} )^3: 
\begin{array}{ll}
& \gcd(t_0,t_1,t_2) = 1, \,\, \max_i|t_i| \leq B, \\
& t_0 x_0^2 + t_1x_1^2 + t_2x_2^2 = 0 \textrm{ has a }\Q\textrm{-point}
\end{array}
\bigg\}.
}
The problem of 
proving asymptotics for $N(B)$ was originally 
raised by Manin in a letter to Serre, who proved in~\cite[page 399]{Ser90} 
that there exists a positive constant $c_1$ such that for all $B\geq 2 $, one has 
 $ N(B) \leq c_1  B^3 (\log B)^{-3/2}$. Subsequently, Hooley~\cite{Hoo93} and Guo~\cite{Guo95}
independently proved that   there exists a positive constant $c_2$ such that for all $B\geq 2 $, one has 
 $ N(B) \geq c_2  B^3 (\log B)^{-3/2}$. Using different arguments, we prove an asymptotic:

\begin{theorem}
\label{thm:main}
 For all $B\geq 2 $,  we have 
$$
N(B) = \frac{2 \cdot c_\infty \prod_p c_p}{\pi^{3/2}   }  
 \frac{B^3}{(\log B)^{3/2}}+O\l(\frac{B^3}{(\log B)^{5/2}}  \r),$$
 where the product is over all primes $p$ and 
$$
c_\infty = 6, \quad c_2 = \frac{49/48}{(1-1/2)^{1/2}}, \quad
c_p = 
\l(1+\frac{1}{p}+ \frac{1}{p^2} \r)
 \frac{    ( p^2 + p/2 + 1 )}{    (p + 1)^2   (1-1/p )^{1/2} }, \quad p \neq 2.$$
 \end{theorem} 
 
The factors $c_\infty$ and $c_p$ admit interpretations as local densities (for example $c_\infty = 6$ as six choices of the signs of the $t_i$ give a conic with a real point). However, the extra prefactors $2$ and $\pi^{-3/2}$ make clear that merely taking the product of local densities fails to give the right constant even for simple equations.  Tschinkel asked  in \cite[\S 2.1]{Tsc09} whether the leading constant in Theorem~\ref{thm:main} has an algebro-geometric explanation, in the context of Guo's result \cite{Guo95}.  The answer is as follows: The factors $c_\infty$ and $c_p$ can be interpreted using Peyre's Tamagawa measure \cite{Pey95}, with the terms $(1-1/p)^{1/2}$ being convergence factors.
The factor $2$ is the order of a \textit{subordinate Brauer group}
and $\pi^{3/2} =\Gamma(1/2)^{3}$.

An asymptotic for a function similar to $N(B)$ without any coprimality conditions on the $\b t $ is given in Remark~\ref{rem:equiv}.
The proof of Theorem~\ref{thm:main} is given in \S\ref{sgenguoresthrme9}.
\begin{remark}
[New constants]
Products of Gamma function values have not appeared before in 
problems of this kind nor in the neighbouring area of the Batyrev--Manin conjecture.
Tauberian theorems (e.g.~Delange's Tauberian theorem \cite{Del54}) give only a single Gamma factor.  
Loosely speaking,
it will transpire that
there are  certain families of equations
where 
each singular equation gives its own Gamma factor in the product and it would be desirable to
have Tauberian theorems
for completely general  multivariate Dirichlet series; special cases have been dealt with by
la Bret\`eche \cite{MR1858338}
and Essouabri \cite{MR1450422}. 
An unexpected feature is that 
there are also families of Diophantine equations 
where 
one instead 
has  a single Gamma value despite having multiple singular equations.
This distinction  
is decided by the geometry of the family, see the definition of $\Gamma(X, f) $ 
in Conjecture \ref{conj:main}
for a precise statement. It would be useful to have more results for fibrations over other bases to get a better understanding of this distinction.
 \end{remark}
  \begin{remark}
[Previous methods]
Serre~\cite{Ser90} proved  
  $N(B) \ll  B^3/(\log B)^{3/2}$ by employing 
  the large sieve.
He also restricted to three  prime coefficients to prove the non-matching
 lower bound
$N(B) \gg  B^3/(\log B)^{3}$.
Subsequently, Hooley~\cite{Hoo93} 
used  Burgess's estimate for  character sums and the Brun sieve
to   prove a lower bound of the same order of magnitude as the upper bound proved by Serre.
  Hooley's method starts      from restricting attention to diagonal coefficients of special shape: odd, square-free and coprime in pairs. 
Guo~\cite{Guo95}, using  similar arguments, proved  an asymptotic for the modified version of   $N(B)$, where 
each $t_i$ is square-free and all pairs $t_i, t_j$ are coprime for all $i\neq j $.  
\end{remark}

\begin{remark}
[Comparison to Guo's work]
Absorbing square factors into the diagonal coefficients 
leads to  square-free coefficients, as in the setting of 
  Guo. However, this causes  
technical complications as the change of variables 
alter the height and one ends up with counting in lopsided boxes.
We thus chose to  follow a treatment that is different than the one used by Guo and Hooley, namely,
we  replaced
 the use of Burgess's character sum bounds
and all sieve theoretic arguments  by    bilinear  sum estimates for quadratic Kronecker symbols.
Our results improve the error term in 
 Guo's work, see Remark~\ref{rem:erterm}.
This requires the adaptation of established techniques, formerly confined to simple contexts and problems, to address a significantly more challenging problem.
\end{remark}

There is a  rapidly  developing industry 
in this area:
\begin{center}
    \begin{longtable}{| l | l | l | l |}
    \hline
   \textbf{Type of random varieties} &  \textbf{Authors}  \\ \hline
   Thue and Fermat equations &  Akhtari--Bhargava \cite{MR3928038},  Browning--Dietmann \cite {MR2537095}, 
\\ &  Dietmann--Marmon \cite{MR3312095} 
    \\ \hline

Conic bundles over elliptic curves & Bhakta--Loughran--Rydin Myerson--Nakahara \cite{arXiv:2109.03746}
      \\ \hline

Plane cubics & Bhargava \cite{Bha14}        \\ \hline 
   Quadrics & Bhargava--Cremona--Fisher--Jones--Keating \cite{MR3544620}       \\ \hline

Negative Pell equation &  Blomer \cite{MR3642347},  Fouvry--Kl\"{u}ners \cite {MR2726105}, 
\\ &  Koymans--Pagano \cite{arXiv:2201.1342} \\ \hline 

Del Pezzo and K3 surfaces & Bright--Browning--Loughran \cite{MR3530447},  Bright \cite{MR3858417}, \\
 & Browning \cite{MR3731306}, Gvirtz--Loughran--Nakahara   \cite{MR3534973}, \\
 & Mitankin--Salgado \cite{MR4454468},  Santens \cite{arXiv:2201.04573}   \\ \hline 
{M}arkoff surfaces  &  Colliot-Th\'{e}l\`ene--Wei--Xu \cite{MR4288633}, Dao \cite{arXiv:2202.07142},
\\ &  Ghosh--Sarnak   \cite{MR4448994}, Loughran--Mitankin \cite{MR4320804}
    \\ \hline  
   {C}h\^{a}telet surfaces &  de la Bret\`eche--Browning \cite{MR3198755}, Rome \cite{MR3976470}    \\ \hline 

Coflasque tori &  de la Bret\`eche--Browning \cite{MR3232765}  \\ \hline 
Fibrations over quadrics & Browning--Heath-Brown \cite{BHB21}, Browning--Loughran   \cite{BL19}, 
\\
&  Browning--Lyczak--Sarapin \cite{arXiv:2203.06881},
Wilson \cite{wilson1}
\\ \hline 
   Fano hypersurfaces &  Browning--Le Boudec--Sawin \cite {boudec},  
\\ &  Poonen--Voloch \cite{PV04}, Le Boudec \cite {MR4261096} 
\\ \hline 

Norm equations & Browning--Newton \cite{BN15}        \\ \hline

Diagonal Fano hypersurfaces & Br\"{u}dern--Dietmann \cite{MR3177289}, Chow~\cite{MR3253291},      
\\ & Hirakawa--Kanamura \cite{MR4322838}  \\ \hline 

Bihomogenous hypersurfaces  &
Fisher-Ho-Park \cite{FHP21} \\ \hline
Special plane conics  & Friedlander--Iwaniec   \cite{MR2675875}     \\ \hline 
 Diagonal plane conics & Hooley \cite{Hoo93}, Guo \cite{Guo95}     \\ \hline
   General plane conics & Hooley \cite{Hoo07}    \\ \hline 
  Fibrations over algebraic groups & Huang \cite{Hua21}, Loughran \cite{Lou18}, \\
  &  Loughran--Takloo-Bighash--Tanimoto  \cite{LTBT20} \\ \hline
  Fibrations over $\P^n$ &  Loughran--Matthiesen \cite {LM19}, Loughran--Smeets \cite{LS16}    \\ \hline 
Affine quadric surfaces &  Mitankin \cite{MR3664529}, Santens \cite{MR4450612}    \\ \hline 
Hyperelliptic curves &  Poonen--Stoll \cite{MR3245014}    \\ \hline 
Conic bundles over $\P^n$ &  Serre \cite{Ser90}, Sofos \cite{Sof16}    \\ \hline 
   Generalised {C}h\^{a}telet varieties &  Balestrieri--Rome \cite{BRbs}, Browning--Sofos--Ter\"av\"ainen \cite{bst},
\\ &
Skorobogatov--Sofos \cite {MR4542704}    \\ \hline 
 Conic bundles over hypersurfaces &  Sofos--Visse-Martindale \cite{MR4353917}    \\ \hline 
\end{longtable}
\end{center} 
\vspace{-20pt}

There is, however, no conjecture regarding the leading constant in the literature. 
In this paper we put forward a general conjecture  on the leading constant for problems of this type (Conjecture \ref{conj:main}). Our conjecture is based on that of Peyre~\cite{Pey95} on the leading constant in the Batyrev--Manin conjecture,
however, some key features are different, for example, multiple Gamma factors and modified Fujita invariants.  
The setting
will not require explicit equations nor the Hasse principle to hold for the equations in the family; 
it will apply to  fairly general 
   fibrations over any projective space $\P^n$.

We now explain our new conjecture by working in a more geometric framework.

\subsection{A new conjecture}
Let $f:Y \to X$ be a dominant morphism of smooth projective varieties over a number field $k$ with geometrically integral generic fibre.
In this paper we are interested in the function
$$N_{\textrm{loc}}(f,B) =  \#\{ x \in X(k) : H(k) \leq B, x \in f(Y(\Adele_k))\}$$
which counts the number of everywhere locally soluble fibres of the morphism $f$ and $H$ is a height function on $X$.
Under suitable assumptions on $f: Y \to X$ and $H$, and possibly after removing a thin set of rational points, we conjecture the asymptotic formula
\begin{equation} \label{eqn:conj_intro}
	c_{f,H} B (\log B)^{\rho(X) - \Delta(f) - 1}.
\end{equation}
Here $\rho(X) = \rank \Pic X$. The invariant $\Delta(f)$ was originally introduced in \cite{LS16}; we recall its definition in \S \ref{sec:set-up}. The main new part of the conjecture is the constant $c_{f,H}$, which is a product of the following (see Conjecture \ref{conj:main} for precise statements):

\begin{enumerate}
	\item Modified definition of Peyre's effective cone constant.
	\item Products of values of the Gamma function.
	\item Modified Fujita invariants.
	\item Order of a subordinate Brauer group.
	\item Adelic Tamagawa volume.
\end{enumerate}

Factors (2) and (3) are completely new and the main challenge was to  identity these factors. For (1), Peyre's original definition \cite[Def.~2.4]{Pey95} of the effective cone constant $\alpha(X)$ contains an extraneous factor $(\rho(X) - 1)!$, which we view as a value of the Gamma function and treat separately from the effective cone. The subordinate Brauer group and Tamagawa measures in (4) and (5) appeared in some form in \cite{Lou18}; our new contribution is a definition of the subordinate Brauer group for more general $f$, and to construct virtual Artin $L$-function associated to the  Picard group and $f$ which give convergence factors for the product of local Tamagawa measures. 
Our conjecture is greatly inspired by the conjectures of Manin \cite{FMT89}, Batyrev--Manin \cite{BM90}, Peyre \cite{Pey95}, and Loughran--Smeets \cite{LS16}, as well as the works of Batyrev--Tschinkel \cite{BT95}, Salberger \cite{Sal98}, and Loughran \cite{Lou18}. 

We verify that the conjecture holds in the following cases, for suitable $f$:

\begin{enumerate}
	\item[(i)] Families with $\Delta(f) = 0$.
	\item[(ii)] Equivariant compactifications of anisotropic tori \cite{Lou18}.
	\item[(iii)] 
	Wonderful compactifications of adjoint semi-simple algebraic groups \cite{LTBT20}.
	\item[(iv)] Compatibility with the circle method.
	\item[(v)] The family of all plane diagonal conics (Theorem~\ref{thm:main}).
\end{enumerate}

Cases (i)--(iii) involve showing that asymptotics in the literature agree with our conjecture. For (iv) we explain how the conjecture agrees with the prediction from the circle method for representing a polynomial as a sum of two squares.

Case (v) is the new 
case obtained in our paper. It was 
 worked out 
by the second and third-named authors.
The  first named-author was working independently on formulating  a  conjecture for the  leading constant for general fibrations
and this case 
  was the key missing example required for the   formulation.
The known  cases in the literature do not have the following features, whereas the fibration in Theorem~\ref{thm:main}
does:
the smooth locus admits non-constant invertible functions and the family corresponds to a transcendental Brauer group element.
 This made clear that it was still Peyre's effective cone constant which should appear, gave the formula for the relevant special values of the Gamma function, and led to the modified Fujita invariants through change-of-height considerations.

\subsection{Open cases of the conjecture and other examples}
We  conjecture the asymptotic \eqref{eqn:conj_intro} under  
specific
 geometric assumptions of $f$. Firstly we assume that the fibre over each codimension $1$ point of $X$ admits an irreducible component of multiplicity $1$; this is necessary in general, as otherwise the set of everywhere locally soluble fibres can be finite \cite[Thm.~1.4]{LM19}. We also want $X$ to have lots of rational points, so we assume that $X$ is \emph{weak Fano}. More restrictively, let $U \subseteq X$ denote the open subset given by the complement of those divisors lying below the non-split fibres. Then we assume that
\begin{equation} \label{eqn:assump_geometric}
	\text{either} \quad \rho(X) = 1 \quad \text{or} \quad k[U]^\times = k^\times,
\end{equation}
where $k[U]^\times$ denotes the group of invertible regular functions on $U$. If either of these conditions are not satisfied, then in fact the stated asymptotic \eqref{eqn:conj_intro} need \emph{not hold}.

\begin{example} \label{ex:crazy}
	One can show without difficulty (see \S\ref{sec:elementary})  that
	$$\#\{ t \in \Q : H(t) \leq B, t \text{ is a sum of two squares} \} = c B/(\log B) + O(B / (\log B)^2) $$
	where the leading constant $c$ agrees with the conjecture and 
	$H(t)$ denotes the anticanonical height on $\P^1$, which is the square of the usual naive height.	
	(Here the relevant family of varieties is given by a smooth proper model of the conic bundle
	$t = x_1^2 + x_2^2$ over $\P^1$.)
	Using this and Dirichlet's hyperbola method, one shows that 
	$$\#\left\{ t_1,t_2 \in \Q^2 : 
	\begin{array}{l}
		H(t_1)H(t_2) \leq B,  \\
		\text{ each } t_i \text{ is a sum of two squares} 
	\end{array} 
	\right\} \sim 2c^2 B (\log \log B)/(\log B).$$
	This counting problem takes place on $\P^1 \times \P^1$ with $\rho(\P^1\times \P^1) = 2$
	and $\Delta(f) = 2$; this gives the correct factor of $\log B$ in
	\eqref{eqn:conj_intro}, but the factor $\log \log B$ is completely 
	unexpected! 	This example defines a split toric variety, hence shows that the
	asymptotic proved in \cite{Lou18} does not hold for non-anisotropic toric varieties in general.
	Note that the assumption \eqref{eqn:assump_geometric} does not hold here
	since $\rho(\P^1\times \P^1) = 2 > 1$ and $k[\Gm^2]^\times/k^\times = \Z^2 \neq 0$.

	More exotic asymptotic formulae can be obtained by taking other suitable products, for example
	one can show that
	$$\#\left\{ t_0,t_1,t_2 \in \Q^3 : 
	\begin{array}{l}
	H(t_0)H(t_1)H(t_2) \leq B, \\
	\text{both } t_1,t_2 \text{ are a sum of two squares}
	\end{array}
	\right \} \asymp B (\log \log B)^2.$$
\end{example}

Given these outlandish asymptotic formulae with currently little geometric explanation, we  focus on examples where we can give precise predictions, with the key geometric assumption being \eqref{eqn:assump_geometric}. This assumption holds when $X = \P^n$, which is our primary base of interest  and the case studied by Serre \cite{Ser90}. Given the lack of understanding of the Hasse principle in general, we also focus our attention on counting only everywhere locally soluble varieties. In important special cases, for example conic bundles or more generally families of products of Brauer--Severi varieties, the Hasse principle holds and the conjecture indeed concerns the number of varieties in the family with a rational point.

\begin{example}
	A natural next case to consider after Theorem \ref{thm:main} would be Fermat curves,
	i.e.~for $d \geq 3$ the counting problem
	$$\#\{ t \in \P^2(\Q): H(t) \leq B, t_0x_0^d + t_1x_1^d + t_2x_2^d = 0
	\text{ is everywhere locally soluble}\}.$$
	Conjecture \ref{conj:main} predicts a precise asymptotic formula for this
	counting problem of the shape
	$c_d B^3 (\log B)^{-\Delta(d)}.$ 
	The value of $\Delta(d)$ can be found in \cite[\S 5.4]{LS16}.
	(For $p$ prime we have $\Delta(p) = 3/p$.) An explicit value for $c_d$
	can be obtained using a similar method to the one used in \S \ref{sec:diag_constant}.
	Proving this for any $d \geq 3$ seems very challenging. A more difficult problem
	still is to count the number which have a rational point. A naive guess would be that
	for $d \geq 4$ only $0\%$ of the everywhere locally soluble members have a rational
	point, whereas for $d = 3$ a positive proportion of the everywhere locally soluble
	members has a rational point. (The case $d=3$ is a famous family studied by 
	Selmer \cite{Sel51}).
\end{example}

\begin{example}
	Another natural case to consider for the conjecture is \emph{conic bundles}.
	The correct lower bound for the conjecture was proven in \cite{Sof16} in the 
	case of conic bundles over the projective line with at most three singular fibres
	over the algebraic closure. An explicit family of such surfaces is given by
	$$y_0 Q_0(x) + y_1Q_1(x) = 0 \quad \subset \P^1 \times \P^2$$
	where each $Q_i$ is a ternary quadratic form. Providing it is smooth,
	such an equation defines a del Pezzo surface of degree five.
\end{example}

\begin{remark}
The asymptotic in Example \ref{ex:crazy} gives a counter-example to the statement of \cite[Prop. 2]{FMT89}; this claims that if two counting functions grow like $B(\log B)^r$ and $B(\log B)^s$ then the counting function of the product grows like $B(\log B)^{r + s + 1}$. An inspection of the proof of \emph{loc.~cit.} reveals that the authors are using the integral representative of the Beta function which is only valid for $r> -1$ and $s> -1$, therefore the result is in fact only proved under this additional assumption.
\end{remark}

\subsection{Structure of the paper}
In \S \ref{sec:Sub_Brauer} we develop the theory of the subordinate Brauer group, which is a key tool in our work. In \S \ref{sec:conj}, we set-up and state our new conjecture (Conjecture~\ref{conj:main}). We also put forward an alternative version of the conjecture (Conjecture \ref{conj:Br}) in the case of zero-loci of Brauer groups which is often easier to work with in practice. In \S \ref{sec:verify}, we verify that our conjecture agrees with results in the literature and the circle method, and also consider some simple new cases. We also verify that the statement of Theorem \ref{thm:main} agrees with Conjecture \ref{conj:main}.
The final \S \ref{sgenguoresthrme9}  is dedicated to our main new result (Theorem \ref{thm:main}).

\begin{ack}
The work behind Theorem~\ref{thm:main} started while NR was visiting ES in the Max Planck Institute in Bonn, the hospitality of which is greatly acknowledged. 
Large part of the work related to Theorem~\ref{thm:main}    was undertaken  during the trimester ``Reinventing Rational Points'' in Institut Henri Poincar\'e in Paris. 
We would like to thank  David Harari, Emmanuel Peyre and Alexei Skorobogatov for organising the trimester. 
The work related to the conjecture was completed
at the ICMS in Edinburgh during the conference ``Rational points on higher-dimensional varieties''; we are grateful to the ICMS for their support and ideal working conditions. We thank Julian Lyczak, Emmanuel Peyre and Alexei Skorobogatov for useful discussions, 
Tim Browning for comments on an earlier draft,
John Cremona for providing computational data
and Sho Tanimoto and Florian Wilsch for discussions on Fujita invariants. DL was supported by UKRI Future Leaders Fellowship \texttt{MR/V021362/1} and ES was 
supported by  
EPSRC
 New Horizons  
grant \texttt{EP/V048236/1}.

\end{ack}

\section{Subordinate Brauer groups} \label{sec:Sub_Brauer}
In this section we define the subordinate Brauer group of a morphism of varieties and study its basic properties. This Brauer group plays a key role in our conjecture. 

\subsection{Brauer groups recap}

The canonical reference for Brauer groups is now the book \cite{Brauer}. For a scheme $X$ we denote its Brauer group by $\Br X :=\HH^2(X,\Gm)$. Let $X$ be an integral scheme over a field $k$ of characteristic $0$. We denote by $\Br_1 X := \ker( \Br X \to \Br X_{\bar{k}})$ its algebraic Brauer group. If $X$ is regular then Grothendieck's purity theorem \cite[Thm.~3.7.2]{Brauer} yields an exact sequence
\begin{equation} \label{seq:purity}
	0 \to \Br X \to \Br k(X) \to \bigoplus_{D \in X^{(1)}} \HH^1(k(D), \Q/\Z)
\end{equation}
where $\res_D: \Br k(X) \to \HH^1(k(D), \Q/\Z)$ is the residue map at $D$. If $\res_D(b) = 0$ then we say that $b$ is \emph{unramified} along $D$. 

If $k$ is a number field, there is the fundamental exact sequence from class field theory
\begin{equation} \label{seq:CFT}
	0 \to \Br k \to \bigoplus_v \Br k_v \to \Q/\Z \to 0
\end{equation}
where the last map is the sum of all local invariants $\inv_v$. We call the pairing
\begin{equation} \label{def:BM_pairing}
X(\Adele_k) \times \Br X \to \Q/\Z, \quad ( (x_v)_{v \in \Val(k)}, b) \mapsto 
\sum_v \inv_v b(x_v),
\end{equation}
the \emph{Brauer--Manin} pairing. For a subset $\br \subset \Br X$ we denote by $X(\Adele_k)^{\br}$ the left kernel of this for $\br$. We use throughout that $\Br \P^n_k = \Br \A^n_k = \Br k$ \cite[\S 6.1]{Brauer}.

\subsection{Definition}
The terminology ``subordinate Brauer group'' was first used by Serre in the appendix of \cite[Ch.~II]{Ser97} in the case were $X = \P^1$ and working with an element of $\Br k(\P^1)$. In \cite[\S 2.6]{Lou18} this was generalised to other $X$ and a finite collection of Brauer group elements. We put this into an even more general setting of the subordinate Brauer group of a morphism of varieties. Let $k$ be a field of characteristic $0$.

\begin{definition} \label{def:Sub}
	Let $f: Y \to X$ be a proper morphism of regular integral noetherian schemes over $k$
	with geometrically integral generic fibre. We define
	$$\Sub(X,f) = \bigcap_{D \in X^{(1)}} \left\{ \alpha \in \Br k(X) : 
	\begin{array}{l}
	\res_E f^* \alpha = 0 \text{ for all  irreducible components }\\
	E \subseteq f^{-1}(D) \text{ of multiplicitiy } 1
	\end{array}\right\}.$$
\end{definition}
The multiplicity $1$ irreducible components of $f^{-1}(D)$ are exactly those which are generically reduced (but not necessarily geometrically irreducible). We have
\begin{equation} \label{lem:Vert}
	\BrV(Y/X) \subseteq 
	f^*\Sub(X,f),
\end{equation}
where  $\BrV(Y/X) = f^* \Br k(X) \cap \Br Y$ is the vertical Brauer group of $Y/X$ \cite[\S 11.1]{Brauer}.

\begin{remark}
	The notation we use here $\Sub(X,f)$ is different to that of \cite[\S 2.6]{Lou18}
	which instead used $\mathrm{Sub}(X,\br)$ (in the context of zero-loci
	of Brauer groups). We decided to change to the former
	as it mirrors existing notation for the vertical Brauer
	group. Warning however: $\Sub(X,f) \not \subseteq \Br X$ in general!
	
	An additional kind of subordinate group was introduced in \cite[\S 2.6]{Lou18}
	and denoted by $\mathrm{Sub}(k(X),\br)$; this is a birational invariant
	of the generic fibre of $f$ (cf.~Lemma \ref{lem:Sub_birational})
	and pulls back to exactly the vertical Brauer
	group of the corresponding family of products of Brauer--Severi varieties 
	\cite[Thm.~2.11]{Lou18}. Despite this having better birational properties,
	in the literature relevant to Conjecture \ref{conj:main}
	it is the group $\Sub(X,f)$ which appears.
\end{remark}

\subsection{Properties}

\begin{lemma} \label{lem:flat}
	Assume that  $f: Y \to X$ is flat and 
	that for all $D \in X^{(1)}$ every irreducible component of $f^{-1}(D)$
	has multiplicity $1$. Then $\BrV(Y/X) = f^*\Sub(X,f)$.
\end{lemma}
\begin{proof}
	By purity \eqref{seq:purity}, it suffices to show that any element 
	$\alpha \in f^*\Sub(X,f)$ is unramified at all
	codimension $1$ points of $Y$. Let $D \in Y^{(1)}$. If $f(D)$ equals the generic
	point of $X$, then certainly $\alpha$ is unramified at $D$. Otherwise, since $f$ is flat,
	we find that $f(D)$ is a codimension one point of $X$. Then by Definition \ref{def:Sub},
	we see that $\alpha$ is also unramified at $D$, as required.
\end{proof}


\begin{lemma} \label{lem:Sub_formula}
$$\Sub(X,f)=\left\{ b \in \Br k(X) :
\begin{array}{l}
\res_D(b) \in \ker( \mathrm{res} : \HH^1(k(D),\Q/\Z) \to \HH^1(\kappa_E, \Q/\Z)) \\
 \text{for all } D \in X^{(1)} \text{ and all }  E \subseteq f^{-1}(D) \text{ of multiplicitiy } 1
\end{array}\right\}$$
where $\mathrm{res}$ denotes the usual restriction map on Galois cohomology
and $\kappa_E$ denotes the algebraic closure of $k(D)$ inside the function field of $E$.
\end{lemma}
\begin{proof}
	This follows by simply writing down and comparing the residue exact
	sequences for $X$ and $Y$ (cf.~\cite[Prop.~11.1.5]{Brauer}).
\end{proof}

$\Sub(X,f)$ is a birational invariant, only depending on $X$ and the generic fibre of $f$.

\begin{lemma} \label{lem:Sub_birational}
	Let $f:Y \to X$ and $g:Z \to X$ be proper morphisms of regular integral schemes over $k$
	with geometrically integral generic fibre and let 
	$h:Y \dashrightarrow Z$ be a birational map over $X$. Then pull-back via $h$ induces an
	isomorphism
	$$\Sub(X,f) \cong \Sub(X,g).$$
\end{lemma}
\begin{proof}
	To prove the result we may assume that $X= \Spec R$ is the spectrum of a discrete
	valuation ring. The result then follows immediately from Lemma \ref{lem:Sub_formula}
	and 	\cite[Cor.~10.1.13]{Brauer}.
\end{proof}

\begin{lemma} \label{lem:codim_1_finite}
	If the fibre over every codimension $1$ point of $X$ contains
	an irreducible component of multiplicity $1$, then 
	$\Sub(X,f)/\Br X$ is finite.
\end{lemma}
\begin{proof}
	For any finite separable field extension $K \subset L$, the kernel
	of the restriction map $\HH^1(K,\Q/\Z) \to \HH^1(L, \Q/\Z)$ is finite.
	Thus by Lemma \ref{lem:Sub_formula}, 
	for any $b \in \Sub(X,f)$ there are only finitely possibilities
	for $\delta_D(b)$ as $D$ varies over the finitely many $D \in X^{(1)}$
	such that $f^{-1}(D)$ is non-split. Since the subgroup 
	$\Br X \subset \Sub(X,f)$ consists of exactly those elements with trivial
	residue at all $D$ by \eqref{seq:purity}, the result follows.
\end{proof}

\begin{remark}
	If there are codimension $1$ fibres with no irreducible component of multiplicity $1$,
	then Definition \ref{def:Sub} is not the correct object (problems of this
	type appear in \cite{BLS22}). One should 
	probably focus on the irreducible components of minimal multiplicity $m_E$
	and ask that $m_E \res_E f^*\alpha$ vanishes instead.
	Brauer groups of this type arise in the study of Campana points; see for example
	\cite[Def.~3.12]{MNS} and \cite[Def.~8.1]{wonderful}.
\end{remark}

\subsection{Families of Brauer--Severi varieties} \label{sec:family_BS}
The subordinate Brauer group can be difficult to calculate in general, as the formula in Lemma \ref{lem:Sub_formula} requires explicit knowledge of a smooth proper model. For families of Brauer--Severi varieties or their products, there is however an alternative formula which is easier to work with.

Let $X$ be a regular integral noetherian scheme over a field $k$ of characteristic $0$ which admits an ample line bundle. Let $U \subseteq X$ be an open subset and $\br \subset \Br U$ a finite subset. We denote by $\langle \br \rangle$ the finite subgroup generated by $\br$. We have the following version of the subordinate Brauer group in this setting.

\begin{definition} \label{def:br_Sub}
	We say that $b \in \Br k(X)$ is \emph{subordinate} to $\br$
	with respect to $X$, if for each $D \in X^{(1)}$ the residue $\res_D(b)$ lies in 
	$\res_D(\langle\br \rangle)$. We let
	$$\Sub(X,\br)=\{ b \in \Br k(X) : \res_D(b) \in \res_D(\langle\br\rangle) \text{ for all } D \in X^{(1)}\},$$
	denote the group of all such elements.
\end{definition}

As $X$ admits an ample line bundle, a theorem of Gabber \cite[Thm.~4.2.1]{Brauer} implies that every element $b \in \Br U$ is the Brauer class of some Brauer--Severi scheme $V_b$  over $U$. (Recall that a Brauer--Severi scheme over $U$ is a scheme over $U$ which is \'etale locally isomorphic to $\P^n_U$ for some $n$.) We take $f:Y \to X$ to be a smooth proper model of the fibre product $\times_{b \in \br} V_b \to U$ over $U$; this exists by Hironaka's theorem on resolution of singularities. Take $V= f^{-1}(U)$. In the statement $\delta_D(f)$ denotes the $\delta$-invariant from \cite{LS16}; we recall the definition in \S \ref{sec:set-up}.

\begin{proposition} \label{prop:calculate_br_Sub}
	In the above setting, the following holds.
	\begin{align}
	&1/|\res_D(\langle\br\rangle)| = \delta_D(f) \text{ for all } D \in X^{(1)}, \label{eqn:res=delta} \\
	&\Sub(X,\br) = \Sub(X,f) \label{eqn:Sub=Sub}.
	\end{align}
\end{proposition}
\begin{proof}
We have to be slightly careful, since the fibre product of smooth proper models of each $V_b$ will not necessarily give a smooth proper model of $\times_{b \in \br} V_b$, e.g.~the fibre product of two conic bundles over $\P^1$ which have a singular fibre over a common point is non-regular. This introduces some technical aspects into our proof.

To prove the result we may assume that $X = \Spec R$ for a discrete valuation ring $R$. We will use some of the arguments and constructions in \cite[\S 2.4, \S2.6]{Lou18}. We first  construct an explicit model $Y$ (this is permissible by Remark \ref{rem:birational_invariant}). Artin \cite[Thm.~1.4]{Art82} has constructed	regular flat proper integral schemes $\mathcal{V}_b \to \Spec R$	whose generic fibres are isomorphic to $V_b$ and whose special fibres are integral, for each $b \in \br$. Frossard \cite[Prop.~2.3]{Fro97} has shown that the algebraic closure of $k$ inside the function field of the special fibre of $\mathcal{V}_b$ is exactly the cyclic field extension $k_b$ of $k$ determined by the residue $\res_D(b)$. We let $K/k$ be the compositum of the $k_b$.  We take $\mathcal{V} = \prod_{b \in \br} \mathcal{V}_{b} \to X$ and take $Y$ to be a desingularisation of $\mathcal{V}$.

Now $\mathcal{V}$ need not be regular; however it is ``almost smooth'' in the terminology of \cite[Def.~2.1]{Lou18}. Explicitly, this means that for a flat discrete valuation ring $R \subseteq R'$ of ramification degree $1$, any $R'$-point of $\mathcal{V}$ lies in the smooth locus of $\mathcal{V}$. By \cite[Lem.~3.11]{LS16} we may calculate $\delta_D(f)$ on an almost smooth model. We thus use  $\mathcal{V}$ and find that
$$\delta_D(f) = \frac{\# \left\{ \gamma \in \Gamma : 
	\begin{array}{l}
		\gamma \text{ fixes an irreducible component} \\
		\text{of } \Spec(\prod_b k_b)
	\end{array}
	\right \}}
	{\# \Gamma } = \frac{1}{\# \Gamma} = \frac{1}{|\res_D(\langle\br\rangle)|}$$
where $\Gamma = \Gal(K/k)$, as required for \eqref{eqn:res=delta}. For \eqref{eqn:Sub=Sub}, it suffices to show that
\begin{equation} \label{eqn:Sub=Sub2}
	\Br Y = \{ f^*\alpha : \alpha \in \Br k(X), \res_D(\alpha) \in \res_D(\langle\br\rangle) \}.
\end{equation}
Our proof is based upon the proof of \cite[Thm.~2.11]{Lou18}. First $\Br Y \subset f^* \Br k(X)$ by a theorem of Amitsur \cite[Thm.~5.4.1]{GS06}. We show the right side of \eqref{eqn:Sub=Sub2} is included in the left. Let $\alpha \in \Br k(X)$ with $\res_D(\alpha) \in \res_D(\langle\br\rangle)$. Given this we may write $\res_D(\alpha) = \sum_{b \in \br} r_b \res_D(b)$ for some $r_b \in \Z$. Then the element $\alpha - \sum_{b \in \br} r_b b$ is unramified on $X$ hence unramified when pulled-back to $Y$. However $f^*b = 0 \in \Br Y$ for all $b \in \br$ by Amitsur \cite[Thm.~5.4.1]{GS06}, thus $f^*\alpha \in \Br Y$ as required. Now let $\alpha \in \Br k(X)$ with $\res_D(\alpha) \notin \res_D(\langle\br\rangle)$. We pull-back $\alpha$ to the regular locus of $\mathcal{V}$ and use functoriality of residues \cite[Prop.~1.4.7]{Brauer}, recalling the above  description of the irreducible components of the special fibre, which all have multiplicity one. This shows that $\alpha$ is ramified on the regular locus of $\mathcal{V}$, hence clearly ramified when pulled-back to $Y$. This shows \eqref{eqn:Sub=Sub2} and completes the proof.
\end{proof}

Formula \eqref{eqn:Sub=Sub} from Proposition \ref{prop:calculate_br_Sub} gives a more useful way than Lemma \ref{lem:Sub_formula} to calculate the subordinate Brauer group for families of Brauer--Severi varieties (e.g.~conic bundles) and their products, since it avoids the need to construct a regular proper model, which can be problematic  (see Lemma \ref{lem:counter-example}).

\subsection{An example}

We give an example where $\BrV(Y/X) \neq f^*\Sub(X,f)$. This concerns families of conics over a surface, and shows that the subordinate Brauer group is a useful geometric invariant for proving non-existence of flat models. It also shows that, even when studying the basic problem of counting the number of conics in a family with a rational point, one \textit{cannot} assume that the family is flat, i.e.~one has to allow fibres of dimension $>1$. This example first appeared in \cite[Ex.~2.9]{Lou18}.

\begin{lemma} \label{lem:counter-example}
	Let $S = \P^1 \times \P^1$ over $\Q$ and let $Y$ be a smooth projective variety
	over $\Q$ with a dominant morphism  $f:Y \to S$  whose generic fibre is isomorphic
	to the conic $t_1^2 + t_2^2 = u_1u_2t_0^2$ over $k(S) = k(u_1,u_2)$. Then
	\begin{enumerate}
		\item $\Sub(S,f)/ \Br \Q \cong (\Z/2\Z)^2$, generated by the quaternion algebras
		$(u_1,-1)$ and $(u_2,-1)$. Thus $f^*\Sub(S,f) / \Br \Q \cong \Z/2\Z$ generated by $(u_1,-1)$.
		\item $\BrV(Y/S) = \Br \Q$, hence $\BrV(Y/S) \neq f^*\Sub(S,f)$
		\item There is no smooth proper variety $Y$ over $\Q$ equipped with a flat morphism
		$\psi: Y \to S$ whose generic fibre is isomorphic to $t_1^2 + t_2^2 = u_1u_2t_0^2$.
	\end{enumerate}
\end{lemma}
\begin{proof}
(1) We will use the formula for the subordinate Brauer group given in Definition \ref{def:br_Sub} (see Proposition \ref{prop:calculate_br_Sub}), since this is much easier to use than Lemma \ref{lem:Sub_formula} which requires an explicit smooth proper model.
We use coordinates $(x_1,y_1) \times (x_2,y_2)$ on $S$, so that $u_i = x_i/y_i$.  Let $\beta = (u_1u_2,-1)$. Then $\beta$ is ramified along the $4$ lines 
$$L_1: x_1= 0, \quad L_2: x_2 = 0, \quad L_3: y_1 = 0, \quad L_4: y_2 = 0$$
each with residue $-1 \in \Q(L_i)^\times/\Q(L_i)^{\times 2}$. Let $\beta_1 = (u_1,-1)$, which is ramified along $L_1$ and $L_3$ with residue $-1$. Let $b \in \Sub(S,f)$. By Proposition \ref{prop:calculate_br_Sub}, we know that $b$ must be ramified along some subset of the $L_i$ with residue $-1$. If $b$ is unramified then $b \in \Br S = \Br \Q$ is constant. If $b$ is ramified along only one of the $L_i$ then $b \in \Br \A^2 = \Br \Q$ is constant; a contradiction. Similarly if $b$ is ramified along at least $3$ of the $L_i$, then $b - \beta$ is constant. So consider the case where $b$ is ramified along exactly $2$ of the $L_i$. Using $S \setminus \{x_1 = 0, x_2 = 0\} \cong S \setminus \{x_1 = 0, y_2 = 0\} \cong S \setminus \{x_2 = 0, y_1 = 0\} \cong S \setminus \{y_1 = 0, y_2 = 0\}\cong  \A^2$ we see that $b$ must be ramified along either $L_1 \cup L_3$ or $L_2 \cup L_4$; this shows that $b$ must differ from one of $\beta_1, \beta$, or $\beta_1 + \beta$ by a constant, as required. When pulling-back to $Y$, the kernel is generated by $\beta$.

(2) Blow-up $S$ in the  intersection of $L_1$ and $L_2$. Let $E$ be the resulting exceptional divisor. Then $u_1$ and $u_2$ both have valuation $1$ along $E$. It follows that $\beta$ is unramified along $E$, but $\beta_1$ is ramified along $E$, thus $f^*\beta_1 \notin \BrV(Y/S)$. The result now follows from (1) and \eqref{lem:Vert}.

(3) Combine (1), (2), and Lemma \ref{lem:flat}.
\end{proof}

\section{The conjecture} \label{sec:conj}

\subsection{Set-up} \label{sec:set-up}
We first provide the setting for the conjecture and introduce some notation.
\begin{definition}
	A smooth projective geometrically integral variety $X$ over a field $k$ is called \emph{weak Fano} if $-K_X$ is nef and big.
\end{definition}

If $k$ has characteristic $0$ any weak Fano variety is rationally connected \cite[Cor. 1.13]{HM07}; in particular $\HH^i(X, \OO_X) = 0$ for all $i> 0 $ and $\Pic X$ is finitely generated torsion free.

Let $X$ be a weak Fano variety over a number field $k$. We denote by $\rho(X) = \rank \Pic X$. Choose an adelic metric $(\| \cdot \|_v)_{v \in \Val(k)}$ on the anticanonical bundle of $X$ and denote the associated height function by  $H$. We let $Y$ be a smooth projective variety and $f: Y \to X$ a dominant morphism with geometrically integral  generic fibre.

\begin{definition}
	For each point
	$x \in X$, we choose some finite group $\Gamma_x$ through which 
	the absolute Galois group $\Gal(\overline{\kappa(x)}/ \kappa(x))$ 
	acts on the irreducible 	components of $f^{-1}(x)_{\overline{\kappa(x)}}:=f^{-1}(x) \otimes_{\kappa(x)} \overline{\kappa(x)}$.
	We define
	$$\delta_x(f) = \frac{\# \left\{ \gamma \in \Gamma_x : 
	\begin{array}{l}
		\gamma \text{ fixes an irreducible component} \\
		\text{of $f^{-1}(x)_{\overline{\kappa(x)}}$ of multiplicity } 1
	\end{array}
	\right \}}
	{\# \Gamma_x }.$$
	Let $X^{(1)}$ denote the set of codimension $1$ points of $X$. Then
	we let
	$$\Delta(f) = \sum_{D \in X^{(1)}} ( 1 - \delta_D(f)).$$
\end{definition}

\begin{definition}
A subset $\Omega \subseteq X(k)$ is called \emph{thin}  if it is
a finite union of subsets which are either contained in a proper closed subvariety of $X$, or contained 
in some $\pi(Y(k))$ where $\pi: Y \to X$ is a generically finite dominant morphism 
of degree exceeding $1$, with $Y$ an integral variety over $k$.
\end{definition}

Our conjecture concerns  the counting function
$$N_{\textrm{loc}}(f,B) = \#\{ x \in X(k) : H(x) \leq B, x \in f(Y(\Adele_k))\}.$$
(Note in the conjecture it may be necessary to remove some thin subset from $X(k)$ to obtain the correct asymptotic formula.)

\begin{remark}
	Previous versions of Manin's conjecture only assumed that $-K_X$ be big, rather than both nef and big.
	However various pathological examples with $-K_X$ big have been found \cite[\S 5.1]{LST22}. It is not
	currently clear what the correct geometric assumptions on $X$ should be to cover all cases of interest
	(for example toric varieties are log Fano, but not necessarily weak Fano).
\end{remark}

\subsection{Peyre's constant}
Our approach is greatly inspired by that of Peyre \cite{Pey95}, who formulated a conjectural expression for the leading constant
in the classical case of Manin's conjecture. Peyre's constant takes the shape
\begin{equation} \label{eqn:Peyre}
	\alpha(X)\beta(X)\tau(X(\Adele_k)^{\Br}).
\end{equation}
Here $\alpha(X)$ is a certain rational number defined in terms of the cone
of effective divisors of $X$, $\beta(X) = \HH^1(k, \Pic \bar{X})$, and $\tau$ is Peyre's Tamagawa measure (the factor $\beta(X)$ first appeared in \cite{BT95} and was given a geometric interpretation in \cite{Sal98}).

\subsection{Effective cone constant}

In \cite[Def.~2.4]{Pey95} Peyre introduced his effective cone constant, denoted by $\alpha(X)$. We require a renormalisation of Peyre's constant; this is because $\alpha(X)$ contains the factor $(\rho(X) - 1)!$ on the denominator, which we will interpret as a special value of the Gamma function and treat separately. This renormalisation already appeared in \cite{BT95}, and we take the definition from there, albeit with different notation.

Let $\Eff(X) \subset \Pic(X)_\R := \Pic(X) \otimes \R$ denote the closure of the cone of effective divisors of $X$. Denote by $\Eff(X)^\wedge$ the dual cone and by $\Pic(X)^\wedge \subset \Pic(X)_\R^\wedge$ the dual lattice. Equip $\Pic(X)_\R^\wedge$ with the Haar measure so that $\Pic(X)^\wedge$ has covolume $1$. We then define
\[
	\alpha^*(X) = \int_{ v \in \Eff(X)^\wedge} e^{-\langle -K_X, v\rangle} \mathrm{d}v.
\]
By \cite[Rem.~2.4.8]{BT95}, we have
\begin{equation} \label{eqn:factorial}
	\alpha^*(X) = \alpha(X)  / (\rho(X) - 1)! \,\, .
\end{equation}
One advantage of this definition is that it is compatible with products, i.e.~$\alpha^*(X_1 \times X_2) = \alpha^*(X_1) \alpha^*(X_2),$
as follows from \eqref{eqn:factorial} and \cite[Lem.~4.2]{Pey95}.

\subsection{A modified Fujita invariant}
Let $\Eff(X)$ denote the pseudo-effective cone of $X$, i.e.~the closure of the cone of effective divisors in $\Pic(X) \otimes \R$.
For a $\Q$-divisor $D$ on $X$, recall that the \textit{Fujita invarian}t of $D$ is defined to be
$$a(D) = \inf\{t > 0 : K_X + tD \in \Eff(X) \}.$$
This occurs in the Batyrev--Manin conjecture as the exponent of $B$ when counting with respect to height functions associated to $D$ (see \cite[Conj.~3.2]{BM90}).

We require a variant of this. For a $\Q$-divisor $D$ on $X$ we define
\begin{equation} \label{def:modified_Fujita}
	\eta(D) =  \sup\{t > 0 : -K_X - tD \in \Eff(X) \}.
\end{equation}
If $\rho(X) = 1$ then $a(D) = \eta(D)$. However, despite the two definitions being superficially similar, there does not seem to be a simple relationship between $a(D)$ and $\eta(D)$ in general. For example $\eta(D)> 0$ for any effective divisor, but $a(D) > 0$ only if $D$ is big. We require the following formula for $\eta(D)$.

\begin{lemma} \label{lem:simplicial}
	Let $X$ be a weak Fano variety. Assume that $\Eff(X)$
	is simplicial with generators $D_1,\dots,D_\rho$. Write 
	$-K_X= a_1 D_1 + \dots + a_\rho D_\rho$ with $a_i \in \R.$ Then $\eta(D_i) = a_i$.
\end{lemma}
\begin{proof}
	First note that as $-K_X$ is effective and $\Eff(X)$ is simplicial we must have $a_i \geq 0$.
	Next, without loss of generality we have $i = 1$. Then 
	$$-K_X - tD = (a_1 - t)D_1 + a_2 D_2 + \dots +a_\rho D_\rho.$$
	However as $\Eff(X)$ is simplicial $ -\varepsilon D_1 + a_2 D_2 + \dots + a_\rho D_\rho$
	is not effective for any $\varepsilon > 0$. The result follows.
\end{proof}

For example, recalling that $\omega_{\P^n}^{-1} = \O_{\P^n}(n + 1)$, we have $\eta(\O_{\P^n}(d)) = (n+1)/d$.

\subsection{Virtual Artin $L$-functions} 
We  shall use the formalism of \textit{virtual Artin $L$-functions}.
A virtual Artin representation over $k$ is a formal finite sum
$V = \sum_{i=1}^n z_iV_i$
where $z_i \in \CC$ and the $V_i$ are Artin representations of $G_k = \Gal(\bar{k}/k)$.
We define $\dim V = \sum_{i=1}^n z_i \dim V_i$
and let $V^{G_k}=\sum_{i=1}^n z_i V_i^{G_k}$.
The $L$-function of $V$ is defined to be
$$L(V,s)=\prod_{i=1}^n L(V_i,s)^{z_i},$$
where $L(V_i,s)$ is the usual Artin $L$-function associated to $V_i$. 
Standard properties of Artin $L$-functions imply that $L(V,s)$ admits a holomorphic
continuation with no zeros to the region $\re s \geq 1$, apart from possibly at $s=1$.
We have
$$L(V,s) = \frac{c_V}{(s-1)^{r}} + O\left(\frac{1}{(s-1)^{r-1}}\right),$$
as $s \to 1$, where $r = \dim V^{G_F}$ and $c_V \neq 0$.
In this notation we shall write
$$L^*(V,1) = c_V.$$
Denote by $L_v(V,s)$ the corresponding local Euler factor at a non-archimedean place $v$.

\subsection{Tamagawa measures}
We now define our Tamagawa measure. The first step is the same as in Peyre's paper \cite[\S 2.2]{Pey95}, but the key difference is that we require different convergence factors. We choose Haar measures on $k_v$ such that $\O_v$ has measure $1$ for almost all $v$ and such that the induced adelic measure satisfies $\vol(\Adele_k/k) = 1$.
We let $\tau_v$ denote Peyre's local Tamagawa measure associated to our choice of adelic metric $\| \cdot \|_v$ and Haar measure on $k_v$. To define the convergence factors, consider the following virtual Artin representation
\begin{equation}\label{def:Pic_br}
	\Pic_f(\overline{X})_\CC=\Pic(\overline{X})_\CC
	- \sum_{D \in X^{(1)}}\left(1-\delta_D(f) \right)\Ind_{k_D}^k \CC.
\end{equation}
Here $\Pic(\overline{X})_\CC = \Pic(\overline{X}) \otimes_\ZZ \CC$, $\Ind$ denotes the induced
Galois representation, and $k_D$ denotes the algebraic closure of $k$ in $\kappa(D)$. 
Next let $\Pic_f(X)_\CC= \Pic_f(\overline{X})^{G_F}_\CC$.
The corresponding virtual Artin $L$-function is
\begin{equation} \label{eqn:L-function}
	L(\Pic_f(\overline{X})_\CC,s)= 
	\frac{L(\Pic(\overline{X})_{\CC},s)}{\prod_{D \in X^{(1)}} \zeta_{k_D}(s)^{1-\delta_D(f)}}
\end{equation}
For each place $v \in \Val(k)$ we define
$$ \lambda_v = \left \{
	\begin{array}{ll}
		L_v(\Pic_f(\overline{X})_\CC,1),& \quad v \text{ non-archimedean}, \\
		1,& \quad v \text{ archimedean}. \\
	\end{array}\right.$$	
Our Tamagawa measure is now defined to be
\begin{equation} \label{def:Tamagawa}
	\tau_f = L^*(\Pic_f(\overline{X})_\CC,1)\prod_v \lambda_v^{-1} \tau_v .
\end{equation}
We have not included a discriminant factor as in Peyre \cite[Def.~2.1]{Pey95}, since we have normalised our Haar measures so that $\vol(\Adele_k/k) = 1$. 
We now show that these $\lambda_v$ are indeed a family of convergence factors in our case.

\begin{theorem} \label{thm:Tamagawa_product}
	The infinite product measure $\prod_v \lambda_v^{-1} \tau_v $
	converges on $f(Y(\Adele_k)) \subseteq X(\Adele_k)$. 
\end{theorem}
\begin{proof}
	It suffices to show that the product
	$\prod_v \tau_v(f(Y(k_v)))\lambda_v^{-1}$
	converges. Recalling \eqref{eqn:L-function}, we can rewrite this as
	$$\prod_v \frac{\tau_v(X(k_v))}{L_v(\Pic(\overline{X})_{\CC},1)}
	 \prod_v \left(1 - \frac{\tau_v(X(k_v) \setminus f(Y(k_v))}{\tau_v(X(k_v))} \right)
	 \prod_{D \in X^{(1)}} \zeta_{k_D,v}(1)^{1-\delta_D(f)}.$$	
	The first Euler product is convergent by a result of Peyre \cite[Prop.~2.2.2]{Pey95}.
	So it suffices to consider the second Euler product.
	To do so, we choose a model $f: \mathcal{Y} \to \mathcal{X}$ 
	for the morphism $f$ over the ring of integers $\O_k$ of $k$. 
	We choose a sufficiently large set of primes $S$ of $k$ and let
	$\fp \notin S$. We first note that 
	\begin{equation} \label{eqn:non-split_1}
		X(k_\fp) \setminus f(Y(k_\fp)) \subseteq \{ x \in \mathcal{X}(k_\fp) : 
		f^{-1}(x) \bmod \fp \text{ is non-split} \}.
	\end{equation}
	Indeed if $f^{-1}(x) \bmod \fp$ were split then $f^{-1}(x) \bmod \fp$ would have 
	a smooth $\F_\fp$-point by the Lang--Weil estimates, providing $S$ is sufficiently
	large, which would give rise to a $k_\fp$-point by Hensel's lemma. We now show
	a partial reverse inclusion using \cite[Thm~2.8]{LS16}.
	This says the following: Let $T \subset X$ be a reduced divisor which contains
	the non-smooth locus of $f$ and $\mathcal{T}$ its closure in $\mathcal{X}$. Then
	there exists a closed subset $\mathcal{Z} \subset \mathcal{T}$ of codimension $2$ in 
	$\mathcal{X}$ which contains	the singular locus of $\mathcal{T}$ such that
	\begin{equation} \label{eqn:non-split_2}
		\left\{ x \in \mathcal{X}(k_\fp) : 
		\begin{array}{l}
			f^{-1}(x) \bmod \fp \text{ is non-split},  \\
			x \bmod \fp^2  \text{ meets } \mathcal{T} \text{ transversely 
			outside of } \mathcal{Z}
		\end{array} \right\} \subseteq X(k_\fp) \setminus f(Y(k_\fp)).
	\end{equation}
	Combining \eqref{eqn:non-split_1} and \eqref{eqn:non-split_2}
	with  formulae for Tamagawa measures \cite[Lem.~3.2]{BL19} yields
	\begin{align*}
	& \frac{1}{\Norm \fp^{2n}}\#\left\{ x \in \mathcal{X}(\O_k / \fp^2) : 
		\begin{array}{l}
			f^{-1}(x) \bmod \fp \text{ is non-split},  \\
			x \bmod \fp^2  \text{ meets } \mathcal{T} \text{ transversely 
			outside of } \mathcal{Z} 
		\end{array} \right\}  \\
	& \leq \tau_v(X(k_v) \setminus f(Y(k_v)))
	\leq \frac{\#\{ x \in \mathcal{X}(\F_\fp) : f^{-1}(x) \text{ is non-split} \}}{\Norm \fp^n},
	\end{align*}
	where $n = \dim X$. However by \cite[Prop.~2.3]{BL19} and the Lang-Weil
	estimates we have
	\begin{align*}
	& \#\left\{ x \in \mathcal{X}(\O_k / \fp^2) : 
		\begin{array}{l}
			f^{-1}(x) \bmod \fp \text{ is non-split},  \\
			x \bmod \fp^2  \text{ meets } \mathcal{T} \text{ transversely 
			outside of } \mathcal{Z} 
		\end{array} \right\} \\
	& = (\Norm \fp^n + O(\Norm \fp^{n-1})\#\{ x \in \mathcal{X}(\F_\fp) : f^{-1}(x) \text{ is non-split},
	x \in \mathcal{T} \setminus \mathcal{Z}\} \\
	& =  \Norm \fp^n \#\{ x \in \mathcal{X}(\F_\fp) : f^{-1}(x) \text{ is non-split} \}
	+ O(\Norm \fp^{2n - 2}).
	\end{align*}
	Altogether this shows that 
	$$\tau_v(X(k_v) \setminus f(Y(k_v)))
	=  \frac{\#\{ x \in \mathcal{X}(\F_\fp) : f^{-1}(x) \text{ is non-split} \}}{\Norm \fp^n}
	+ O\left(\frac{1}{\Norm \fp^{n-2}}\right).$$
	Using $\tau_\fp(X(k_\fp)) = \#\mathcal{X}(\F_\fp)/\Norm \fp^n = 1 + O(1/\Norm \fp)$, 
	we have reduced	to proving that
	\begin{equation} \label{eqn:Euler_split}
		\prod_{\fp \notin S} 
		\left(1 - \frac{\#\{ x \in \mathcal{X}(\F_\fp) : f^{-1}(x) \text{ is non-split} \}}{\Norm \fp^n} \right)
	 	\prod_{D \in X^{(1)}} \zeta_{k_D,\fp}(1)^{1-\delta_D(f)}
	\end{equation}
	converges. To do so we use that if $\sum_{k =1}^\infty a_k$ converges then
	$\prod_{k = 1}^\infty (1+a_k)$ converges. 
	For a finite field extension $k \subset K$,	write
	$$\zeta_{K,\fp}(1) = 1 + a_{K,\fp}/ (\Norm \fp) + O(1/ \Norm \fp^2).$$
	Then the prime ideal theorem implies that 
	$$\sum_{\Norm \fp \leq X} \frac{a_{K,\fp}}{\Norm \fp} = \log \log X + O(1)$$
	(Mertens for number fields).
	We conclude that 
	$$\sum_{D \in X^{(1)}} \sum_{\Norm \fp \leq X} (1 - \zeta_{k_D,\fp}(1)^{1-\delta_D(f)})
	= \sum_{D \in X^{(1)}} (1 - \delta_D(f))\log \log X + O(1) =
	\Delta(f) \log \log X + O(1).$$
	However it follows from \cite[Prop.~3.10]{LS16} that
	$$\sum_{\Norm \fp \leq X}
	\#\{ x \in \mathcal{X}(\F_\fp) : f^{-1}(x) \text{ is non-split} \}
	= \Delta(f) \sum_{\Norm \fp \leq X} \Norm \fp^{n-1} + O(X^n/ (\log X)^2).$$
	(The result \emph{loc.~cit.} is stated without an error term, but a minor modification
	of the argument gives the stated error term on using Serre's version of the 
	Chebotarev density theorem \cite[Thm.~9.11]{Ser12}).
	An easy application of partial summation and the prime ideal theorem then shows 
	that
	$$\sum_{\Norm \fp \leq X}
	\frac{\#\{ x \in \mathcal{X}(\F_\fp) : f^{-1}(x) \text{ is non-split} \}}{\Norm \fp^n}
	= \Delta(f) \log \log X + O(1).$$
	Thus the leading terms in the asymptotics cancel each other out, which shows
	convergence of \eqref{eqn:Euler_split}, as required.	
\end{proof}

\begin{remark}
	The product in Theorem \ref{thm:Tamagawa_product} may be only
	conditionally convergent: for an example, see Proposition \ref{prop:easy}. 
	This is in contrast to Peyre's Tamagawa measure in Manin's conjecture, 
	which is always absolutely convergent.	
\end{remark}

\subsection{The conjecture}
We take the set-up of  \S \ref{sec:set-up}. Let $U \subseteq X$ be the open subset of $X$ given by removing the closure of those codimension $1$ points of $X$ which lie below a non-split fibre. We take $V= f^{-1}(U)$.

\begin{conjecture} \label{conj:main}
	Assume that $\Br X = \Br_1 X$ and that $f$ admits a smooth fibre over a rational
	point which is everywhere locally soluble. 
	Assume that the fibre over every codimension $1$ point of $X$
	contains an irreducible component of multiplicity $1$. 
	Assume that either $\rho(X) = 1$ or $k[U]^\times =k^\times$.
	Then there exists a thin subset $\Omega \subset X(k)$ such that
	$$\#\{ x \in X(k) \cap f(Y(\Adele_k)) : H(x) \leq B, x \notin \Omega \} 
	\sim c_{f,H}B(\log B)^{\rho(X)  - \Delta(f) - 1}$$
	where
	\begin{align*}
		c_{f,H} &= \frac{\alpha^*(X) \cdot |\Sub(X,f)/\Br k| \cdot 
		\tau_f((\prod_v f(V(k_v)))^{\Sub(X,f)})}
	{\Gamma(X,f)} \cdot  
	\prod_{D \in X^{(1)}}\eta(D)^{1- \delta_D(f)}, \\
	\Gamma(X,f) &= 
	\begin{cases}
		\prod_{D \in X^{(1)}}\Gamma(\delta_D(f)) , & \text{if } \rho(X) = 1, \\
		\Gamma( \rho(X) - \Delta(f)), & \text{if }k[U]^\times = k^\times.
	\end{cases}
	\end{align*}
\end{conjecture}

Note that  $\Sub(X,f)/\Br k$ is finite by Lemma \ref{lem:codim_1_finite} and the finiteness of $\Br X/\Br k$. The key case in the paper is $X = \P^n$, which satisfies all the above hypotheses.

\begin{remark}
	When $f:X \to X$ is the identity, one recovers Peyre's constant \eqref{eqn:Peyre}
	on using \eqref{eqn:factorial} and 
	$|\Sub(X,f)/\Br k| = |\Br X/ \Br k| = \HH^1(k, \Pic \bar{X})$  since $\Br X = \Br_1 X$.
\end{remark}

\begin{remark}
	When there is a product of Gamma factors it seems highly doubtful that standard Tauberian techniques (for example, Delange's Tauberian theorem \cite{Del54}) can be used, since these only give a single Gamma factor. We suspect that some clever use of multiple Dirichlet series will be required if trying to use height zeta function methods.
\end{remark}

\begin{remark} \label{rem:birational_invariant}
	Conjecture \ref{conj:main} only depends on $X$ and the generic fibre of $f$,
	i.e.~is a birational invariant
	(providing $Y$ is smooth and projective). The birational
	invariance of $\Sub(X,f)$ is Lemma \ref{lem:Sub_birational}, and the birational
	invariance of $\delta_D(f)$ is \cite[Lem.~3.11]{LS16}.
\end{remark}

\begin{remark}
	One could interpret the factors $\eta(D)^{1- \delta_D(f)}$ measure-theoretically
	and  combined with the effective cone constant $\alpha^*(X)$, as the volume of a certain
	integral on the virtual Picard group $\Pic_f(X)_\R$ (see \eqref{def:Pic_br}).
\end{remark}


\subsubsection{The thin set $\Omega$} \label{sec:thin_set}
In the case where $X = \P^n$, which is the primary case of interest in the paper, no thin set $\Omega$ is required. Indeed here a theorem of Cohen \cite[Thm.~13.3]{Ser97M} implies that any thin set contains $O_\varepsilon(B^{1 - 1/2(n+1) + \varepsilon})$ points of anticanonical height at most $B$ for any $\varepsilon > 0$, which is negligible for the conjecture.

For more general $X$, a thin set is required for Manin's conjecture hence certainly for us (even if $\rho(X) = 1$ there can be Zariski dense accumulating thin sets in Manin's conjecture; for example lines in intersections of two quadrics \cite[Ex.~3.6]{BL17}). In a series of works \cite{LT17,LST22} tools from birational geometry were used to study the exceptional set in Manin's conjecture and obtain an explicit description. It would be interesting to see whether such tools can be used to predict what shape $\Omega$ should take in our case. (This would require, for example, a version of Conjecture \ref{conj:main} for arbitrary big and nef line bundles, rather than just the anticanonical bundle.)

Let us emphasise that there is indeed new behaviour which needs to be controlled. In a surprising work, Browning--Lyczak--Sarapin \cite{arXiv:2203.06881} considered a fibration over a quadric surface $X$. Despite no thin set being required here for Manin's conjecture, there is a thin set which dominates $N_{\textrm{loc}}(f,B)$. This demonstrates that the thin set $\Omega$  depends on both $f$ and $X$ in general, rather than just $X$. This fibration over a quadric is not 
 covered by our conjecture as it fails the assumption \eqref{eqn:assump_geometric}; in a tour-de-force
 Wilson~\cite{wilson1} recently proved       an asymptotic of the shape $c B(\log \log B)/(\log B)$  for some $c>0$ 
when this thin set is removed.

\subsubsection{Volume of the Brauer--Manin set} \label{sec:volume_BM}
We discuss in detail the Tamagawa measure $\tau_f((\prod_v f(V(k_v)))^{\Sub(X,f)})$ which appears, as the expression is somewhat subtle (similar phenomena occur in the case of Campana points, see \cite[\S8]{wonderful}). We have $\Sub(X,F) \subset \Br U$. The Brauer--Manin pairing for $\Br U$ is well-defined on $U(\Adele_k)$. However it is not difficult to see that $\tau_f(f(V(\Adele_k))) = 0$; this is because the convergence factors force the infinite product measure to diverge to $0$. The Brauer--Manin pairing is not well-defined on $\prod_v U(k_v)$ in general, since the local invariant can be non-constant at infinitely many places. Thus the notation $(\prod_v f(V(k_v)))^{\Sub(X,f)}$ is not well-defined.

We interpret this measure as follows. For any finite set of places $S$ of $k$ we consider the pairing
$$\prod_{v \in S} U(k_v) \times \Br U \to \Q/\Z, \quad ((u_v),b) \mapsto \sum_{v \in S} \inv_vb(u_v).$$
For a subset $W \subset \prod_{v \in S} U(k_v)$ and subset $\br \subset \Br U$ we denote by $W^{\br}$ the orthogonal complement to $\br$ on $W$. Then the term which appears in the leading constant is defined via the following abuse of notation, which serves as a convenient shorthand:
\begin{equation} \label{def:tau_f_Sub}
\tau_f((\prod_v f(V(k_v)))^{\Sub(X,f)}):=L^*(\Pic_f(\overline{X})_\CC,1)\lim_S \prod_{v \in S} \lambda_v^{-1} \tau_v
		((\prod_{v \in S}f(V(k_v)))^{\Sub(X,f)})
\end{equation}
where the limit is over all finite sets of places $S$ of $k$.
This is quite complicated to calculate directly. However, we give a simpler formula in terms of a finite sum of Euler products, which moreover shows that the limit over $S$ exists. We use the map $\Q/\Z \to \C^\times, x \mapsto e^{2 \pi i x}$.

\begin{lemma} \label{lem:sum_Euler_products}
	For each $b \in \Sub(X,f)$ the Euler product
	\begin{align*}
	\hat{\tau}_f(b) &:= \L^*(\Pic_f(\overline{X})_\CC,1)
	\prod_v \lambda_v^{-1}\int_{f(V(k_v))} e^{2 \pi i \inv_v b(x_v) } d\tau_{v}(x_v)
	\end{align*}
	exists. 
	Let $\br$ be a finite group of representatives	of $\Sub(X,f)/ \Br k$.
	Then
	$$
		|\Sub(X,f)/\Br k| \cdot \tau_f((\prod_v f(V(k_v)))^{\Sub(X,f)}) = 
		\sum_{b \in \br} \hat{\tau}_f(b).
	$$
	In particular $\tau_f((\prod_v f(V(k_v)))^{\Sub(X,f)})$ exists.
\end{lemma}
\begin{proof}
	The existence of $\hat{\tau}_f(b)$ follows from the bound $|e^{2 \pi i \inv_v b(x_v) }| \leq 1$
	and Theorem \ref{thm:Tamagawa_product}.
	For the second part, the map
	$$\br \to \C^\times, \quad b \mapsto 		e^{2 \pi i \left(\sum_{v \in S}\inv_v b(x_v)\right) }$$
	is a character. Hence character orthogonality implies that
	$$|\br|\prod_{v\in S}\tau_v
		((\prod_{v \in S}f(V(k_v)))^{\Sub(X,f)})
		=\int_{ \prod_{v \in S} f(V(k_v))}
		\sum_{b \in \br}	e^{2 \pi i \left(\sum_{v \in S}\inv_v b(x_v)\right) } \prod_{v \in S} d\tau_{v}(x_v).$$
	Changing the order of summation, applying Fubini's theorem, and taking the limit over $S$
	therefore gives the result.
\end{proof}

The Euler factors of $\hat{\tau}_f(b)$ can be different from those of $\hat{\tau}_f(0)$ at infinitely many places. However in the following special case they differ at only finitely many places.

\begin{lemma} \label{lem:calculate_Tamagawa}
	Let $b \in \Sub(X,f)$ and assume that $f^*b \in \Br Y$.
	Let $S$ be a finite set of places such that $f^*b$ evaluates trivially on $Y(k_v)$
	for $v \notin S$.
	Then for all $v \notin S$ we have
	$$\int_{f(V(k_v))} e^{2 \pi i \inv_v b(x_v) } d\tau_{v}(x_v) = \tau_v(f(Y(k_v))).$$
\end{lemma}
\begin{proof}
	The existence of $S$ follows from $f^*b \in \Br Y$.
	The result then follows from the fact that the boundary $f(Y(k_v)) \setminus f(V(k_v))$ 
	has measure zero since it is supported
	on a proper closed subscheme.
\end{proof}

Being defined by a complicated limit, and also a sum of finitely many complex-valued Euler products, it is not immediate that our measure is non-zero. We verify this now.

\begin{proposition}
	$\tau_f((\prod_v f(V(k_v)))^{\Sub(X,f)}) \neq 0$ in the setting of Conjecture \ref{conj:main}.
\end{proposition}
\begin{proof}
	We show that the measure is non-zero by constructing an explicit subset 
	upon which the Tamagawa measure is given by a convergent non-zero product.
	Let $u \in U(k)$ be a rational point below a smooth everywhere locally
	soluble fibre. Let $\br$ be a group of representatives
	for the elements of $\Sub(X,f)/\Br k$ which evaluate trivially at $u$. 
	Let $S_0$ be a sufficiently large finite set of places of $k$. By continuity
	of the local Brauer--Manin pairing and the implicit function theorem,
	there is an open subset
	$u \in W_v \subset f(V(k_v))$ such that every $b \in \br$ evaluates trivially
	on $W_v$ for all $v \in S_0$. 
	
	For $v \notin S_0$ we use 
	some of the techniques from the proof of Theorem \ref{thm:Tamagawa_product}.	
	Choose a model $f: \mathcal{Y} \to \mathcal{X}$ 
	for the morphism $f$ over the ring of integers $\O_k$ of $k$. 
	Let $T \subset X$ be a reduced divisor which contains
	the non-smooth locus of $f$ and $\mathcal{T}$ its closure in $\mathcal{X}$.
	Let $Z \subset T$ be a sufficiently large codimension $2$ subset of $X$ which 
	contains the non-flat
	locus of $f$ and the non-smooth locus of $\mathcal{T}$.
	Denote by $\mathcal{Z}$ its closure in $\mathcal{T}$.
	For $\fp \notin S_0$ we set
	$$
		W_\fp = \{ x \in U(k_\fp) : 
			f^{-1}(x) \bmod \fp \text{ is split and if }\x \bmod \fp \in 
			\mathcal{T} \text{ then } 
			x \bmod \fp \notin \mathcal{Z}\}.
	$$
	We first claim that	
	\begin{equation} \label{eqn:W_v}
		\prod_{v \in S} W_v \subset 	(\prod_{v \in S}f(V(k_v)))^{\br}.
	\end{equation}
	For $v \in S_0$ this is by construction. For $\fp \notin S_0$,
	let $x_\fp \in W_\fp$. Providing $S_0$ is sufficiently
	large, the Lang--Weil estimates and
	Hensel's Lemma imply that
	there is $y_\fp \in V(k_\fp)$ such that $f(y_\fp) = x_\fp$
	and such that if $y_\fp \bmod \fp \in f^{-1}(\mathcal{T})$,
	then $y_\fp \bmod \fp$ lies in a smooth point of
	an irreducible component of $f^{-1}(\mathcal{T}\setminus \mathcal{Z})$
	of multiplicity $1$. We have $b(x_\fp) = (f^*b)(y_\fp)$ for and all $b \in \br$. However 
	let $E \subset f^{-1}(T \setminus Z)$ be the union of the smooth
	loci of the irreducible components of multiplicity $1$ with corresponding
	$\mathcal{E}$. Then by Definition \ref{def:Sub}
	we have $f^*b \in \Br V \cup E$. It thus follows that $f^*b$  evaluates
	trivially on $(\mathcal{V} \cup \mathcal{E})(\O_\fp)$ for all but finitely many
	$\fp$ and all $b \in \br$.
	However by construction we have $y_\fp \in (\mathcal{V} \cup \mathcal{E})(\O_\fp)$,
	so the claim \eqref{eqn:W_v} follows after enlarging $S_0$.

	We next claim that for all $\fp \notin S_0$ we have
	\begin{equation} \label{eqn:W_v_complement}
		\tau_\fp(W_\fp) = \tau_\fp(f(Y(k_\fp)) + O(1/\Norm \fp^2).
	\end{equation}
	To see this, we consider the measure of the complement $f(V(k_\fp)) \setminus W_\fp$,
	which by the arguments
	in the proof of Theorem \ref{thm:Tamagawa_product} is contained in
	$$
		\{ x \in  f(V(k_\fp)) : 
			x \bmod \fp \in \mathcal{Z} \text{ or } x \bmod \fp^2 \text{ meets
			} \mathcal{T} \text{ non-transversely}\}.
	$$
	(If $x \in f(V(k_\fp))$ and $f^{-1}(x) \bmod \fp$ is non-split, then 
	it must meet $\mathcal{T}$ non-transversely providing $x \bmod \fp \notin \mathcal{Z}$).
	However by \cite[Lem.~3.2]{BL19} and \cite[Prop.~2.3]{BL19}, this set has measure
 	$$\ll \frac{\#\mathcal{Z}(\F_\fp)}{\Norm \fp^{\dim X}} + 	
	\frac{\#\mathcal{T}(\F_\fp)}{\Norm \fp^{\dim X+1}} \ll \frac{1}{\Norm \fp^2}$$
	where the latter bound follows from the Lang--Weil estimates. This proves
	\eqref{eqn:W_v_complement}.
	
	As each $W_v$ clearly has positive measure, it now follows from 
	Theorem \ref{thm:Tamagawa_product} and \eqref{eqn:W_v_complement}
	that $\tau_f(\prod_v W_v) > 0$. In the light of \eqref{eqn:W_v} this completes the proof.
\end{proof}

\subsection{Specialisation of Brauer group elements}
Serre's original paper \cite{Ser90}, as well as the papers \cite{Lou18,LTBT20}, are written in terms of the language of specialisation of Brauer group elements. Working with Brauer group elements instead of a family of varieties can be advantageous as one does not have to worry about constructing models, which can be problematic (see Lemma \ref{lem:counter-example}). We explain  how to apply Conjecture \ref{conj:main} to this problem via the resulting family of products of Brauer--Severi varieties.

Let $X$ be as in \S \ref{sec:set-up}. Let $U \subseteq X$ be an open subset and $\br \subset \Br U$ a finite subset. We are interested in counting rational points in  the zero-locus 
$$X(k)_\br := \{x \in X(k) : b(x) = 0 \text{ for all } b \in \br \}.$$
(Here we take the convention that $b(x) \neq 0$ if $x$ lies in the ramification locus of $b$).
There is an analogous way to define a Tamagawa measure in this setting: In place of \eqref{def:Pic_br}, one takes the virtual Artin representation
$$
	\Pic_{\br}(\overline{X})_\CC=\Pic(\overline{X})_\CC
	- \sum_{D \in X^{(1)}}\left(1-1/|\res_D(\langle\br\rangle)| \right)\Ind_{k_D}^k \CC.
$$
The local factors $\lambda_v$ of the corresponding virtual $L$-function are then taken as the convergence factors for a Tamagawa measure $\tau_{\br} = L^*(\Pic_\br(\overline{X})_\CC,1)\prod_v \tau_v /\lambda_v$. This is well-defined on the adelic zero-locus
$$X(\Adele_k)_\br := \{(x_v) \in X(\Adele_k) : b(x_v) = 0 \text{ for all } b \in \br
\text{ and all } v  \}.$$
One can show this using an analogous argument to Theorem \ref{thm:Tamagawa_product}, or alternatively it also follows from Lemma \ref{lem:Lang-Nishimura} below.
Our conjecture is now as follows.

\begin{conjecture} \label{conj:Br}
	Assume that $\Br X = \Br_1 X$, that $U(k)_\br \neq \emptyset$,
	and that either $\rho(X) = 1$ or $k[U]^\times =k^\times$.
	Then there exists a thin subset $\Omega \subset X(k)$ such that
	$$\#\{ x \in X(k)_\br : H(x) \leq B, x \notin \Omega \} 
	\sim c_{\br,H} B(\log B)^{\rho(X)  - \Delta(\br) - 1}$$
	where
	\begin{align*}
	\Delta(\br) & = \sum_{D \in X^{(1)}} ( 1 - 1/|\res_D(\langle\br\rangle)|), \\
		c_{\br,H} &= \frac{\alpha^*(X) \cdot |\Sub(X,\br)/\Br k| \cdot  \tau_\br(X(\Adele_k)_\br^{\Sub(X,\br)})}
	{\Gamma(X,\br)} \cdot  
	\prod_{D \in X^{(1)}}\eta(D)^{1- 1/|\res_D(\langle\br\rangle)|}, \\
	\Gamma(X,\br) &= 
	\begin{cases}
		\prod_{D \in X^{(1)}}\Gamma(1/|\res_D(\langle\br\rangle)|) , & \text{if } \rho(X) = 1, \\
		\Gamma( \rho(X) - \Delta(\br)), & \text{if }k[U]^\times = k^\times.
	\end{cases}
	\end{align*}
\end{conjecture}

Here $\tau_\br(X(\Adele_k)_\br^{\Sub(X,\br)})$ is interpreted via a similar convention to \eqref{def:tau_f_Sub}. We relate this to Conjecture \ref{conj:main} using the work in \S \ref{sec:family_BS}.
We take $f:Y \to X$ to be a smooth proper model of the fibre product $\times_{b \in \br} V_b \to U$ over $U$ where $V_b$ denotes the Brauer--Severi scheme associated to $b$.

\begin{lemma} \label{lem:Lang-Nishimura}
	We have
	$$U(k)_\br = f(V(\Adele_k)), \quad U(\Adele_k)_\br = f(V(\Adele_k))$$
\end{lemma}
\begin{proof}
Combine Lang--Nishimura with the following: a Brauer--Severi variety
over a field has a rational point if and only if the associated Brauer group element is trivial, and an element of $\Br k$ is trivial if and only if it is trivial everywhere locally \eqref{seq:CFT}.
\end{proof}

By comparing conjectures and matching up relevant factors, we deduce the following from Proposition \ref{prop:calculate_br_Sub} and Lemma \ref{lem:Lang-Nishimura}.

\begin{corollary} \label{cor:sub}
	$Y \to X$ satisfies Conjecture \ref{conj:main}  if and only if $\br$
	satisfies Conjecture~\ref{conj:Br}.
\end{corollary}

\section{Verifying the conjecture} \label{sec:verify}

In this section we gather various known results from the literature and some other new results and show that they are compatible with Conjecture \ref{conj:main}. Our main new result (Theorem \ref{thm:main}) will be proved in later sections.

\subsection{Local densities for projective space}
To assist with later calculations, we give some formulae for calculating local densities when the base variety $X = \P^n_\Q$. For a place $v$ of $\Q$ we let $\tau_v$ denote Peyre's Tamagawa measure on $\P^n(\Q_v)$. For a subset $\Omega_v \subset \P^n(\Q_v)$ we denote by $\widehat{\Omega}_v \subset \Q_v^{n+1}$ its affine cone. We let $\mu_v$ be the usual Haar measure on $\Q_v^{n+1}$.

\begin{proposition}  \label{prop:local_densities}
	Let $v$ be a place of $\Q$ and let $\Omega_v \subset \P^n(\Q_v)$ be measurable.
	\begin{enumerate}
		\item $\tau_\infty(\Omega_\infty) = ((n+1)/2)\cdot \mu_\infty(\widehat{\Omega}_v \cap [-1,1]^{n+1}).$
		\item $\tau_p(\Omega_p) = (1 + 1/p + \dots + 1/p^n) \cdot \mu_p(\widehat{\Omega}_p \cap \Z_p^{n+1})$.	
		\item $\tau_p(\Omega_p) = (1-1/p)^{-1} \cdot \mu_p\{ \x \in \widehat{\Omega}_p \cap \Z_p^{n+1} : p \nmid \x\}.$
	\end{enumerate}
\end{proposition}
\begin{proof}
	For (1), let $\one: \R^{n+1} \to \{0,1\}$ denote the indicator function of $\widehat{\Omega}_\infty$. Then by definition we have
	\begin{align*}
	\tau_\infty(\Omega_\infty) & = \int_{\substack{\mathbf{u} \in \R^n  
	 }} \frac{\one(1,u_1,\dots,u_n) }{\max\{1, |u_1|, \dots, |u_n|\}^{n+1}} \mathrm{d} \mathbf{u} \\
	& = \frac{-1}{2}\int_{\substack{\mathbf{u} \in \R^{n+1} \\ \max\{1,|u_1|,\dots,|u_n|\}^{n+1} \leq |u_0| 
	 }} \frac{\one(1,u_1,\dots,u_n) }{u_0^2} \mathrm{d} \mathbf{u} \\
	& = \frac{n+1}{2}\int_{\substack{\x \in \R^{n+1} \\ \max\{|x_0|,\dots,|x_n|\} \leq 1 \\
	  }}  \one(\x) \mathrm{d} \x
	\end{align*}
	where we make the change of variables $(u_0,u_1,\dots,u_n) = (1/x_0^{n+1}, x_1/x_0,\dots,x_n/x_0)$,
	which has jacobian determinant $-(n+1)x_0^{-(2n + 2)}$. This shows (1).
	
	Now let $p$ be a prime.  
	The $\sigma$-algebra on $\P^n(\Q_p)$ is generated by the residue discs
	$$D_{\mathbf{y}} : = \{ \x \in \P^n(\Z_p) : \x \equiv \y \bmod p^r\}, \quad \y \in \P^n(\Z/p^r\Z).$$
	By \cite[Thm.~2.13]{Sal98}, these have volume
	$$\tau_p(D_\y) = \frac{\#\P^n(\F_p)}{p^n\#\P^n(\Z/p^r\Z)}
	= \frac{1}{p^{nr}}.$$
	It suffices to prove the result for the residue discs. Writing $\y = (y_0:\dots:y_n)$,
	without loss of generality each $y_i \in \Z_p$ and $y_0 = 1$. Then we have
	$$\widehat{D}_{\mathbf{y}} \cap \Z_p^{n+1} = 
	\left\{ \mathbf{x} \in \ZZ_p^{n+1} : \left| \frac{x_i}{p^{v_p(x_0)}} - \frac{y_ix_0}{p^{v_p(x_0)}} \right|_p \leq p^{-r}, i = 1,\dots,n\right\}.$$
	For $(2)$ we have
	\begin{align*}
		\mu_p(\widehat{D}_{\mathbf{y}}\cap \ZZ_p^{n+1}) & = \sum_{m = 0}^\infty 
		\{ \mathbf{x} \in \ZZ_p^{n+1} : v_p(x_0) = m, |x_i - y_i x_0|_p \leq p^{-r-m}, i = 1,\dots,n\}	\\
		& = \sum_{m = 0}^\infty \frac{1}{p^m} \left(1 - \frac{1}{p}\right) \left(\frac{1}{p^{r+m}}\right)^n	 
		 = \frac{1}{p^{nr}} \cdot \frac{1-1/p}{1-1/p^{n+1}}
	\end{align*}
	as required. 
	For $(3)$ we have
	\begin{align*}
		\mu_p\{ \x \in \widehat{D}_{\mathbf{y}} \cap \Z_p^{n+1} : p \nmid \x\} & = 
		\{ \mathbf{x} \in \ZZ_p^{n+1} : v_p(x_0) = 0, |x_i - y_i x_0|_p \leq p^{-r}, i = 1,\dots,n\}	\\
		& = \left(1 - \frac{1}{p}\right) \frac{1}{p^{nr}}	 
	\end{align*}
	as required.
\end{proof}

\subsection{The case $\Delta(f) = 0$}
Here Conjecture \ref{conj:main} is known in numerous cases by work of Loughran--Smeets \cite{LS16}, Huang \cite{Hua21}, and Browning--Heath-Brown \cite{BHB21}. These are all proved via suitable applications of the geometric sieve (also called the sieve of Ekedahl). Poonen and Voloch's  \cite{PV04} paper was one of the earliest applications of this sieve.

\begin{proposition}
	Conjecture \ref{conj:main} holds when $\Delta(f) = 0$, $k= \Q$, and
	$X$ is either $\P^n$, a split toric variety, 	or a smooth quadric hypersurface of dimension
	at least $3$.
\end{proposition}
\begin{proof}
	The results \cite[Thm.~1.3]{LS16}, \cite[Thm.~1.5]{Hua21}, and \cite[Cor.~1.6]{BHB21}
	show that the limit
	$$\lim_{B \to \infty}  \frac{\#\{ x \in X(\Q) : H(x) \leq B, x \in f(Y(\Adele_\Q)), x \notin \Omega\}}
	{\#\{x \in X(\Q) : H(x) \leq B, x \notin \Omega\}}$$
	exists and equals the correct product of local densities.
	Here $\Omega$ is the set of points in the complement of the open torus 
	when $X$ is a toric variety, and empty otherwise. Since $\Delta(f) = 0$
	we have $\delta_D(f) = 1$ for all $D$. It therefore suffices to show that 
	\begin{equation} \label{eqn:Sub_Delta=0}
		\Sub(X,f) = \Br X.
	\end{equation}
	(Note that in all these cases $\Br X = \Br \Q$ since $X$ is rational).
	We prove this via Lemma~\ref{lem:Sub_formula}.
	Let $D \in X^{(1)}$. The fibre over $D$ is a so-called pseudo-split scheme,
	which means that every element of the absolute Galois group of $k(D)$
	fixes some irreducible component of $f^{-1}(D)$ of multiplicity $1$.
	It easily follows that
	$$\bigcap_E \ker \left(\mathrm{res} : \HH^1(k(D),\Q/\Z) \to \HH^1(\kappa_E,\Q/\Z)\right)$$
	is trivial, where the intersection is over all irreducible components
	$E$ of $f^{-1}(D)$ of multiplicity $1$. The claim \eqref{eqn:Sub_Delta=0}
	now follows from Lemma \ref{lem:Sub_formula} and \eqref{seq:purity}.
 \end{proof}

\subsection{Rational numbers as sum of two squares} \label{sec:elementary}
Our first example with $\Delta(f) > 0$ is the counting problem
\[N(B):=\#\{t\in \Q:H(t) \leq B, t=x^2+y^2  \textrm{ for some } x,y\in \Q \} .\]
This is an elementary warm up  for some of the more difficult examples which will be treated in the paper. It is included for completeness as it appears in Example \ref{ex:crazy}. It also highlights an important subtle point which will occur later on; namely, it is essential that one uses the anticanonical height in Conjecture~\ref{conj:main}! Changing the height function will change the leading constant due to the appearance of $\log B$ factors; the modified Fujita invariants were introduced to balance out these additional factors coming from changing the height function. We begin with the following asymptotic.

\begin{proposition} \label{prop:easy}
	We have
	\[N(B)
	=
 	 \frac{3}{2\pi} \frac{B^2}{\log B}
   \prod_{p=3}^\infty  
    \left( 1+ \frac{\chi(p)}{p} + \frac{1 - \chi(p)}{p(p+1)} \right)
  	+
	O\left( \frac{B^2}{(\log B)^{1.4}} \right) 
,	\]  
	where $\chi$ denotes the non-principal Dirichlet character modulo $4$.
\end{proposition}
\begin{proof}
Choosing coprime  integers $a,b$ to represent $t$, we can write 
 \[N(B)=O(1)+ \frac{1}{2} \sum_{\substack{0< \vert a\vert ,\vert b \vert \leq B \\ \gcd(a,b)=1  }} \mathds 1(aX^2=bY^2+bZ^2 \textrm { soluble in } \Q).\] 
To ensure real solubility, both $a$ and $b$ must have the same sign (we choose them both to be positive).
By the Hasse principle and Hilbert reciprocity, solubility is equivalent to $v_p(ab) \equiv 0 \bmod{2}$ for all $p\equiv 3 \bmod{4}$. 
 Hence,  \[N(B)=O(1)+ \sum_{\substack{0<a,b \leq B \\ \gcd(a,b)=1  
 }}  \prod_{\substack{ p\mid ab   \\   p\equiv 3 \bmod{4} 
 }  }   \mathds 1(v_p(ab) \equiv 0 \bmod{2}  ).\]
Write $ a=2^\alpha s  , b=2^\beta t $ for odd integers $s,t $. As $\gcd(a,b) = 1$, we have
\[
N(B)=O(1)+ \sum_{\substack{ 0\leq \alpha, \beta \leq 2 \log B \\  \alpha \beta =0  }} N(B 2^{-\alpha}, B 2^{-\beta}),
\]
where
$$
N(S,T) =  \sum_{\substack{0<s \leq  S    \\ 2\nmid s   }}  f(s)   \sum_{\substack{ 0< t \leq T, 2\nmid t    \\   \gcd(t,s)=1   }}     f(t) \quad{}\text{ and } \quad{}  f( k ) :=
 \prod_{\substack{ p\mid k   \\   p\equiv 3 \bmod{4}    }  }   \mathds 1(v_p(k ) \equiv 0 \bmod{2}  ) .  $$
Note that $f$ is multiplicative. 
The bound $N(S,T) \leq ST$ shows that the contribution of   $(\alpha, \beta)$ for which $\min\{B 2^{-\alpha}, B 2^{-\beta} \}\leq B^{1/2}$ is 
$$
\ll   \sum_{\substack{ \beta \geq 0  }} B^{1/2} (B 2^{-\beta})
+ \sum_{\substack{  \alpha \geq 0  }} (B 2^{-\alpha} )B^{1/2}
\ll B^{3/2}
,$$ which is negligible. Thus  we only have to asymptotically 
estimate $N(S,T)$ in the range 
  $B^{1/2}\leq S,T \leq B$.
 By the Selberg--Delange method  \cite[Ch.~II.5, Thm.~5.2]{Ten95} 
 (or a modification of the case $q=1,d=s$ of Friedlander--Iwaniec~\cite[Lemma 1]{MR2675875}), 
 we   get     
\begin{equation}\label{eq:twosquaresfirst}
\sum_{\substack{ 0< t \leq T, 2 \nmid t   \\   \gcd(t,s)=1   }}     f(t) 
=c(s) \frac{T}{\sqrt{\log T} } 
 \l\{1+O\l(\frac{(\log \log 3s)^{2}}{\log T}\r)\r\} 
+O\l( \frac{\tau(s)T}{(\log T)^{10}}\r), 
\end{equation}
where $\tau$ is the divisor function, the implied constant is absolute,  and
\[ c(s):= 
\frac{1}{\pi^{1/2}  }    
\prod_{p   } \l(1+\mathds 1 (p\nmid 2s)
\sum_{k\geq 1 } \frac{f(p^k ) }{p^k}\r) \l(1-\frac{1}{p } \r)^{1/2} = c_0 G(s)  ,\] 
    where  
$$ 
c_0 =\prod_{p=3}^\infty 
\begin{cases}
 \l(1-\frac{1}{p } \r)^{-1/2},  & p\equiv 1 \bmod 4 \\
\l(1-\frac{1}{p } \r)^{1/2}
\left( 1-\frac{1}{p^2 } \r)^{-1} , 
& p\equiv 3 \bmod 4
\end{cases}
 ,
G(s) = \prod_{\substack{p \mid s\\ p \equiv 1 \bmod 4}} \frac{p-1}{p} \prod_{\substack{p \mid s \\ p \equiv 3 \bmod 4}} \frac{p^2-1}{p^2}.
$$ 
Let us point out   the few required  modifications  
of the arguments in~\cite[pg.109-111]{MR2675875}. The Dirichlet series of $f(t) \mathds 1(\gcd(t,2s)=1)$
is 
$$  \prod_{\substack{ p\equiv 1 \bmod 4 \\ p\nmid s } } \frac{1}{1-p^{-z}} \prod_{\substack{ p\equiv 3 \bmod 4  \\ p\nmid s} } \frac{1}{1-p^{-2z}}
=\widetilde{L_s}(z) \sqrt{\zeta(z)} \widetilde{R}(z) $$ for all complex numbers $z$ with $\Re(z)>1$, where we let 
$$\widetilde{L_s}(z)=
 \prod_{\substack{ p\equiv 1 \bmod 4 \\ p\mid s } } \l(1-p^{-z} \r) \prod_{\substack{ p\equiv 3 \bmod 4  \\ p\mid s} } \l(1-p^{-2z}\r)
$$ and 
$$ \widetilde{R}(z)=
\prod_{\substack{ p\equiv 3 \bmod 4  } } \l(1-p^{-2z}\r)^{-1}
\prod_{\substack{ p\equiv 1 \bmod 4  } } \l(1-p^{-z}\r)^{-1/2}
\prod_{\substack{ p\equiv 2,3 \bmod 4  } } \l(1-p^{-z}\r)^{1/2}
.$$ This is in accordance to the first equation in~\cite[pg.110]{MR2675875} and the rest of the argument can be continued     similarly.
There is only one  difference, namely that the function $a(h)$ in~\cite[pg.110]{MR2675875} has to be replaced by   $\widetilde{a}(h)$
defined as follows: it is supported on odd numbers of the form $h=h_1 h_3^2$, where $h_i$ is a square-free integer  divisible only
by primes of the form $i \bmod 4 $ and for such $h$   we have $\widetilde{a}(h)=\mu(h_1) \mu(h_2)$. This leads to the bound $|\widetilde{a}(h)| \leq 1$
which is not as good as   $|a(h)|\leq 1/\tau(h)$ that is used in the middle of~\cite[pg.110]{MR2675875}. The analogous argument here gives 
$$\sum_{h\mid s^\infty} \frac{|\widetilde{a}(h)| }{h} \log h
\leq 
\sum_{h\mid s^\infty} \frac{ \log h}{h}=
\l(\prod_{p\mid s} \l(1-p^{-z}\r)^{-1} \r)'_{z=1},$$ which is $$
\ll \prod_{p\mid s} \l(1-1/p\r)^{-1} \sum_{p\mid s} \frac{\log p}{p} \ll (\log \log 3s )^2.$$  
We proceed by applying \eqref{eq:twosquaresfirst}.
Summing its error terms      yields an error term of size  
$$\ll \sum_{0< s\leq S} f(s)  c(s) \frac{T}{\sqrt{\log T} } 
 \frac{(\log \log 3s)^{2}}{\log T} 
+
 \sum_{0< s\leq S} f(s)  \tau(s)\frac{T}{(\log T)^{10}}.$$ By the   bounds $|f(s)| \leq 1, |c(s)| \ll |G(s)| \ll \log \log 3s$
this becomes 
$$ \ll  \frac{ST}{(\log T)^{3/2} }(\log \log 3S)^{3} +S (\log S) \frac{T}{(\log T)^{10}}.$$
Using the bounds $B^{1/2}\leq S,T \leq B$ makes this $O(ST (\log B)^{-1.4} )$, which is admissible.

Summing the  main term in   \eqref{eq:twosquaresfirst} leads to the sum  
$
\sum_{\substack{0<s \leq S \\ s \text{ odd}}} f(s) G(s),
$
which another application of Selberg--Delange shows is equal to 
$$
\frac{1}{\sqrt{2\pi}} \frac{S}{\sqrt{\log S}} \prod_{p \neq 2} \left(1 - \frac{1}{p}\right)^{1/2} \left( 1 + \sum_{k \geq 1} \frac{f(p^k)G(p^k)}{p^k}\right)
+ O\left( \frac{S}{(\log S)^{\frac{3}{2}}} \right).
$$     
This Euler product   simplifies to 
\[ \prod_{p=3}^\infty 
\left( 1- \frac{1}{p}\right)^{1/2} 
\left( 1 + \frac{1}{p^{1+\alpha_p}}\right) 
,\] where $\alpha_p=0$ or $1$ respectively when $p$
is $1$ or $3\bmod 4$.
Multiplying by $c_0$   gives 
\[
\kappa :=  \prod_{p=3}^\infty  \left(  1+ \frac{\chi(p)}{p} + \frac{1 - \chi(p)}{p(p+1)} \right),
\] after a straightforward  calculation.
Therefore $N(B)$ is asymptotic to
\[
\frac{ \kappa}{(\sqrt{2 \pi})^2 }
\sum_{\alpha \beta =0}  \frac{B^2}{2^{\alpha + \beta}\log B} =  
\frac{B^2}{\log B} \frac{3}{2\pi} \prod_{p=3}^\infty 
   \left(1+ \frac{\chi(p)}{p} + \frac{1 - \chi(p)}{p(p+1)}\right) 
.  \qedhere
\] 
 \end{proof} 

We now explain how Proposition \ref{prop:easy} is compatible with Conjecture \ref{conj:main}. The  base variety is $\P^1$. The open set is $U = \Gm$ and the boundary divisors are $0$ and $\infty$. We have 
$$\alpha^*(\P^1) = 1/2, \,\, |\Sub(X,f)/\Br k| = 2, \,\, a(0) =a(\infty) = 2,\,\, \Gamma(X,f) = \Gamma(1/2)^2 = \pi.$$
This calculation of the subordinate Brauer group is relatively easy using Faddeev reciprocity \cite[Thm.~1.5.2]{Brauer}, and the non-trivial representative is given simply by the quaternion algebra $(t,-1)$.
For the local densities, we use Lemma \ref{lem:calculate_Tamagawa} with $S = \emptyset$.

\begin{lemma}  
We have
	$$
	\tau_p(f(Y(\Q_p)))= 
	\begin{cases}
		2, & p = \infty,\\
		3/4, & p = 2, \\
		1 + \frac{1}{p}, &p \equiv 1 \md 4, \\
		1 - \frac{p-1}{p(p+1)}, & p \equiv 3 \md 4.
	\end{cases}$$
\end{lemma}
\begin{proof}
	We use Proposition \ref{prop:local_densities}. This gives 
	$$\tau_\infty(f(Y(\R))) = \mu_\infty\{ (t_0,t_1) \in [-1,1]^2 : t_0/t_1 > 0 \} = 2.$$
	For $p \equiv 1 \md 4$ every conic in the family has a $\Q_p$-point. So
	let $p \equiv 3 \md 4$. The relevant measure is 
\begin{align*}
&(1-1/p)^{-1}\mu_p\{ (t_0,t_1) \in \Z_p^2 : p \nmid \mathbf{t},  v_p(t_0t_1) \in 2 \Z\} \\
 = & (1-1/p)^{-1}(  2(1-1/p)/(1+1/p) - (1-1/p)^2) \\
 =  &2/(1+1/p) - (1-1/p)  
 =  1 - \frac{p-1}{p(p+1)} 
\end{align*}
as expected. (Here we used   $\mu_p\{ x \in \Z_p : v_p(x) \in 2\Z\} = (1+1/p)^{-1}$ and 
inclusion--exclusion.)
For $p=2$ write $t_i = 2^{a_i}u_i$ where $u_i$ is a $2$-adic unit. Then we have 
$$(t_0t_1,-1)_2 = (-1)^{{(u_1u_2 - 1)}/2}, $$ thus the relevant volume is
\begin{equation*}
(1+1/2)\mu_2\{ (t_0,t_1) \in \Z_2^2 : (t_0t_1)/2^{v_2(t_0t_1)} \equiv 1 \bmod 4\}
= (1 + 1/2) \cdot 1/2 = 3/4. \qedhere
\end{equation*}
\end{proof}

The virtual Picard group here is $\Z - (1/2)\Z - (1/2)\Z$, so no convergence factors are required for the Tamagawa measure. Combining all the above with Proposition \ref{prop:easy}, one sees that we are off by a factor of $2$ in Conjecture \ref{conj:main}. This is because we used the naive height in Proposition \ref{prop:easy}, and Conjecture \ref{conj:main} uses the anticanonical height! Therefore it is the counting function $N(B^{1/2})$ which actually gives the correct constant, and this missing factor of $2$ comes from the product of modified Fujita constants $\eta(0)^{1/2}\eta(\infty)^{1/2} = 2$ (here we use Lemma \ref{lem:simplicial}).

\subsection{Polynomial represented by a binary quadratic form -- compatibility with the circle method}
We now consider a more difficult example. Let $g \in \Q[x_0,\dots,x_n]$ be an irreducible homogeneous
polynomial of even degree $d$ and $a \in \Z$ an element which is not a square in the function field
of the divisor $g(x) = 0$.

\begin{lemma} \label{lem:circle_method}
	Conjecture \ref{conj:main} predicts that
	\begin{align*}
		&\#\left\{ x \in \P^n(\Q) : 
		\begin{array}{l}
		H_{-K_{\P^n}}(x) \leq B,  \\
		g(x) = t_0^2 - at_1^2 \text{ for some } t_0,t_1 \in \Q 
		\end{array}
		\right\} 
		\sim \frac{2 \cdot  \prod_v \omega_v }{\sqrt{\pi d(n+1)}} \frac{B}{(\log B)^{1/2}},
	\end{align*}
	where
	\begin{align*}
		\omega_\infty &= 
		\begin{cases}
			(n + 1)2^n	, & a > 0, \\
			((n+1)/2) \cdot \vol\{ \x \in [-1,1]^{n+1} : g(\x) > 0 \}, & a < 0,
		\end{cases} \\
		\omega_p &= 
			(1-1/p)^{1/2}(1+1/p + \dots + 1/p^n)
			\vol \{ \x \in \Z_p^{n+1} : (g(\x),a)_p = 1\}
	\end{align*}
	and $(\cdot,\cdot)_p$ denotes the Hilbert symbol.
\end{lemma}
Note that no non-singularity hypothesis is required on $g$ ($g$ may even be geometrically reducible). In the above volume, we are implicitly ignoring the subset $g(\x) = 0$, which has measure zero.  Note also that $(g(\x),a)_p = 1$ whenever $a \in \Q_p^{\times 2}$.

\begin{proof}
	Firstly no thin set $\Omega$ is required (see \S \ref{sec:thin_set}).
	By Corollary \ref{cor:sub}, we can use Conjecture \ref{conj:Br} instead.
	As $g$ has even degree and is irreducible,
	the quaternion algebra $\alpha = (g,a)$
	is ramified exactly along the divisor $D = \{ g(x) = 0\}$ with residue
	$a \in \kappa(D)$, which is non-trivial by assumption. We have
	$$\alpha^*(\P^n) = 1/(n+1), \quad \eta(D) = (n+1)/d, \quad \Gamma(1/2) = \sqrt{\pi}.$$
	This gives the factors on the denominator. We next claim that
	$$\Sub(\P^n,\alpha) /\Br \Q \cong \Z/2\Z$$
	generated by the image of $\alpha$ (this gives the factor $2$ on the numerator). To see this, 
	let $\beta \in \Sub(\P^n,\alpha)$ be non-constant. Then $\alpha - \beta$ is everywhere
	unramified, hence is constant as $\Br \P^n = \Br \Q$. This proves the claim.
	
	For the local densities, we use Lemma \ref{lem:calculate_Tamagawa} with $S = \emptyset$.
	To calculate these we use Proposition~\ref{prop:local_densities}. 
	For the real density, when
	$a > 0$ there is always a real point so we recover the usual real density of $(n+1)2^n$.
	When $a < 0$ there is a real solution if and only if $g(x)$ is positive. 
	For the $p$-adic densities, the condition that there is a $\Q_p$-point
	is exactly that $(g(\x),a)_p = 1$.
	To finish it suffices to note that the convergence factors are given by the local Euler
	factors of $\zeta(s)^{1/2}$, since the virtual representation here is just 
	$\Z - (1/2)\Z$, and that $\lim_{s \to 1}(s-1)^{1/2} \zeta(s)^{1/2} = 1$.
\end{proof}

We now prove that the formula in Lemma \ref{lem:circle_method} is compatible with predictions from the Hardy--Littlewood circle method, at least when $a = -1$.   Letting   $H$ be the naive Weil height in $\P^{n}$
we shall study
\[N(B):=\#\left\{x\in \P^{n}(\Q): H(x) \leq B, t_0^2+t_1^2=g( x ) \text{ has a } \Q\text{-point}\right\}.\]
By Lemma \ref{lem:circle_method} our result verifies Conjecture \ref{conj:main} after considering the anticanonical counting function $N(B^{1/(n+1)})$.
\begin{theorem}
 \label{thm:birchcircle} 
Let $g \in \Q[x_0,\dots,x_n]$ be an irreducible homogeneous polynomial of even degree $d$.
Assuming that  $g$ is non-singular and that $n+1>(d-1) 2^d$
we have 
$$ \lim_{B\to\infty}
\frac{N(B)B^{-n-1} }{(\log B)^{-1/2} }
=
 \frac{ \vol( \x \in [-1,1]^{n+1}  :g(\x )> 0)}{(\pi d)^{1/2}} 
\prod_{p }  \l(1-\frac{1}{p}\r)^{1/2} 
\hspace{-0,1cm}
\l(1+\frac{1}{p}+\cdots+\frac{1}{p^n} \r) \ell_p
,$$
where for any prime $p$ we let $$\ell_p :=\mathrm{vol}\l(\b x  \in \Z_p^{n+1}:  t_0^2+t_1^2=g( \b x ) \text{ has a } \Q_p\text{-point}\r).$$ 
\end{theorem}
\begin{proof}
Our method of proof is based upon \cite{MR4353917}, as well as Birch's seminal work \cite{Bir62}.

\textbf{ \texttt{Step 1 (using the circle method.)}}
By M\"obius inversion we may use arguments akin
to~\cite[Lemma 2.1]{MR4353917} to write 
 \beq{eq:vislem2.1}{N(B)=\frac{1}{2} \sum_{l\in \N \cap[1, \log(2B) ] } \mu(l) \Theta(B/l)+O(B^{n+1}(\log 2B)^{-1} ),}where 
 $ \Theta(P):=\#\left\{\b x\in \Z^{n+1} : |\b x |\leq P,g(\b x)\neq 0 , t_0^2+t_1^2=g( \b x ) \text{ has a } \Q\text{-point}  \right\}$
and $|\cdot |$ denotes the supremum norm in $\R^{n+1}$.
Letting $\mathrm e(z):= \mathrm e^{2 \pi i z }$ and using the standard circle method integral 
$\int_0^1 \mathrm e(\alpha (g(\b x ) - m )) \mathrm d \alpha $ to detect whether $g(\b x )$   equals 
 a given integer $m $ we can show as in~\cite[Equation (2.2)]{MR4353917} 
that   $$ \Theta(P)= \int_0^1 S(\alpha) \overline{E_\Q(\alpha)}\mathrm d \alpha ,$$ where 
$ S(\alpha):=\sum_{|\b x|\leq P }\mathrm e (\alpha g(\b x )) $ is the standard exponential sum associated to $g$
and $$E_\Q(\alpha)=\sum_{\substack{ 0< m \leq  \max\{g([-1,1]^{n+1} )\} P^d  \\ 
 t_0^2+t_1^2=m \text{ has a } \Q\text{-point}    } } \mathrm e(\alpha m ).$$

\textbf{\texttt{Step 2 (bounding away the minor arcs.)}}
We next show that the `minor arcs' contribute to the error term.
The trivial bound $E_\Q(\alpha)=O(P^d )$ shows that for any set $\mathcal A \subset [0,1)$ one has 
\beq{eq:bwv76}{ \left| \int_{\mathcal A } S(\alpha) \overline{E_\Q(\alpha)}\mathrm d \alpha \right|
\ll P^d  \int_{\mathcal A } \left|S(\alpha) \right|\mathrm d \alpha.} Birch~\cite[Equations (10)-(11), Lemma 4.3]{Bir62} showed 
that there exist $\delta, \theta_0>0 $ that depend
only on $g $ such that when $\c A$ is defined as the complement of 
$$ \bigcup_{1\leq q \leq P^{(d-1)\theta_0}} 
 \bigcup_{a\in (\Z/q\Z)^\times  } \left[\frac{a}{q}-\frac{ P^{-d+(d-1)\theta_0}}{2q} ,\frac{a}{q}+\frac{ P^{-d+(d-1)\theta_0}}{2q} \right]
$$ one has $  \int_{\mathcal A } |S(\alpha) |\mathrm d \alpha=O(P^{n+1-d-\delta})$.
As shown in~\cite[Lemma 4.1]{Bir62}
these intervals are disjoint, hence, by~\eqref{eq:bwv76} we obtain 
\beq{eq:birhtheta_1}{ \Theta(P)= \sum_{\substack{ 1\leq  q \leq P^{(d-1)\theta_0}\\
 a\in (\Z/q\Z)^\times   }} 
  \int_{|\beta|\leq P^{-d+(d-1)\theta_0}} S\left(\frac{a}{q}+\beta\right) \overline{E_\Q}\left(\frac{a}{q}+\beta\right)
\mathrm d \beta 
+O(P^{n+1-\delta}) .} Note that we have replaced the condition 
$|\beta| \leq  \frac{ P^{-d+(d-1)\theta_0}}{2q}$ by 
$ |\beta | \leq  P^{-d+(d-1)\theta_0}$; this can be done owing to the remarks in~\cite[page 253]{Bir62}.

\textbf{\texttt{Step 3 (first approximation in the major arcs.)}}
Now that we have obtained a representation of $\Theta(P)$ as an integral over the `major arcs'
we proceed to approximate     $S(a/q+\beta)$ by  a product of an exponential sum and an exponential integral in each arc. 
Namely, letting 
$$S_{a,q}=\sum_{\b t \in (\Z/q\Z)^{n+1}  } \mathrm e(a f( \b t )/q) 
\ \textrm{ and } \ 
I(\gamma)=\int_{\boldsymbol \zeta \in \R^{n+1} \cap [-1,1]^{n+1} } \mathrm e (\gamma g(\boldsymbol \zeta ))\mathrm d \boldsymbol \zeta,$$
Birch showed in~\cite[Lemma 5.1]{Bir62} that there exists  $\eta>0$ that depends at most on $g $ such that 
$$S\left(\frac{a}{q}+\beta\right)
=q^{-n-1} P^{n+1} S_{a,q}I(P^d \beta)+O(P^{n+2\eta} )
$$ whenever $ |\beta | \leq  P^{-d+(d-1)\theta_0}$.
Injecting this into~\eqref{eq:birhtheta_1}
yields   constants $\theta_i>0$ that depend only on $ g$
such that 
\beq{eq:birhtheta_234}
{
 \Theta(P)= P^{n+1}
 \sum_{\substack{ 1\leq  q \leq P^{ \theta_1}\\
 a\in (\Z/q\Z)^\times   }} \frac{S_{a,q}}{q^{n+1} } 
  \int_{|\gamma|\leq P^{ \theta_2}} 
I(\gamma)
\overline{E_\Q}\left(\frac{a}{q}+\frac{ \gamma}{P^d}\right)
\frac{\mathrm d \gamma}{P^d}
+O(P^{n+1-\theta_3}) 
,} where we used the change of variables $\gamma= P^d \beta$.

\textbf{\texttt{Step 4 (second approximation in the major arcs.)}}
To estimate the term $E_\Q$ inside the integral we use~\cite[Lemma 3.6]{MR4353917}. This  result 
is essentially related to  the   density of integers $m$ in an arithmetic progression that are a sum of two integer squares.
It   will therefore only be useful when $q$ and $\gamma$ have smaller size than in~\eqref{eq:birhtheta_234}
and we thus continue by   restricting the summation and integration range in~\eqref{eq:birhtheta_234}.
To do so we recall  that 
Birch~\cite[Lemmas 5.2 \& 5.4]{Bir62}
showed that there exist constants $\epsilon_i>0$ 
depending only on $g $ such that \beq{eq:absoluteboundsbirch}{
|I(\gamma)| \ll \min\{1,|\gamma|^{-1-\epsilon_1}\}
\ \textrm{ and } \ 
\frac{|S_{a,q}| }{q^{n+1}}\ll q^{-1-\epsilon_2} .}
Combining them with 
the standard result $\max \left\{ |E_\Q(\alpha)|: \alpha \in \R \right\} \ll P^d (\log P)^{-1/2}$
we can use a similar argument as in~\cite[Lemma 4.2]{MR4353917}
to conclude that for each  
constants $A_1,A_2>0$
there exists
a positive constant $A_3$ depending on 
$A_1,A_2$ and $g $ such that $$
\frac{ \Theta(P)}{ P^{n+1}}=
 \sum_{\substack{ 1\leq  q \leq (\log P)^{ A_1}\\
 a\in (\Z/q\Z)^\times   }} \frac{S_{a,q}}{q^{n+1} } 
  \int_{|\gamma|\leq (\log P)^{ A_2}} 
I(\gamma)
\overline{E_\Q}
\left(\frac{a}{q}+\frac{ \gamma}{P^d}\right)
\frac{\mathrm d \gamma}{P^d}
+O\l( (\log P)^{-1/2-A_3}\r).$$
Lemma 3.6 in~\cite{MR4353917}
states   that for all constants $A>0$ and $q\leq (\log P)^A$ one has
$$
E_\Q\left(\frac{a}{q}+\frac{ \gamma}{P^d}\right)
=2^{1/2} \c C_0 \mathfrak F(a,q)
\int_{2}^{\max\{g([-1,1]^{n+1} )\} P^d} \frac{\mathrm e(\gamma P^{-d} t ) }{\sqrt{\log t } }\mathrm dt 
+O\left(\frac{q^3(1+|\gamma|)P^d}{(\log P)^{1/2+1/7}}\right)
,$$ where  $  \mathfrak F(a,q)  $ is defined in~\cite[Equation (3.7)]{MR4353917} and
\beq{def:standc0}{ \c C_0:= \prod_{p\equiv 3 \md 4 }(1-p^{-2})^{1/2}.}
Taking $A_1$ and $A_2$ appropriately small and using~\eqref{eq:absoluteboundsbirch}
shows that there exist $A_i>0$ depending only on $g$
such that 
\beq{eq:bb01}{
\frac{ \Theta(P)}{ P^{n+1}}=
 2^{1/2} \c C_0
\hspace{-0,2cm}
 \sum_{\substack{    q \leq (\log P)^{ A_1}\\
 a\in (\Z/q\Z)^\times   }} 
\hspace{-0,4cm}
\overline{ \mathfrak F(a,q)}
\frac{S_{a,q}}{q^{n+1} } 
\l(\int_{|\gamma|\leq (\log P)^{ A_2}} 
I(\gamma) 
\c H_P(\gamma) 
\frac{\mathrm d \gamma}{P^d}
\r)
+o\l( (\log P)^{-1/2}\r),}
where   $$
\c H_P(\gamma) 
:=  \int_{2}^{\max\{g([-1,1]^{n+1} )\} P^d} \frac{\mathrm e(\gamma P^{-d} t ) }{\sqrt{\log t } }  
\mathrm dt 
 .$$  
By~\cite[Equation (3.11)]{MR4353917} the function $\mathfrak F(a,q)$ is bounded independently of $a$ and $q $,
hence, by~\eqref{eq:absoluteboundsbirch}
we get 
\beq{eq:bb02}{
 \sum_{\substack{ 1\leq  q \leq (\log P)^{ A_1}\\
 a\in (\Z/q\Z)^\times   }} 
\frac{S_{a,q}}{q^{n+1} } 
\overline{ \mathfrak F(a,q)}=
 \sum_{\substack{ q\in \N \\
 a\in (\Z/q\Z)^\times   }} 
\frac{S_{a,q}}{q^{n+1} } 
\overline{ \mathfrak F(a,q)}
+o(1).} Similarly, the trivial bound $\c H_P(\gamma)\ll P^d (\log P)^{-1/2}$
combined with~\eqref{eq:absoluteboundsbirch} shows  that     \beq{eq:bb03}{
\int_{|\gamma|\leq (\log P)^{ A_2}} 
I(\gamma) 
\c H_P(\gamma) 
\frac{\mathrm d \gamma}{P^d}
=\int_{-\infty}^{+\infty}
I(\gamma) 
\c H_P(\gamma) 
\frac{\mathrm d \gamma}{P^d}
+o((\log P)^{-1/2})
.}
Injecting~\eqref{eq:bb02}-\eqref{eq:bb03} into~\eqref{eq:bb01} then yields
\beq{eq:birchconcludsion}{
\frac{ \Theta(P)}{ P^{n+1}}= 
 \mathfrak S
\int_{\R}  I(\gamma)  \frac{\c H_P(\gamma) }{P^d}\mathrm d \gamma
+o\l( (\log P)^{-1/2 }\r),}
where 
 is the analogue of the classical \textit{singular series} in our setting
and given by 
$$ \mathfrak S:=  \sum_{q\in \N} \sum_{ a\in (\Z/q\Z)^\times  } 
\frac{S_{a,q}}{q^{n+1} } \l( 2^{1/2} \c C_0
\overline{ \mathfrak F(a,q)} \r)
.$$ One can simplify the integral over $\gamma$ by using the Fourier analysis approach in~\cite[Lemmas 4.3-4.4]{MR4353917}.
This will result in 
\beq{eq:birchinterpr}{
\int_{\R}  I(\gamma)  \frac{\c H_P(\gamma) }{P^d}\mathrm d \gamma
=\frac{\textrm{vol}(\b t \in [-1,1]^{n+1} :g(\b t )> 0)}{\sqrt{\log(P^d) }}+O((\log P)^{-3/2}) ,} which, when combined with~\eqref{eq:birchconcludsion},
results in
\beq{eq:birchconclreqemion23}{
\lim_{B\to \infty}
\frac{ \Theta(P)}{ P^{n+1}(\log (P^d))^{-1/2}}= 
 \textrm{vol}(\b t \in [-1,1]^{n+1} :g(\b t )> 0)
 \mathfrak S
.}
 \textbf{\texttt{Step 5 (The singular series.)}} One may now use the multiplicative properties of $S_{a,q}$
to express $ \mathfrak S$ as an Euler product; this is a standard procedure albeit somewhat awkward in the presence of the terms
$ \mathfrak F(a,q)$. One can deal with these terms
similarly as in~\cite[Lemmas 5.1]{MR4353917} using nothing more than Ramanujan exponential sums.
Successively, the individual terms of the Euler product can be written as a limit of densities as in~\cite[\S 3.5.2–3.5.3]{fishthesis}; this will result in
the decomposition
$$  \sum_{q\in \N} \sum_{ a\in (\Z/q\Z)^\times  } 
\frac{S_{a,q}}{q^{n+1} }  
\overline{ \mathfrak F(a,q)}  = \ell^*_2 \prod_{p\equiv 3 \md 4} (1-1/p)^{-1} \ell^*_p,$$
where  $$\ell^*_p:=
\lim_{N\to \infty }\frac{\#\{\b x \in (\Z\cap[0,p^N))^{n+1} : t_0^2+t_1^2=g(\b  x ) \text{ has a } \Q_p\text{-point}\}}{p^{N (n+1) }} 
 .$$A straightforward argument involving Haar measures shows that  $\ell^*_p=\ell_p$.
In light of~\eqref{eq:birchconclreqemion23}
we conclude that 
\beq{eqozet584h}{
\lim_{B\to \infty}
\frac{ \Theta(P)(\log P)^{1/2}}{ P^{n+1}}= 
\frac{ \textrm{vol}(|\b t|\leq 1  :g(\b t )> 0)}{d^{1/2}}
2^{1/2} \c C_0  \ell_2 \prod_{p\equiv 3 \md 4}\frac{  \ell_p}{1-1/p}
.} \textbf{\texttt{Step 6 (Conclusion of the proof.)}} Fix any $\epsilon>0$. Denoting the right-hand side of~\eqref{eqozet584h} by $c$
and using~\eqref{eq:vislem2.1} and~\eqref{eqozet584h} 
we infer that for all sufficiently large $B$ one has 
\beq{beatamontev}{
\l|\frac{N(B)}{B^{n+1} }-\frac{c}{2} \sum_{l\leq \log(2B)  } \frac{\mu(l)  l^{-n-1}}{ (\log( B/l ))^{1/2}}\r|\leq \epsilon \sum_{l \leq \log(2B) } 
\frac{ l^{-n-1} }{(\log (B/l))^{1/2}} +O\l(\frac{1}{\log B}\r)
.} For $l \leq \log(2B)$ one has \beq{eq:taylorrr}{(\log (B/l))^{-1/2} = (\log B)^{-1/2}\l (1+O\l(\frac{\log \log B}{\log B}\r)\r),}hence, 
the right-hand side of~\eqref{beatamontev} becomes 
$$ \epsilon \sum_{l \leq \log(2B) } 
\frac{ l^{-n-1} }{(\log (B/l))^{1/2}} +O\l(\frac{1 }{\log B}\r)
\ll 
\frac{1 }{(\log B)^{1/2}} 
  \sum_{l \in \N } 
\frac{\epsilon}{l^{n+1} } + \frac{1 }{\log B} \ll \frac{\epsilon  }{(\log B)^{1/2}} . $$
Furthermore, by~\eqref{eq:taylorrr} the sum in the left-hand side of~\eqref{beatamontev}
is  
$$  \sum_{l\leq \log(2B)  } \frac{\mu(l)  l^{-n-1}}{ (\log( B/l ))^{1/2}}=
\frac{1+o(1) }{(\log B)^{1/2}}
 \sum_{l\leq \log(2B)  } \frac{\mu(l) }{ l^{n+1}}
=\frac{1/\zeta(n+1) +o(1)}{(\log B)^{1/2} }. $$ Putting these estimates together shows that 
$$
\lim_{B\to\infty}\frac{N(B)(\log B)^{1/2}}{B^{n+1} } =
 \frac{ \textrm{vol}(|\b t|\leq 1  :g(\b t )> 0)}{d^{1/2}}
\frac{2^{1/2} \c C_0  \ell_2}{2\zeta(n+1) } 
 \prod_{p\equiv 3 \md 4}\frac{  \ell_p}{1-1/p}
.$$
Alluding to~\eqref{def:standc0} and 
  the standard identity 
 $$\frac{1}{\zeta(n+1) }
= \prod_p \l(1-\frac{1}{p} \r) \l(1+\frac{1}{p}+\cdots+\frac{1}{p^n} \r) 
$$ shows  that  
  \begin{align*}
\frac{ \c C_0   }{ \zeta(n+1) } 
& \prod_{p\equiv 3 \md 4}\frac{  \ell_p}{1-1/p}
=\frac{ 1  }{ \zeta(n+1) } 
 \prod_{p\equiv 3 \md 4} \l(1-\frac{1}{p}\r)^{-1/2}  \l(1+\frac{1}{p}\r)^{1/2}  \ell_p 
\\=
\l(1-\frac{1}{2^{n+1} }\r)
&\prod_{p\equiv 1 \md 4}  \l(1-\frac{1}{p}\r)^{1/2} \l(1+\frac{1}{p}+\cdots+\frac{1}{p^n} \r) \l(1-\frac{1}{p} \r)^{1/2}
\\\times
&\prod_{p\equiv 3 \md 4} 
 \l(1-\frac{1}{p}\r)^{1/2}  \l(1+\frac{1}{p}+\cdots+\frac{1}{p^n} \r) 
 \l(1+\frac{1}{p}\r)^{1/2}  \ell_p 
.
 \end{align*} Upon using the standard fact 
$$ \frac{2}{\pi^{1/2} } =  \prod_{p\equiv 1 \md 4 } \l(1-\frac{1}{p} \r)^{1/2}\prod_{p\equiv 3 \md 4 } \l(1+\frac{1}{p} \r) ^{1/2}
$$the previous expression becomes   \begin{align*}
\frac{2}{\pi^{1/2} } \l(1-\frac{1}{2^{n+1} }\r)
&\prod_{p\equiv 1 \md 4}  \l(1-\frac{1}{p}\r)^{1/2} \l(1+\frac{1}{p}+\cdots+\frac{1}{p^n} \r) 
\\
\times 
&\prod_{p\equiv 3 \md 4} 
 \l(1-\frac{1}{p}\r)^{1/2}  \l(1+\frac{1}{p}+\cdots+\frac{1}{p^n} \r) 
  \ell_p,
 \end{align*} 
which, owing to $\ell_p$ being $1$ for $p\equiv 1 \md 4 $, simplifies to $$ \frac{2}{\pi^{1/2} }
 \l(1-\frac{1}{2^{n+1} }\r)
\prod_{p\neq 2 }  \l(1-\frac{1}{p}\r)^{1/2} \l(1+\frac{1}{p}+\cdots+\frac{1}{p^n} \r) \ell_p .$$ 
This shows that as $B\to\infty$, the function $ N(B)B^{-n-1} (\log B)^{1/2}$ has limit 
$$
  \frac{ \textrm{vol}(|\b t|\leq 1  :g(\b t )> 0)}{(\pi d)^{1/2}}
 2^{1/2}   \ell_2 
 \l(1-\frac{1}{2^{n+1} }\r)
\prod_{p\neq 2 }  \l(1-\frac{1}{p}\r)^{1/2} \l(1+\frac{1}{p}+\cdots+\frac{1}{p^n} \r) \ell_p
,$$ which equals
$$
 \frac{ \textrm{vol}(|\b t|\leq 1  :g(\b t )> 0)}{(\pi d)^{1/2}} 
\prod_{p }  \l(1-\frac{1}{p}\r)^{1/2} \l(1+\frac{1}{p}+\cdots+\frac{1}{p^n} \r) \ell_p,$$ thereby concluding the proof.
 \end{proof}

 \begin{remark}
\label{rem:interpreta}
In  Theorem~\ref{thm:birchcircle}, the  Fujita invariant is  $(n+1)/d$.  Using the anticanonical height 
allows us to see that the term $ d^{1/2}$ in the constant in Theorem~\ref{thm:birchcircle}  is essentially coming from this invariant.
In the proof of  Theorem~\ref{thm:birchcircle}, this constant comes from the singular integral as can be seen from~\eqref{eq:birchinterpr}.
The archimedean density $ \textrm{vol}(\b t \in [-1,1]^{n+1}  :g(\b t )> 0)$ 
also emanates from the singular integral. Finally, the Gamma factor  $\Gamma(1/2)=\pi^{1/2}$ and the product of non-archimedean densities 
come out of the singular series. \end{remark}

\subsection{Anisotropic tori} \label{sec:tori}

Let $T \subset X$ be a smooth equivariant compactification of an anisotropic torus $T$.
Let $H$ be the associated Batyrev-Tschinkel height and $\br \subset \Br_1 T$ a finite
subgroup. Here as $T$ is anisotropic we have $k[T]^\times = k^\times$, so we are in the second
case of Conjecture \ref{conj:Br}.

\begin{lemma} \label{lem:anisotropic}
	Conjecture \ref{conj:Br} holds for $(X,H,\br)$.
\end{lemma}
\begin{proof}
	We use \cite[Thm.~5.15]{Lou18}. However this has some mistakes
	related to the correct way to interpret the Tamagawa measure \eqref{def:tau_f_Sub}
	(specifically  the factor $\tau_\br(T(\Adele_F)_\br^{\Sub(X,\br)})$ appears,
	but as we have explained in \S\ref{sec:volume_BM} this diverges to $0$
	by the choice of convergence factors).
	We take the opportunity to correct these mistakes here, which thankfully is relatively
	simple using Lemma \ref{lem:sum_Euler_products}.
	
	We choose a finite collection $\mathscr{S}$ of representatives for the group
	$\Sub(X,\br)/\Br k$ that evaluate trivially at the identity of $T$.
	As in \cite[\S5.6]{Lou18} and the proof of \cite[Thm.~5.15]{Lou18}, Poisson summation
	shows that the leading constant $c_{X,\br,H}$ equals
    $$\sum_{b \in \mathscr{S}/\Be(T)} \frac{1}{\Gamma(\rho(X) - \Delta(\br))\vol(T(\Adele_F)/T(F))}\lim_{s \to 1} (s-1)^{ \rho(X) - \Delta_X(\br)}
    \prod_v\int_{T(k_v)_\br} \frac{e^{2 \pi i \inv_vb(t_v)}}{ H_v(t_v)^s}\mathrm{d}\mu_v$$
	where the latter integral is the Fourier transform with respect to $b$,
	where $b$ is viewed as an automorphic character via the Brauer--Manin pairing
	\cite[Thm.~4.5]{Lou18}. Using the relation \cite[(5.13)]{Lou18}, we find that	
	\[
    \prod_v\int_{T(k_v)_\br} \frac{e^{2 \pi i \inv_vb(t_v)}}{ H_v(t_v)^s}\mathrm{d}\mu_v
    = \L(X^*(\bar{T}),1) 
    \prod_v\int_{T(k_v)_\br} \frac{e^{2 \pi i \inv_vb(t_v)}}{ H_v(t_v)^{s-1}}\mathrm{d}\tau_v
    \]
    where $\L(X^*(\bar{T}),s)$ is the Artin $L$-function associated to the character
    lattice $X^*(\bar{T})$ of $T$ (this is holomorphic at $s = 1$ as $T$ is anisotropic).
    As in the proof of \cite[Thm.~4.5]{Lou18} we now introduce the factor 
    $\L(\Pic_{\br}(\overline{X})_\CC,s)$, which shows that the leading constant equals
    $$\sum_{b \in \mathscr{S}/\Be(T)} \frac{\L(X^*(\bar{T}),1) }{\Gamma(\rho(X) - \Delta(\br))\vol(T(\Adele_F)/T(F))} \hat{\tau}_f(b)$$
    using the notation of Lemma \ref{lem:sum_Euler_products}. 
    But $\vol(T(\Adele_F)/T(F)) = \L(X^*(\bar{T}),1) |\Pic T|/|\Sha(T)|$ 
    \cite[(4.9)]{Lou18}, $|\Sha(T)| = |\Be(T)|$ \cite[(4.5)]{Lou18} and $\alpha^*(X) = 1/|\Pic T|$.
    This gives the leading constant
        $$\sum_{b \in \mathscr{S}/\Be(T)} \frac{\alpha^*(X)}{\Gamma(\rho(X) - \Delta(\br))|\Be(T)|} \hat{\tau}_f(b).$$    
     Extending the sum over all $b \in \mathscr{S}$ now cancels with the factor $\Be(T)$ 
     on the denominator. Then Lemma \ref{lem:sum_Euler_products} shows that we obtain
     the correct leading constant in Conjecture \ref{conj:Br}, except for possibly
     the modified Fujita invariants. However $\eta(D) = 1$ for any
    divisor $D$ on the boundary. This follows from Lemma \ref{lem:simplicial}
    since $-K_X$ is the sum of the boundary divisors 
    and the boundary divisors freely generate the cone of effective divisors
    \cite[Prop.~1.2.12, Prop.~1.3.11]{BT95}.
\end{proof}

\begin{remark}
	As a warning, we remind the reader that \eqref{eqn:conj_intro}
	need not hold for smooth equivariant compactifications of split tori
	(see Example \ref{ex:crazy}.) It would however be very  interesting to prove Conjecture
	\ref{conj:Br} for an anisotropic torus $T$ and an arbitrary finite subgroup $\br \subset
	\Br T$, i.e.~possibly consisting of transcendental Brauer group elements.
\end{remark}

\subsection{Wonderful compactifications of adjoint semisimple algebraic groups} \label{sec:wonderful}

Let $G \subset X$ be the wonderful compactification of an adjoint semi-simple algebraic group
over a number field $k$.  Let $H$ be an anticanonical height function associated to a smooth adelic metric and $\br \subset \Br_1 G$ a finite subgroup.
Here as $G$ is semi-simple we have $k[G]^\times = k^\times$, so we are in the second
case  of Conjecture \ref{conj:Br}.

\begin{lemma}
	Conjecture \ref{conj:Br} holds for $(X,H,\br)$.
\end{lemma}
\begin{proof}
	We apply \cite[Thm.~4.20]{LTBT20}, with the same caveat in the proof 
	of Lemma \ref{lem:anisotropic} regarding fixing the interpretation of the volume
	of the Brauer--Manin set. A minor variation of proof of Lemma \ref{lem:anisotropic},
	using the spectral decomposition and Lemma \ref{lem:sum_Euler_products},
	shows that
	$$c_{X,\br,H}=\frac{|\Sub(X,\br)/\Br k| \cdot \tau_{\br}((\prod_v G(k_v)_{\br})^{\Sub(X,\br)})}
	{|\Pic G| \cdot \Gamma(\rho(X) - \Delta_X(\br)) \cdot
	\prod_{\alpha \in \mathcal{A}} (1+\kappa_\alpha)^{1/|\res_{D_\alpha}(\br)|}}$$
	where $\mathcal{A}$ denotes the set of boundary components of $X \setminus G$.
	Here $1 + \kappa_\alpha$ is  the coefficient of $D_\alpha$ in $K_{-X}$ 
	\cite[Prop.~3.3]{LTBT20}, thus $\eta(\alpha) = 1+\kappa_\alpha$ by Lemma \ref{lem:simplicial}
	since the boundary divisors freely generate the cone of effective divisors.
	Moreover by \cite[Thm.~4.21]{LTBT20} we have
	$\alpha^*(X) = (|\Pic G| \cdot \prod_{\alpha \in \mathcal{A}} (1+\kappa_\alpha))^{-1}.$
	This agrees with the conjecture.
\end{proof}

\subsection{Diagonal plane conics} \label{sec:diag_constant}
In this final section, we explain how the formula given in Theorem \ref{thm:main} agrees with Conjecture \ref{conj:Br} (this is equivalent to Conjecture \ref{conj:main} here by Corollary \ref{cor:sub}). The proof of Theorem \ref{thm:main} will occur later in the paper.

Consider the variety $V \subset \P^2 \times \P^2$ given by the family of diagonal
plane conics. We have the projection $f: V \to \P^2$ which is a conic bundle over $\P^2$. We are interested in the counting function
$$N(f,B) =  \#\{ x \in \P^2(k) : H(k) \leq B, x \in f(V(\Q))\}$$
where $H$ is the naive height on $\P^2$.
We view $\Gm^2 \subset \P^2$ as the complement of the coordinate axis and use coordinates
$x,y$ on $\Gm^2$. Then the conic bundle corresponds to the quaternion algebra 
$\alpha = (x,y) \in \Br \Gm^2$. This is ramified along the three coordinate axes $D_i: x_i = 0$
with residue of order $2$. We have $\rho(\P^2) = 1$ and $-K_{\P^2} = 3D_i$ in the Picard group.
We take $\br = \langle \alpha \rangle$. 
In the notation of Conjecture \ref{conj:Br} this altogether gives
$$\Delta(\br) = 3/2, \quad \alpha^*(\P^2) = 1/3, \quad \Gamma(\P^2,\br) = \Gamma(1/2)^3 = \pi^{3/2},
 \quad \eta(D_i) = 3$$
The relevant virtual Artin $L$-function is $\zeta(s)^{-1/2}$, which gives the  
convergence factors $(1-1/p)^{1/2}$ and  $\lim_{s \to 1}(s-1)^{-1/2} \zeta(s)^{-1/2} = 1$.

\begin{lemma}
	$\Sub(\P^2,\alpha)/\Br \Q \cong \Z/2\Z$, with generator given by the image of $\alpha$.
\end{lemma}
\begin{proof}
	Let $\beta$ be subordinate to $\alpha$. Then $\beta$ is ramified along
	some subset of the three coordinate axes. 
	It must be ramified along at least two axes since $\Br \A^2 = \Br \Q$. 
	But then
	$\alpha - \beta$ is ramified along at most one axis, thus constant.
\end{proof}
Hence, 	Conjecture \ref{conj:Br} predicts that
\beq{eq:bachactustrgcs}
{
\lim_{B\to \infty } \frac{N(f,B^{1/3} )}{B (\log B)^{-3/2}}
=
  \frac{(1/3) \cdot 2  \cdot  \tau_f(\P^2(\Adele_\Q)_{\br})}
	{\pi^{3/2} } \cdot  3^{3/2}  
 .} Note that we considered $N(f,B^{1/3})$ so as to work with the anticanonical height function on $\P^2$.
For the $2$-adic density  we  shall 
use  the following standard fact which follows from~\cite[pg~79-80, Theorem 4.1, Lemma 4.3]{MR522835}.
\begin{lemma}
\label{lem:qq22} Let  $\b  r \in (\Z_2\setminus \{0\} )^3$ with $v_2(r_0r_1r_2) \in \{0,1\}$. Then $\sum_{i=0}^3 r_1 x_i^2=0$ has a $\Q_2$-point 
if and only if
$$\begin{cases}
	r_i+r_j \equiv 0 \bmod 4, & \text{if } v_2(r_0r_1r_2) = 0, \text{ for some } i \neq j, \\
	r_i+r_j+s r_k \equiv 0 \bmod 8, & \text{if } v_2(r_k) = 1, \text{ for } \{i,j,k\}=\{0,1,2\}
	\text{ and some } s \in \{0,1\}.
\end{cases}$$
\end{lemma}

\begin{lemma}\label{lem:oddprimeshaar}
	We have 
	$$
	\tau_p(f(V(\Q_p)))= 
	\begin{cases}
		9, & p = \infty,\\
		49/48, & p = 2, \\
		\l(1+\frac{1}{p}+ \frac{1}{p^2} \r) \frac{(2p^2+p+2)}{2(p+1)^2}, &
		p \ \mathrm{ odd}.
	\end{cases}$$
 \end{lemma}
 \begin{proof} 
 	We use Proposition \ref{prop:local_densities}. For $p = \infty$ one uses
 	that $ \sum_{i=0}^2 t_i x_i^2=0$ has a real solution exactly when not all signs of 
the $t_i$ coincide; this gives $(3/2) \cdot 6 = 9$. So consider the case where $p$ is a prime.
We need to calculate
\[\mu_p\l(\b t \in \Z_p^3: (-t_0t_1,-t_0t_2)_p = 1 \r)\]
where $(\cdot,\cdot)_p$ denotes the Hilbert symbol. (In the above we are implicitly
removing the subset $t_0t_1t_2 = 0$ of measure zero where the Hilbert symbol is undefined.)
Writing $t_i=p^{v_p(t_i)}{c_i}$ (so that $c_i$ is a unit)
and absorbing squares we find that this volume is
\begin{align*}
=\sum_{\substack{\boldsymbol \lambda\in \{0,1\}^3, \b v \in (\Z\cap[0,\infty))^3\\
\forall i : v_i\equiv \lambda_i \md 2
}}
p^{-v_0-v_1-v_2}
 \mu_p\l( \b c \in \Z_p^{\times 3}:
(-p^{\lambda_0 + \lambda_1}c_0c_1,-p^{\lambda_0 + \lambda_2}c_0c_2)_p = 1\r).
\end{align*}
For $\lambda \in \{0,1\}$
we have 
$
\sum_{   v \geq 0 ,  v \equiv \lambda  \md 2}
p^{-v}=  p^{-\lambda}  (1-p^{-2})^{-1}  $, hence the  above becomes 
\[
  (1-p^{-2})^{-3}
\sum_{\boldsymbol \lambda\in \{0,1\}^3 }p^{-\lambda_0-\lambda_1-\lambda_2}
 \mu_p\l( \b c \in 
\Z_p^{\times 3}:
(-p^{\lambda_0 + \lambda_1}c_0c_1,-p^{\lambda_0 + \lambda_2}c_0c_2)_p = 1 \r).\]
By symmetry this is 
\beq{eq:nonzero}{
 (1-p^{-2})^{-3}\l\{
  (1+1/p^3 ) 
\beta'_p
 +3
(1/p+1/p^2 )
  \beta''_p \r\}
,}
where 
$$
\beta'_p:=\mu_p\l( \b c \in \Z_p^{\times 3}: (-c_0c_1,-c_0c_2)_p = 1 \r),
\quad  \beta''_p:=
\mu_p\l( \b c \in 
{\Z_p^\times}^3:
(-pc_0c_1,-c_0c_2)_p = 1 \r).$$
We now assume that $p$ is odd. Here $(u_1,u_2)_p=1$ for units $u_1,u_2$ thus 
\[\beta'_p=
\mu_p( \Z_p^{\times 3} )=(1-1/p)^3.\] Furthermore $(-c_0c_1,-c_0c_2)_p = 1$ if and only 
$-c_0c_2$ is a square modulo $p$. Thus, 
\[  \beta''_p=   \l(1-\frac{1}{p}\r)\mu_p\l(\b c \in \Z_p^{\times 2}, \l(\frac{-c_0c_2}{p}\r)=1 \r)=  \l(1-\frac{1}{p}\r)\frac{1}{2}  \l(1-\frac{1}{p}\r)^2.\] The proof for odd primes concludes by using~\eqref{eq:nonzero} and simplifying.
By
Lemma \ref{lem:qq22} we can  write 
$\beta'_2$ as 
\[ \sum_{\substack{\b s \in \{1,3\}^3 \\ \b s \notin \{ (1,1,1), (3,3,3)\}}}
  \mu_2\l( \b c \in \Z_2^3:  \b c \equiv \b s  \md 4   \r) =
\frac{1}{4^3} \sum_{\substack{\b s \in \{1,3\}^3 \\ \b s \notin \{ (1,1,1), (3,3,3)\}}}1=\frac{3}{32} .\] 
Similarly, by Lemma \ref{lem:qq22} we see that $\beta''_2$ equals  
\[
 \sum_{\substack{\b s \in \{1,3, 5, 7\}^3 \\ s_1+s_2 \in \{ 0,-s_0  \}\bmod 8}} 
 \hspace{-1cm}
 \mu_2\l( \b c \in \Z_2^3:  \b c \equiv \b s  \md 8   \r) =
\frac{\#\{\b s \in \{1,3, 5, 7\}^3 : s_1+s_2 \in \{ 0,-s_0  \}\bmod 8\}}{8^3},\]
which equals $1/16$,
since   fixing $s_0$ and $s_1$ uniquely 
 determines two distinct values for $ s_2$.
 Alluding to~\eqref{eq:nonzero} shows that 
\[
\tau_2(f(V(\Q_2))) = 
\frac{(1 + 1/2 + 1/4)}{(1-2^{-2})^{3}}
\l\{
  \l(1+\frac{1}{2^3}  \r) 
\frac{3}{32}
 +3
\l(\frac{1}{2} +\frac{1}{2^2} \r )
\frac{1}{16} \r\}
=\frac{49}{48}  . \qedhere
\]
 \end{proof}
Therefore
\begin{align*}
\tau_f(\P^2(\Adele_\Q)_{\br}
&=
  \tau_\infty (\P^2(\R)_\br)
  \prod_{\substack{ p \text{ \rm prime }  } }
  \frac{   \tau_p(\P^2(\Q_p)_\br)}{     (1-1/p )^{1/2} }
\\&
=
9\cdot 
  \frac{  	49/48 }{     (1-1/2 )^{1/2} }
\cdot   \prod_{\substack{ p \text{ \rm prime } \\ p\neq 2  } }
\l(1+\frac{1}{p}+ \frac{1}{p^2} \r)
  \frac{  	 (p^2+p/2+1)  }{  (p+1)^2   (1-1/p )^{1/2} }
,
\end{align*}
which equals $3/2 \cdot c_\infty \prod_p c_p$ in the notation of  Theorem \ref{thm:main}.
Hence, by~\eqref{eq:bachactustrgcs}, 
 	Conjecture~\ref{conj:Br} predicts
$$
\lim_{B\to \infty } \frac{N(f,B^{1/3} )}{c_\infty \prod_p c_p B (\log B)^{-3/2}}
=
  \frac{3^{3/2} }{\pi^{3/2} }  
.$$ 
This agrees with the leading constant in 
Theorem \ref{thm:main}, on taking into account the sign discrepancy between $N(f,B^{1/3})$
and $N(B)$ (see \eqref{def:nalphax}). This shows that Theorem \ref{thm:main} implies
 Conjecture \ref{conj:Br}, hence Conjecture \ref{conj:main}, 
in the case of  diagonal plane  conics.

\begin{remark}
For expositional purposes we decided to pull the factor $3/2$ out of the archimedean density in Theorem \ref{thm:main} so as to use a Lesbegue measure instead of a Tamagawa measure (see Proposition \ref{prop:local_densities}). The reciprocal factor $2/3$ exactly corresponds to the choice of sign multiplied by $\alpha(\P^2) = 1/3$.
\end{remark}

\section{Diagonal plane conics}
\label{sgenguoresthrme9} 
In this section we prove  Theorem~\ref{thm:main}.
The underpinning force behind it is 
Theorem~\ref{thm:genguo}, which  we state and 
prove 
in \S\ref{ss:underspin}.  It generalises   the result of Guo \cite{Guo95} by allowing for  more general equations and   for lopsided height conditions. 
Theorem~\ref{thm:main} is then proved as an application of  Theorem~\ref{thm:genguo} in \S\ref{sec:proofmain}.

\subsection{Statement of Theorem~\ref{thm:genguo}}
\label{ss:underspin}
For   fixed 
 non-zero
 integers  $m_{12},m_{13},m_{23}$  we are interested in      the frequency with which     
 \beq{defeqgg}{
m_{23} n_1 x_1^2+  m_{13} n_2 x_2^2  = m_{12} n_3 x_3^2  }
is soluble in $\Q$, as the  three   square-free integers $(n_1,n_2,n_3)$ range over  
  a box of the form $\prod_{i=1}^3 [-X_i,X_i]$.
Guo's work corresponds to the   case when each $m_{ij}$ equals $1$
and $X_1=X_2=X_3$. 
For the proof of Theorem~\ref{thm:main} we need to allow   
the coefficients $m_{ij}$ to tend to infinity with $X_i$.
For the application it is necessary  to add more flexible assumptions:
for  $X_1,X_2,X_3\geq 1 $  and $\b b =( b_1,b_2,b_3),  \b m=(m_{12},m_{13}, m_{23}) \in \N^3$  we  
let    \beq{def:nbmx}{ \c N_{\b b , \b m }(\b X) =  \sum_{\substack{\b n \in \N^3 \\ \eqref{eq:2dnilb}, \eqref{eq:sn4ws} }}
\mu^2(n_1n_2n_3) 
\begin{cases} 1, \text{ if } ~\eqref{defeqgg}  \text{ has a }\QQ\text{-point},\\ 0, \text{ otherwise,} \end{cases}
}where  the conditions in the summation are
\begin{align}
 & n_1    \leq    X_1,  \quad    n_2   \leq    X_2,  \quad
n_3   \leq   X_3, \label{eq:2dnilb} \\
& \gcd(    n_1 n_2 n_3, m_{12} m_{13} 
m_{23}    ) =  \gcd( n_1 ,   b_2,b_3)  = \gcd( n_2 ,   b_1,b_3)  = \gcd( n_3 ,   b_1,b_2)  =1.\label{eq:sn4ws} 
\end{align}
Define 
\beq{defbetabm}
{
\beta(\b b , \b m )= 
 \prod_{  p\neq 2  } 
 \l(1-\frac{1}{p} \r)^{\frac{3}{2} }
\l(1+  \frac{\#\{1\leq i < j  \leq 3 : p\nmid m_{12} m_{13}  m_{23} \cdot \gcd(b_i,b_j)  \}}{2 p} \r) }
and 
\[ c(\b b , \b m )=  
\begin{cases}
2, & \text{ if } 2 \text{ divides  } m_{12}m_{13}m_{23} ,\\
 3+\#\{ 1\leq  i< j \leq 3  : 2\nmid \gcd(b_i,b_j)\}   , & \text{ otherwise}.
 \end{cases}\] 
For a positive integer $m$ let $m_{\textrm{odd}}:=m2^{-v_2(m)}$, where $v_2$ is the standard $2$-adic valuation.
We denote by $\tau(m)$ the standard divisor function.
 \begin{theorem}\label{thm:genguo} Fix any $0<\eta<1$ and any $A>0$. Then for all $\b b, \b m \in \N^3$ with 
$m_{12}m_{13}m_{23}$ square-free, 
\begin{eqnarray}  
\gcd(b_1,b_2,b_3   )=
\gcd(m_{12},b_3   )= \gcd(m_{13},b_2   )= \gcd(m_{23},b_1   )=1, \nonumber  \\
 b_1,b_2,b_3,m_{12},m_{13},m_{23} \leq(\log\min_{i=1,2,3}  X_i)^A,  \label{eq:should}
\end{eqnarray}
and  all $X _1,X_2,X_3\geq 2  $   with 
\begin{equation}\label{eq:should23}\min_{i=1,2,3} X_i \geq 
\l(\max_{i=1,2,3} X_i \r)^\eta\end{equation}   
we have \[  \mathcal{N}_{  \b b, \b m}( \b X) =
 \frac{    (2\pi )^{-3/2}  \beta(\b b , \b m )
c(\b b , \b m )
  }{  2 \tau(( m_{12}    m_{13} m_{23})_{\rm{odd}})   }
  \prod_{i=1}^3  \frac{X_i}{ (\log X_i)^{1/2} }
 + O  \l ( \frac{X_1 X_2 X_3 (\log \log 3 \c B )^{3/2}}{ (\log  X_1)^{5/2} }\r )   ,\]  where 
$\c B=\max\{ b_1,b_2,b_3,m_{12},m_{13},m_{23}\}$ and 
the implied constant depends at most on $A$ and $\eta$.  \end{theorem}

The structure of the proof of Theorem~\ref{thm:genguo} is as follows:
we transform the counting function  into  averages of  explicit arithmetic  functions in \S\ref{s:transformdetector}.
Asymptotics  for  averages of   certain auxiliary three-dimensional
arithmetic functions are given in \S\ref{s:preparatory}.
The main term contribution is then studied in  \S\ref{s:maintermthrm}.
The error term treatment comprises two parts: 
characters of large conductor (in \S\ref{s:largecond})
and characters of small conductor (in \S\ref{s:smalcond}). 
The final step in the proof of Theorem~\ref{thm:genguo} is given in   \S\ref{sec:finstep}. 

\begin{remark} 
\label{rem:erterm}
 Theorem~\ref{thm:genguo} generalises the result of Guo~\cite[Theorem 1.1]{Guo95}.
Indeed, taking $A=\eta=1/2$, $\b m =\b b = (1,1,1)$ and $ \b X=(N,N,N)$ 
and multiplying by $6$ to   account for positive and negative signs
yields the leading constant 
\[
6 \frac{    (2\pi )^{-3/2}  6  }{  2   } \beta(\b 1 , \b 1 )
= 
   \frac{  1 }{ (\sqrt \pi /2)^3   } 
  \prod_{  p   } 
 \l(1-\frac{1}{p} \r)^{\frac{3}{2} }
\l(1+  \frac{3}{2 p} \r) 
 ,\] which coincides with the one in~\cite[Theorem 1.1]{Guo95}. Furthermore, 
our error  term     is $O(N^3\log N)^{-5/2})$ which improves the error term 
$O(N^3 (\log N)^{-2})$ in~\cite[Theorem 1.1]{Guo95}. We believe our error term
is best possible.
 \end{remark}
 
\subsection{Using the Hasse principle to find detector functions}
\label{s:transformdetector}
We now start   the proof of Theorem~\ref{thm:genguo}.
The opening  move is to find 
 detector functions for solubility over $\mathbb Q$. 
Although it   seems     unusual,    we separate right from the start 
 those    terms that will later give rise to the main term in Theorem~\ref{thm:genguo}.

 For a  prime $p$ and $a,b \in \Q_p$  we denote by 
$(a,b)_p$ the standard Hilbert symbol (see, for example, \cite[Chapter III]{Ser73}.)

 \begin{lemma}\label{lem:predamnlabel} Let  
$a,b,c \in \N$  be such that $abc$ is divisible by at least one odd prime.
Then  
\[ \frac{\l( 1+(ac,bc)_2\r) }{2 \tau(r) } \bigg(1+(ac,bc)_2+\sum_{\substack{\delta \in \N, \delta \mid r  \\  \delta\notin \{1,r\}} } 
\prod_{p\mid \delta } (ac,bc)_p 
\bigg)= \begin{cases} 1, \text{ if } aX^2+bY^2=cZ^2  \text{ has a }\QQ\text{-point},\\ 0, \text{ otherwise,} \end{cases}
  \] where  $r$ is the square-free integer given by  the product of all odd prime divisors of $abc$. 
 \end{lemma}
\begin{proof}By  the Hasse--Minkowski theorem we write the indicator function as 
\[ \prod_{p \text{ prime}} \l( \frac{1+ (ac,bc )_p }{2}\r) ,\] because the conic has a point in $\R$. Recall  that $p\nmid 2r$ implies that $ (ac,bc )_p=1$, hence
\[ \prod_{p \text{ prime}} \l( \frac{1+ (ac,bc )_p }{2}\r) =\prod_{p \mid 2 r } \l( \frac{1+ (ac,bc )_p }{2}\r) =
 \frac{(1+ (ac,bc )_2) }{2} \prod_{p \mid   r } \l( \frac{1+ (ac,bc )_p }{2}\r) .\] 
The product over $p\mid r $ in the outmost right-hand side  becomes 
 \[  \sum_{\substack{ \delta\in \N  \\ \delta \mid r }}\prod_{p\mid \delta }  (ac,bc)_p = 
1+
 \bigg( \sum_{\substack{ \delta\in \N,   \delta \mid r\\ \delta \notin\{ 1 , r\} } } \prod_{p\mid \delta }  (ac,bc)_p \bigg)
+  \prod_{p\mid r  }  (ac,bc)_p .\] By Hilbert's product formula we then obtain  
$$\prod_{p\mid r }  (ac,bc)_p=\prod_{\substack{ p \textrm{ prime } \\ p\neq 2 }} (ac,bc)_p=(ac,bc)_2.$$ 
Here we implicitly used our assumption that $r>1$. This concludes the proof. 
\end{proof}
  \begin{lemma}\label{lem:abmdbvl55} 
Assume that $a,b,c$ are as in Lemma~\ref{lem:predamnlabel} and that $abc$ is square-free. 
Then the sum over $\delta 
\notin \{1,r\}
$ in Lemma~\ref{lem:predamnlabel} equals \[
\sum_{\substack{\boldsymbol  \delta , \widetilde {\boldsymbol \delta} \in \N^3:
\delta_1  \delta_2\delta_3 \neq 1, 
\widetilde {\delta_1}\widetilde {\delta_2}\widetilde {\delta_3}\neq 1 
 \\ \delta_1\widetilde {\delta_1}= a_{\rm{odd}}, \delta_2\widetilde {\delta_2}=b_{\rm{odd}}, \delta_3 \widetilde {\delta_3}= c_{\rm{odd}}} } 
\left(\frac{bc}{\delta_1}\right)  \left(\frac{ac}{\delta_2}\right) \left(\frac{-ab}{\delta_3}\right).\] \end{lemma}
\begin{proof} Since $a,b,c $ are pairwise coprime we can write every $\delta \mid abc $ uniquely as $\delta_1\delta_2\delta_3$, where 
$$\delta_1\mid a_{\rm{odd}}, \, \delta_2\mid b_{\rm{odd}}
\ \textrm { and }  \ \delta_3 \mid c_{\rm{odd}}.$$ Defining $\widetilde {\delta_1}= a_{\rm{odd}}/\delta_1$ and similarly for $\widetilde {\delta_2}, 
\widetilde {\delta_3}$,  we can then write the sum as 
\[\sum_{\substack{\boldsymbol  \delta , \widetilde {\boldsymbol \delta} \in \N^3:
\delta_1  \delta_2\delta_3 \neq 1, 
\widetilde {\delta_1}\widetilde {\delta_2}\widetilde {\delta_3}\neq 1 
 \\ \delta_1\widetilde {\delta_1}= a_{\rm{odd}}, \delta_2\widetilde {\delta_2}=b_{\rm{odd}}, \delta_3 \widetilde {\delta_3}= c_{\rm{odd}}} }  \prod_{p\mid \delta_1 \delta_2\delta_3  } (ac,bc)_p .\] For any integers $s,t $ and an odd prime $p$ with $v_p(s)=1, v_p(t)=0$ we can write $(s,t)_p=(\frac{t}{p})$, where $(\frac{\cdot }{\cdot })$ is the Jacobi symbol. Using that  $\delta_1 \delta_2\delta_3$ is square-free we get 
\[\prod_{p\mid \delta_1   } (ac,bc)_p=  \left(\frac{bc}{\delta_1}\right)
\ \textrm{ and } \ \prod_{p\mid  \delta_2   } (ac,bc)_p=  \left(\frac{ac}{\delta_2}\right)
.\] Recalling that  $(s,t)_p=(\frac{-st p^{-2} }{p})$ for $s,t \in \Z$,  $p\neq 2 $ with $v_p(s)=1=v_p(t)$, yields 
\begin{equation*}
\prod_{p\mid \delta_3   } (ac,bc)_p=  \left(\frac{-ab}{\delta_3}\right). \qedhere
  \end{equation*} 
\end{proof}

\begin{lemma}\label{lem:detectlem4}  Under the assumptions of Theorem~\ref{thm:genguo} we have 
  $$\c N_{\b b , \b m }(\b X)=\frac{2\c M_{\b b , \b m }(\b X)}{\tau(( m_{12}    m_{13} m_{23})_{\rm{odd}})} 
+\frac{\c E_{\b b , \b m }(\b X)}{\tau(( m_{12}    m_{13} m_{23})_{\rm{odd}})} 
+O(1),$$
where the implied constant is absolute, 
\[\c M_{\b b , \b m }(\b X) :=\sum_{\substack{\b n \in \N^3, \eqref{eq:2dnilb}, \eqref{eq:sn4ws}  \\
(n_1   n_3   m_{12}  m_{23}  ,n_2    n_3  m_{13} m_{12} )_2 =1}} 
 \frac{\mu^2(n_1n_2n_3)}{\tau((n_1n_2n_3 )_{\rm{odd}})} 
\]and    $ \c E_{\b b , \b m }(\b X)$ is given by 
\begin{align*}
  &
\sum_{\substack{ \b h, \widetilde{\b h }\in \N^3 \\ h_{ij}\widetilde{ h_{ij} }=(m_{ij})_{\rm{odd}}  }}  
 \sum_{\substack{ \boldsymbol \sigma \in \{0,1\}^3
 \\ \eqref{eq:bwv228}
} }
 \sum_{\substack{ \b d   , \widetilde{\b d } \in \N^3, \eqref{eq:commonivy1},\eqref{eq:commonivy2}
\\
d_i \widetilde d_i \leq X_i 2^{-\sigma_i} 
 } }  \l(\frac{-1}{d_3m_{12} }\r) 
  \frac{\mu^2(2 d_1d_2d_3 \widetilde d_1  \widetilde d_2  \widetilde d_3 ) }{\tau(  d_1d_2d_3 \widetilde d_1  \widetilde d_2  \widetilde d_3 )}
\times 
\\
\times 
& 
  \l(\frac{
2^{\sigma_2+\sigma_3}  
d_2 \widetilde {d_2} d_3 \widetilde {d_3}m_{12}m_{13}}{d_1 h_{23}}\r)
 \l(\frac{2^{\sigma_1+\sigma_3}   d_1 \widetilde {d_1} d_3 \widetilde {d_3}m_{12}m_{23}}{d_2 h_{13}}\r)
 \l(\frac{2^{\sigma_1+\sigma_2}   d_1 \widetilde {d_1} d_2 \widetilde {d_2}m_{13}m_{23}}{d_3 h_{12}}\r)
,\end{align*} where  
\end{lemma}
 \begin{equation}\label{eq:bwv228}
\begin{cases}  
\sigma_1+\sigma_2+\sigma_3\leq 1, \quad \gcd(2^{\sigma_1+\sigma_2+\sigma_3}  ,m_{12}m_{13}m_{23} )=1,
\\
\gcd(2^{\sigma_1}   ,b_2,b_3)=\gcd(2^{\sigma_2}   ,b_1,b_3)=\gcd(2^{\sigma_3}   ,b_1,b_2)=1,
 \end{cases} \end{equation} 
and
\begin{equation}\label{eq:commonivy1}
\begin{cases}     
d_1 d_2 d_3  h_{12}h_{13}h_{23} \neq 1, \quad \widetilde d_1  \widetilde d_2 \widetilde d_3
\widetilde {h_{12}}\widetilde {h_{13}}\widetilde {h_{23}} \neq 1,
\\(2^{\sigma_1+\sigma_3}  d_1 \widetilde d_1d_3 \widetilde d_3  m_{12}  m_{23}  ,2^{\sigma_2+\sigma_3}
d_2 \widetilde d_2 d_3 \widetilde d_3  m_{13} m_{12} )_2 =1,
\end{cases} \end{equation} 
  \begin{equation}\label{eq:commonivy2}
\begin{cases}     \gcd(  d_1 \widetilde d_1 ,b_2,b_3) =\gcd(d_2 \widetilde d_2  ,b_1,b_3)=\gcd(d_3 \widetilde d_3  ,b_1,b_2) =1, \\
\gcd(d_1 \widetilde d_1d_2 \widetilde d_2d_3 \widetilde d_3,m_{12}m_{13}m_{23})=1.
\end{cases}
\end{equation}

\begin{proof}  To employ Lemma~\ref{lem:abmdbvl55} 
with $ a=n_1       m_{23}, b= n_2     m_{13}, c = n_3  m_{12} $ we must ensure that  $n_1n_2n_3$ is divisible by a prime $p>2$.
Owing  to  the term $\mu^2(n_1n_2n_3)$    in $\c N_{\b b , \b m }(\b X)$, this is  equivalent to $n_1n_2n_3>2$.
 Hence, ignoring all terms with 
 $1\leq n_1n_2n_3\leq 2 $ shows that $\c N_{\b b , \b m }(\b X)$ equals 
$$
\sum_{\substack{\b n \in \N^3, n_1n_2n_3>2, \eqref{eq:2dnilb}, \eqref{eq:sn4ws}  \\
(n_1   n_3   m_{12}  m_{23}  ,n_2    n_3  m_{13} m_{12} )_2 =1}} 
\hspace{-0,3cm}
\frac{\mu^2(n_1n_2n_3)}{\tau((n_1n_2n_3m_{12}    m_{13} m_{23})_{\rm{odd}})} 
 \left(2+\hspace{-0,4cm}
\sum_{\substack{\boldsymbol  \delta , \widetilde {\boldsymbol \delta} \in \N^3\setminus\{\b 1 \} , \delta_1\widetilde {\delta_1}= (n_1       m_{23})_{\rm{odd}}\\ \delta_2\widetilde {\delta_2}=( n_2     m_{13})_{\rm{odd}}, \delta_3 \widetilde {\delta_3}= (n_3  m_{12})_{\rm{odd}}} } 
\theta(\boldsymbol  \delta )  \right) $$  up to $O(1)$, where $$\theta(\boldsymbol  \delta )=
\left(\frac{n_2   n_3  m_{12} m_{13}   }{\delta_1}\right)  
\left(\frac{n_1       m_{23}n_3  m_{12} }{\delta_2}\right) \left(\frac{-n_1       m_{23}n_2     m_{13}}{\delta_3}\right)
.$$ We can ignore the condition $n_1n_2n_3>2$ at the cost of an error term that is $O(1)$.
The factor $2$ in the brackets then gives rise to $\c M_{\b b , \b m }(\b X) $. To analyse the term containing     $\theta$ 
we use the assumption that $n_i$ is coprime to $m_{jk}$ to write 
$\delta_i = d_i h_{jk}$ and $ \widetilde{ \delta_i} =  \widetilde{ d_i } \widetilde {h_{jk} }$ 
for some $d_i,\widetilde {d_i}, h_{jk}, \widetilde {h_{kj} } \in \N$
satisfying  $$d_i\widetilde {d_i}=( n_i)_{\rm{odd}} \ \textrm { and } \ h_{jk}\widetilde {h_{jk}}=(       m_{jk})_{\rm{odd}}.$$  
Defining $\sigma_i=v_2(n_i)$   and using $n_i =2^{\sigma_i} d_i\widetilde {d_i}$ we infer that $\theta$ equals 
\[\l(\frac{-1}{d_3m_{12} }\r)\prod_{\{i,j,k\}=\{1,2,3\}}\l(\frac{2^{\sigma_j+\sigma_k}d_j \widetilde {d_j} d_k \widetilde {d_k}m_{ij}m_{ik}}{d_i h_{jk}}\r), \] 
which concludes the proof.\end{proof}

\subsection{Asymptotic averages of $3$-dimensional arithmetic functions} \label{s:preparatory}
We establish asymptotics for sums of the form  
$$
\sum_{\substack{ \b n \in \N^3  \\ \forall i : 0<n_i \leq X_i  }} \mu^2(n_1n_2n_3) g_1(n_1)  g_2(n_2)  g_3(n_3) 
,$$ where $g_i $ are certain multiplicative functions that essentially behave like the inverse of the divisor function.
There are multidimensional generalisations of the Selberg--Delange theory
 under the condition that the underlying   
Dirichlet series    has poles of    integral  order~\cite{MR1858338}. In our   application the series   behaves as  $(\zeta(s_1)\zeta(s_2)\zeta(s_3) )^{1/2}$, 
where $\zeta$ denotes the Riemann zeta function. 
The poles of this series   are half-integers orders, thus we cannot use the existing theory.
We shall instead   use the convolution identity $\mu^2(n)=\sum_{d^2\mid n }\mu(d) $ to
reduce  to the   $1$-dimensional Selberg--Delange setting.
  \begin{lemma}
\label{lem:bwv225hcab}Assume that $g_1,g_2,g_3:\N\to \mathbb C$ are multiplicative function with $|g_i(n)| \leq 1 $ for all $n$ and $i$.
Then for all $X_1,X_2,X_3>1$, $1\leq z\leq (X_1X_2X_3)^{1/4} $ and $\b q \in \N^3, \b a \in \Z^3$ with $\gcd(a_i,q_i)=1$ for all $i$, 
we have 
\begin{align*}
&\sum_{\substack{ k\in \N \cap [1,z]  } } \mu(k)\hspace{-0,3cm}
\sum_{\substack{ \b n \in \prod_{i=1}^3 (\N \cap [1,X_i] ), k^2\mid n_1n_2n_3   \\n_i\equiv a_i \md {q_i} \forall i  } }
\hspace{-0,5cm} g_1(n_1)g_2(n_2)g_3(n_3) 
\\
=&\sum_{\substack{ \b n \in \prod_{i=1}^3 (\N\cap[1,X_i])\\ \forall i : n_i\equiv a_i \md{q_i } } } 
\hspace{-0,5cm}
\mu^2(n_1n_2n_3) g_1(n_1)g_2(n_2)g_3(n_3) 
+O\l(\frac{X_1X_2X_3}{z^{1/2} }(\log (X_1X_2X_3) )^3\r),
\end{align*} where the implied constant     is absolute.\end{lemma}
\begin{proof}Using $\mu^2(n_1n_2n_3) =\sum_{k^2\mid n_1n_2n_3 } \mu(k)$ we can write the sum in the lemma as 
\[\sum_{\substack{ k\in \N } } \mu(k)
\sum_{\substack{ \b n \in \prod_{i=1}^3 (\N \cap [1,X_i] )\\
k^2\mid n_1n_2n_3   , n_i\equiv a_i \md {q_i} \forall i  } }
 g_1(n_1)g_2(n_2)g_3(n_3).\]The contribution of $k>z$ is trivially \[\ll \sum_{k>z} \sum_{\substack{t\leq X_1X_2X_3\\  k^2 \mid t }} \tau_3(t)=
\sum_{k>z} \sum_{s\leq X_1X_2X_3/k^{2}   } \tau_3(k^2s )
,\] where $\tau_3(t)$ denotes the number of ways 
that $t$ can be written as a product of three positive integers. Using the inequality 
$\tau_3(ab)\leq 
\tau^2(ab ) \leq 
\tau^2(a)\tau^2 (b)$, valid for all $a,b\in \N$,
we obtain the bound  \[\sum_{k>z} \tau^2(k^2 )\sum_{s\leq X_1X_2X_3/k^{2}   } \tau^2( s )\ll 
\sum_{k>z} \tau^2(k^2 ) \frac{ X_1X_2X_3}{k^2} (\log (X_1X_2X_3) )^3
 ,\] where we used the fact that  the average of $\tau^2(s)$ is $\log^3 s$.
The proof now concludes by using 
 and the bound $\tau^2(k^2)\ll k^{1/2}$ .\end{proof}
\begin{lemma}\label{lem:bwv2abcfe764}For all $k\in \N$ and $\epsilon>0$ we have $$
\sum_{\substack{ n\in \N,  k^2 \mid n \\ p\mid n \Rightarrow p\mid k } }  n^{-1/2-\epsilon} \ll_\epsilon  k^{-1-\epsilon} ,$$ where the implied constant depends at most on $\epsilon$.
\end{lemma}\begin{proof} Letting $n=k^2 m $ makes the sum equal to 
$$
\sum_{\substack{ m\in \N  \\ p\mid m  \Rightarrow p\mid k } } \frac{k^{-1-2\epsilon} }{m^{1/2+\epsilon}}
\leq 
\sum_{\substack{ m\in \N  \\ p\mid m  \Rightarrow p\mid k } } \frac{k^{-1-2\epsilon}}{m^{1/2}}
=
k^{-1-2\epsilon}\prod_{p\mid k } \frac{1}{1-1/\sqrt p } 
\leq 
k^{-1-2\epsilon}\prod_{p\mid k } \frac{1}{1-1/\sqrt 2 }
$$ Noting that $ 1-1/ \sqrt 2  \geq 1/8 $ and that $8^{\#\{p\mid k \} }\leq \tau(k)^3\ll k ^\epsilon$
concludes the proof.
 \end{proof}
\begin{lemma}
\label{lem:bwv225hcab676764}
Keep the setting of Lemma~\ref{lem:bwv225hcab}. Then for all $w\geq 1 $ we have \begin{align*}
&\sum_{\substack{ k\in \N \cap [1,z]  } } \mu(k)\hspace{-0,2cm}
 \sum_{\substack{ \b n' \in (\N\cap[1,w])^3, p\mid n_1'n_2'n_3' \Rightarrow p\mid k \\
k^2\mid n_1'n_2'n_3' ,\ 
\gcd(n_i',q_i)=1 \forall i } }\hspace{-0,3cm}
g_1(n_1')g_2(n_2')g_3(n_3') 
\prod_{i=1}^3 \Bigg(
\sum_{\substack{  n'' \in \N \cap [1, X_i/n'_i],  \gcd(n'',k)=1
  \\   n'' \equiv a_i/n'_i \md {q_i} }}
g_i(n'') \Bigg)
\\ 
=&\sum_{\substack{ k\in \N \cap [1,z]  } } \mu(k)
\sum_{\substack{ \b n \in \prod_{i=1}^3 (\N \cap [1,X_i] )  \\
k^2\mid n_1n_2n_3  , n_i\equiv a_i \md {q_i} \forall i  } }g_1(n_1)g_2(n_2)g_3(n_3) 
+O\l(\frac{X_1X_2X_3}{w^{1/3} } \r)
,
\end{align*} where the implied constant     is absolute.
      \end{lemma}\begin{proof}
Write each $n_i$ as $n'_in''_i$, where $n''_i$ is coprime to $k$ and all prime factors of $n_i'$ divide $k$. The sum over $k $ in 
Lemma~\ref{lem:bwv225hcab}
then becomes 
  \[\sum_{\substack{ k\in \N \cap [1,z]  } } \mu(k)
 \sum_{\substack{ \b n' \in \N^3, k^2\mid n_1'n_2'n_3'   \\  p\mid n_1'n_2'n_3' \Rightarrow p\mid k\\
\gcd(n_i',q_i)=1 \forall i } }
g_1(n_1')g_2(n_2')g_3(n_3') 
\prod_{i=1}^3 \sum_{\substack{  n'' \in \N \cap [1, X_i/n'_i]   \\ \gcd(n'',k)=1  \\   n'' \equiv a_i/n'_i \md {q_i} }}g_i(n'') 
.\]  The contribution of terms with $\max_i n_i'>w$ is 
\[
\ll \sum_{\substack{ k\in \N \cap [1,z]  } }  
 \sum_{\substack{ \b n' \in \N^3, k^2\mid n_1'n_2'n_3'   \\  p\mid n_1'n_2'n_3' \Rightarrow p\mid k\\
\max n_i'>w} }  \frac{X_1X_2X_3}{n_1'n_2'n_3'}
\leq  X_1X_2X_3 \sum_{k\in \N }   \sum_{\substack{ n \in \N , k^2\mid n   \\  p\mid n \Rightarrow p\mid k\\ n>w} } \frac{\tau_3(n) }{n} 
\leq  \frac{X_1X_2X_3}{w^{1/3} }
 \sum_{k\in \N }   \sum_{\substack{ n \in \N , k^2\mid n   \\  p\mid n \Rightarrow p\mid k } } \frac{\tau_3(n) }{n^{2/3}} 
 .\] The inequality  $\tau_3(n) \leq \tau^2(n) \ll n^{1/12 }$ and   
Lemma~\ref{lem:bwv2abcfe764} with $\epsilon=1/12$ show  that the sum over $k  $ is bounded.  \end{proof}

We define
$$t_0=\frac{1}{\sqrt \pi} \prod_{p \text{ prime}} t_p \l(1-\frac{1}{p}\r)^{1/2}, \qquad 
t_p:=1+\sum_{k=1}^\infty \frac{1}{(k+1)p^k}.$$

\begin{lemma}
\label{lem:1overdiv} Let  $a, d$ be   odd integers and  $q\in \{4,8\} $. Fix any $C>0$.
For all $x\geq 2$ we have \[\sum_{\substack{ 1\leq n \leq x \\ \gcd(n,d)=1\\n\equiv a\md q } }\frac{1}{\tau(n)}=
\frac{t_0 }{\phi(q)
(\prod_{p\mid 2d} t_p)}
\frac{x}{\sqrt{\log x } }
\l\{1+O\l(\frac{(\log \log 3|d|)^{3/2}}{\log x}\r)\r\} 
+O\l(\frac{\tau(d)x}{(\log x)^C}\r),\] where the implied constant depends at most on $C$.
\end{lemma}\begin{proof} A straightforward argument that uses   the zero-free region for Dirichlet $L$-functions
and  a Hankel contour integral, as in the proof of~\cite[Lemma 1]{MR2675875}, for example, shows that   
\[
\sum_{\substack{ 1\leq n \leq x \\ \gcd(n,d)=1  } }\frac{\chi(n)}{\tau(n)} = \frac{t_0 \mathds 1 (\chi=\chi_0)}{
(\prod_{p\mid dq} t_p)}
\frac{x}{\sqrt{\log x } }
\l\{1+O\l(\frac{(\log \log 3dq)^{3/2}}{\log x}\r)\r\} 
+O\l(\frac{\tau(d)qx}{(\log x)^C}\r)\] holds for each Dirichlet character $\chi $ modulo $q$,
where $\chi_0$ is the principal character.  Using orthogonality of characters concludes the proof.
\end{proof} 
 For $\b d \in \Z \setminus\{\b 0 \} $ we define 
\beq{defgambd}
{\gamma(\b d )= 
  \prod_{p\neq 2  } 
\l(1-\frac{1}{p} \r)^{3/2}
\l(1+  \frac{\#\{1\leq i  \leq 3 : p\nmid d_i \}}{2 p} \r) 
 .}  
\begin{lemma}
\label{lem:1overdivmultiple} Assume that $a_1,a_2,a_3, d_1,d_2,d_3 $ are odd integers
and that $q_1,q_2,q_3\in \{4,8\}^3$.
Fix any $C>0$.
Then for all $\b X \in \R^3$ with $\min_i X_i\geq  \max_i |d_i| $ we have
 \begin{align*}
&\sum_{\substack{ \b n \in \prod_{i=1}^3 (\N \cap [1,X_i] )   \\ \gcd(n_i ,d_i )=1 \forall i \\n_i\equiv a_i \md {q_i} \forall i  } }\frac{\mu^2(n_1n_2n_3) }{\tau(n_1)\tau(n_2)\tau(n_3)}
 = 
\frac{\gamma(\b d )}{(2\pi )^{3/2} }
\l\{
\prod_{i=1}^3 \frac{X_i }{\phi(q_i)\sqrt{  \log X_i}   }  \r\}  
\\&+
O\l(\frac{X_1X_2X_3(\log \log 3\max  |d_i|)^{3/2}}{\l(\prod_{i=1}^3 (\log X_i)\r)^{1/2} (\log \min X_i )} 
+\frac{X_1X_2X_3\tau(d_1   )\tau(d_2   )\tau(d_3   )
}{(\log (3\max X_i)  )^{-2}(\log \min  X_i)^C} \r),\end{align*}
where  the implied constant depends at most on $C$.
\end{lemma}\begin{proof}  Fix any $C>4$. We employ Lemmas~\ref{lem:bwv225hcab}-\ref{lem:bwv225hcab676764} with 
\[g_i(n)=\frac{\mathds 1(\gcd(n, d_i)=1)}{\tau(n)}, z=w=(\log \min X_i )^{3C}.\] This shows that the sum over $\b n $ in our lemma 
equals 
\[\sum_{\substack{ k\in \N \cap [1,z]  } } \mu(k)\hspace{-0,2cm}
 \sum_{\substack{ \b n' \in (\N\cap[1,w])^3, p\mid n_1'n_2'n_3' \Rightarrow p\mid k \\
k^2\mid n_1'n_2'n_3' ,\ 
\gcd(n_i',2d_i )=1 \forall i } } 
 \prod_{i=1}^3\frac{1}{\tau( n_i')} \Bigg(
\sum_{\substack{  n'' \in \N \cap [1, X_i/n'_i],  \gcd(n'',kd_i)=1
  \\   n'' \equiv a_i/n'_i \md {q_i} }}\frac{1}{\tau(n'')}
  \Bigg)\]
up to an error term of size
\[\ll \frac{X_1X_2X_3}{z^{1/2} }(\log (3\max X_i)  )^3+\frac{X_1X_2X_3}{z^{1/3} } 
  \ll \frac{X_1X_2X_3(\log (3\max X_i)  )^3}{(\log \min X_i )^{C} }  .\]
We now   apply Lemma~\ref{lem:1overdiv} to each of the sums over $n''$. When estimating the sum for $i=3$
 we introduce an error term originating in $\tau(d) x (\log x )^{-C}$ of Lemma~\ref{lem:1overdiv}. This contribution 
is  \[\ll \sum_{k\in \N} \sum_{\substack{ \b n' \in \N^3, p\mid n_1'n_2'n_3' \Rightarrow p\mid k \\
k^2\mid n_1'n_2'n_3'  } } 
\frac{X_1X_2}{n_1'n_2'} \frac{\tau(d_3 k )X_3}{n_3' (\log X_3)^C}
\ll 
\frac{X_1X_2X_3\tau(d_3   )}{(\log \min_i X_i)^C}
\sum_{k\in \N}\frac{ \tau(k ) }{k^2} \sum_{\substack{ m \in \N, k^2\mid m  \\ p\mid m \Rightarrow p\mid k    } } \frac{\tau(m)\tau_3(m) }{m}
.\] This is satisfactory since one can use $\tau(k) \tau(m)\tau_3(m) \ll (km)^{1/4}$ and Lemma~\ref{lem:bwv2abcfe764} with $\epsilon=1/4$.
Similar arguments provide the same error term contribution for $i\in \{1,2\}$. This process leaves us with  the main term 
\begin{align*}
\sum_{\substack{ k\in \N \cap [1,z] \\2\nmid k  } } 
 &\mu(k) 
\sum_{\substack{ \b n' \in (\N\cap[1,w])^3, p\mid n_1'n_2'n_3' \Rightarrow p\mid k \\
k^2\mid n_1'n_2'n_3' ,\ 
\gcd(n_i',d_i )=1 \forall i } } 
\frac{1}{\tau( n_1')\tau( n_2')\tau( n_3')} \times 
\\
&  \times 
\prod_{i=1}^3
\Bigg( 
  \frac{t_0 }{\phi(q_i)
(\prod_{p\mid 2kd_i} t_p)}
\frac{X_i/n'_i}{\sqrt{\log (X_i/n_i')} }
\l\{1+O\l(\frac{(\log \log 3 k|d_i|)^{3/2}}{\log (X_i/n'_i)}\r)\r\} 
 \Bigg),\end{align*} where we used the fact that since each  $q_i $ is even, the conditions $\gcd(n_i',2)=1$ 
are equivalent to   $2\nmid k $.
Note that  when $n'_i \leq w= (\log X_i)^{3C}$ one has 
\[
\frac{1}{\log(X_i/n_i')} =\frac{1}{(\log X_i)} \frac{1}{1-\frac{\log n_i'}{\log X_i}}
=\frac{1}{(\log X_i)} \l(1+O_C\l(\frac{\log n_i'}{\log X_i}\r) \r)
\] and similarly 
\beq
{eq:tlrsqrtlog}{
\frac{1}{\sqrt{\log (X_i/n'_i)}  }= \frac{1}{\sqrt{\log X_i}  } \l( 1+O_C\l(\frac{\log n_i ' }{\log X_i }\r)\r)  .}
Hence,  rearranging the order of summation  gives 
\[ 
\frac{t_0^3}{t_2^3}
\l\{ \prod_{i=1}^3 
\frac{X_i}{ \phi(q_i)
(\prod_{p\mid  d_i  } t_p)
\sqrt{\log X_i}
} \r\}\sum_{\substack{ k\in \N \cap [1,z]  \\2\nmid k } }  
\frac{ \mu(k) \xi(k ) }{\prod_{i=1}^3\l(\prod_{p\mid k,p\nmid d_i}t_p\r)} 
,\] where $\xi(k )  $ is given by 
\[ 
\sum_{\substack{ \b n' \in (\N\cap[1,w])^3, p\mid n_1'n_2'n_3' \Rightarrow p\mid k \\
k^2\mid n_1'n_2'n_3' ,\ 
\gcd(n_i',d_i )=1 \forall i } } 
 \l\{\prod_{i=1}^3
\frac{1 }{ \tau( n_i') n_i'}  
\r\}
 \l\{1
+O\l(\frac{ E_1   +E_2 }{\log  \min_i X_i}
\r)
 \r\} 
,\] where $E_1= \log (n_1'n_2'n_3')$ and  $E_2=(\log \log (3 k |d_1d_2d_3| ))^{3/2} $.

The contribution of $E_1$ towards the sum over $k$ is 
\[
\ll   
\frac{1}{\log  \min_i X_i}
\sum_{ k\in \N   }  \sum_{\substack{ n\in \N, k^2 \mid n \\  p\mid n \Rightarrow p\mid k  }}
\frac{\tau_3(n) \log n }{n}  
, \] which is satisfactory as can be seen by using $\tau_3(n) \log n \ll n^{1/4}$ and Lemma~\ref{lem:bwv2abcfe764} with $\epsilon=1/4$.

When both $k,d\in \N$ are 
 large enough,
we have  $ \log (k d) \leq 2\log \max\{k, d\},$ hence, $$(\log \log (k d) )^{3/2} \ll
(\log \log k)^{3/2}+(\log \log d)^{3/2}
\ll k^{1/4}
(\log \log d)^{3/2}
 .$$ Thus, 
the contribution of $E_2$ 
  towards the sum over $k$ is 
\[\ll 
\frac{(\log \log |d_1d_2d_3| )^{3/2} }{\log  \min_i X_i}
 \sum_{k\in \N }
k^{1/4} 
\sum_{\substack{ n\in \N, k^2 \mid n \\  p\mid n \Rightarrow p\mid k  }} \frac{\tau_3(n)   }{n} 
 ,\] which is satisfactory due to 
 $\tau_3(n) \ll n ^{1/8} $ and taking $\epsilon=3/8$ in 
Lemma~\ref{lem:bwv2abcfe764}.

We are now  left with estimating
 \[ 
\frac{t_0^3}{t_2^3}
\l\{ \prod_{i=1}^3 
\frac{1}{ 
(\prod_{p\mid  d_i  } t_p)
 } \r\}\sum_{\substack{ 1\leq k \leq z  \\2\nmid k } }  
\frac{ \mu(k) 
 }{\prod_{i=1}^3\l(\prod_{p\mid k,p\nmid d_i}t_p\r)} 
\sum_{\substack{ \b n' \in (\N\cap[1,w])^3, p\mid n_1'n_2'n_3' \Rightarrow p\mid k \\
k^2\mid n_1'n_2'n_3' ,\ 
\gcd(n_i',d_i )=1 \forall i } } 
 \l\{\prod_{i=1}^3
\frac{1 }{ \tau( n_i') n_i'}  
\r\}
.\]
Ignoring the conditions $k\leq z $ and $\max n_i'\leq w $ can be done 
at the cost of negligible error terms using arguments similar to the ones 
in the proof of Lemmas~\ref{lem:bwv225hcab}-\ref{lem:bwv225hcab676764}.
We obtain  \[  \sum_{\substack{ k\in \N \\2\nmid k  }} 
\frac{\mu(k)}{\prod_{i=1}^3 \l( \prod_{p\mid k, p\nmid d_i }t_p\r) }
 \sum_{\substack{ m  \in \N , k^2\mid m    \\  p\mid m \Rightarrow p\mid k   } }
 \frac{1}{m } 
\sum_{\substack{ \b n' \in \N^3,  m=n_1'n_2'n_3' \\ \gcd(n'_i ,d_i )=1 \forall i
   } }   \prod_{i=1}^3    \frac{1 }{ \tau(n'_i)}.\]   
The   sum over $m$ equals $\prod_{p\mid k } c_p ,$ where $c_p$ is given by 
\[
 \sum_{j=2}^\infty \frac{1}{p^j } 
\sum_{\substack{
\b a \in (\Z\cap [0,\infty) )^3  , a_1+a_2+a_3=j
\\\gcd(p^{a_i}, d_i)=1  \forall i }
}   \prod_{i=1}^3 \frac{1}{(1+a_i) }
=
t_p^{\#\{1\leq i \leq 3: p\nmid d_i \}}
-1-\frac{\#\{1\leq i \leq 3: p\nmid d_i \}}{2p}
,\]  hence, the sum over $k$ equals $$ \prod_{p\neq 2  } \l(1-\frac{c_p}{t_p^{\#\{1\leq i \leq 3: p\nmid d_i \}}} \r)=
 \prod_{p\neq 2  } 
 \frac{1}{t_p^{\#\{1\leq i \leq 3: p\nmid d_i \}}}  
\l(1+  \frac{\#\{1\leq i  \leq 3 : p\nmid d_i \}}{2 p} \r). $$
Thus, we obtain  the main term 
\begin{align*}
&
\l\{ \prod_{p\neq 2  } 
 \frac{1}{t_p^{\#\{1\leq i \leq 3: p\nmid d_i \}}}  
\l(1+  \frac{\#\{1\leq i  \leq 3 : p\nmid d_i \}}{2 p} \r)\r\}
\frac{t_0^3}{t_2^3}
\frac{1}{\prod_{p} t_p^{\#\{1\leq i \leq 3: p\mid d_i \}}} 
= 
\frac{\gamma(\b d ) }{(2\pi )^{3/2} } 
,\end{align*}
where we used that each $d_i $ is odd.  This completes the proof.
\end{proof}

\subsection{The main term in Theorem~\ref{thm:genguo}}
\label{s:maintermthrm}
Recall the definition of  $\c M_{\b b , \b m }$ in Lemma~\ref{lem:detectlem4}.
\begin{lemma}\label{maintermlemfg7}Under the assumptions of Theorem~\ref{thm:genguo} we have  
\[
\c M_{\b b , \b m }(\b X)
=
\frac{\beta(\b b , \b m )  }{(2\pi )^{3/2} }
\frac{c(\b b , \b m )}{4 }
\l\{
\prod_{i=1}^3 \frac{X_i }{ \sqrt{  \log X_i}   }  \r\}  
 +O\l(\frac{X_1X_2X_3 (\log \log 3\c B)^{3/2}}{(  \log X_1)^{5/2} }\r) 
.\]
\end{lemma}
\begin{proof}
We first assume that
 $2\mid m_{23}$.  By~\eqref{eq:sn4ws} we see that   $n_1n_2n_3 m_{12}m_{13}$ in $\c M_{\b b , \b m }$ 
is odd, hence,  Lemma~\ref{lem:qq22} gives 
$$\c M_{\b b , \b m }(\b X)=\sum_{ \b a \in \c A } 
\sum_{\substack{\b n \in \N^3, \eqref{eq:2dnilb}, \eqref{eq:sn4ws}  \\
\b n \equiv \b a \mod 8 }} 
 \frac{\mu^2(n_1n_2n_3)}{\tau(n_1n_2n_3 )} 
,$$where   $\c A:= \l\{ \b a \in  ((\Z/8\Z)^*)^3 :  a_2    m_{13} =      a_3   m_{12}   \ \  \text{ or }  \  \ 
a_1    m_{23} +a_2    m_{13}     =     a_3   m_{12}\r\} $.
Note that the condition $\b n \equiv \b a \md 8$ 
implies that each $n_i $ is odd, thus, we can equivalently express~\eqref{eq:sn4ws}
as $\gcd(n_i,d_i)=1$, where  $d_i$ are respectively the odd parts of 
$$
m_{12} m_{13}  m_{23}    \gcd(   b_2,b_3), \,\,
 m_{12} m_{13}  m_{23}    \gcd(   b_1,b_3), \,\, m_{12} m_{13}  m_{23}    \gcd(   b_1,b_3) 
.$$ With these values of $d_i$ and with  $q_1=q_2=q_3=8$ 
we apply     Lemma~\ref{lem:1overdivmultiple} to infer that in the setting of Theorem~\ref{thm:genguo} one has 
\begin{align*}  
\c M_{\b b , \b m }(\b X)
&
=
\frac{\#\c A}{4^3}
\frac{\beta(\b b , \b m )  }{(2\pi )^{3/2} }
\l\{
\prod_{i=1}^3 \frac{X_i }{ \sqrt{  \log X_i}   }  \r\}  
 +O\l(\frac{X_1X_2X_3 (\log \log 3\c B)^{3/2}}{(\prod_{i=1}^3 \log X_i)^{1/2}\log \min X_i }\r) 
\\&
+O\l(\frac{X_1X_2X_3\tau(d_1   )\tau(d_2   )\tau(d_3   )
(\log (3\max X_i)  )^2}{(\log \min  X_i)^C} \r),\end{align*}
with an implied constant depending at most on $C$.
By~\eqref{eq:should23}
we can replace each $\min X_i$ and $\max X_i$ with $X_1$ by allowing the implied constants to depend on $\eta$.
We have $\tau(d_i) \leq d_i\leq (\max_i \log  X_i)^{5A} $ by~\eqref{eq:should}, 
hence, taking $C$ large enough compared to $A$  leads to an error term that agrees with the one claimed in Theorem~\ref{thm:genguo}.
The proof in the case $2\mid m_{23}$ concludes by noting that  $\c A$ has  $ 32$ elements.
The proof in the case $2\mid m_{12}m_{13}$ is   similar; in each such case  one obtains a set $\c A$ that has $ 32$ elements again, therefore, 
the resulting leading constants is the same. 

It remains to consider the case  $2\nmid m_{12}m_{13}m_{23}$. Using    that 
$n_1n_2n_3 m_{12}m_{13}m_{23}$ is square-free, we can write  $\c M_{\b b , \b m }(\b X) $
as $\sum_{i=0}^3 \c M_i$, where
\[\c M_0:=\sum_{\substack{\b n \in \N^3, \eqref{eq:2dnilb}, \eqref{eq:sn4ws}, 2\nmid n_1n_2n_3   \\
(n_1   n_3   m_{12}  m_{23}  ,n_2    n_3  m_{13} m_{12} )_2 =1 }} 
 \frac{\mu^2(n_1n_2n_3)}{\tau(n_1n_2n_3) } 
\] and \[
\c M_i
:=\sum_{\substack{\b n \in \N^3, \eqref{eq:2dnilb}, \eqref{eq:sn4ws}, 2\mid n_i   \\
(n_1   n_3   m_{12}  m_{23}  ,n_2    n_3  m_{13} m_{12} )_2 =1 }} 
 \frac{\mu^2(n_1n_2n_3)}{\tau(n_1n_2n_3/2) } 
, \ \ \ \ \ (1\leq i \leq 3 ) .\] 
Note that $\c M_1$ is non-zero only when $2\nmid \gcd(b_2,b_3)$ (analogously for $\c M_2$ and $\c M_3$).
Hence  $$\c M_1=\mathds 1(2\nmid \gcd(b_2,b_3) )
\sum_{\substack{\b n  \in \N^3, n_1 \leq X_1/2,n_2\leq X_2, n_3\leq X_3, 
\eqref{eq:sn4ws}, 2\nmid n_1n_2n_3  \\
(2n_1   n_3   m_{12}  m_{23}  ,n_2    n_3  m_{13} m_{12} )_2 =1 }} 
 \frac{\mu^2(n_1n_2n_3)}{\tau(n_1n_2n_3) } .$$ By Lemma~\ref{lem:qq22} the sum over $ \b n $ equals 
  \[ \sum_{\b a \in \c A' }
\sum_{\substack{  n_1 \leq X_1/2,n_2\leq X_2, n_3\leq X_3, \\
\eqref{eq:sn4ws},  
\b n \equiv \b a \mod 8 }} 
 \frac{\mu^2(n_1n_2n_3)}{\tau(n_1n_2n_3) }
,\] where $\c A':= \l\{ \b a \in  ((\Z/8\Z)^*)^3 :  a_2    m_{13} =      a_3   m_{12}   \ \  \text{ or }  \  \ 
2a_1    m_{23} +a_2    m_{13}     =     a_3   m_{12}\r\} .$
Invoking  Lemma~\ref{lem:1overdivmultiple} as above leads to  
\[
\frac{\#\c A'}{4^3}
\frac{\beta(\b b , \b m )  }{(2\pi )^{3/2} }
\l\{\frac{1}{2} 
\prod_{i=1}^3 \frac{X_i }{ \sqrt{  \log X_i}   }  \r\}  
 + O_{A,\eta}
  \l ( \frac{X_1 X_2 X_3 (\log \log \c B )^{3/2}}{ (\log  X_1)^{5/2} }\r ) .\]
One checks that $\#\c A'=32$. A similar argument deals with $\c M_2$ and $\c M_3$, resulting in
$$\sum_{i=1}^3\c M_i=
\frac{\#\{  i< j : 2\nmid (b_i,b_j)\}  }{4 }
\frac{\beta(\b b , \b m )  }{(2\pi )^{3/2} }
\l\{ 
\prod_{i=1}^3 \frac{X_i }{ \sqrt{  \log X_i}   }  \r\}  
 + O 
  \l ( \frac{X_1 X_2 X_3 (\log \log X_1 )^{3/2}}{ (\log  X_1)^{5/2} }\r ) 
.$$ Finally, following a facsimile argument   shows that  
$\c M_0$ equals
\[ 
\frac{\#\c A''}{2^3}
\frac{\beta(\b b , \b m ) }{(2\pi )^{3/2} }
\l\{ 
\prod_{i=1}^3 \frac{X_i }{ \sqrt{  \log X_i}   }  \r\}  
 + O   \l ( \frac{X_1 X_2 X_3 (\log \log \c B )^{3/2}}{ (\log  X_1)^{5/2} }\r ) ,\]
where 
$$\c A'':= \l\{ \b a \in  ((\Z/4\Z)^*)^3 :  
a_2    m_{13} =      a_3   m_{12}   \ \  \text{ or }  \  \ 
a_1    m_{23} =      a_3   m_{12}   \ \  \text{ or }  \  \ 
a_1    m_{23} +a_2    m_{13}   =0
\r\} .$$ This set has $6$ elements, an observation that concludes the proof of  Lemma~\ref{maintermlemfg7}.
\end{proof}

\subsection{The error term in Theorem~\ref{thm:genguo}: large conductors}
 \label{s:largecond} Recall the function $\c E_{\b b , \b m }(\b X)$ defined 
in Lemma~\ref{lem:detectlem4}. Our primary focus in this section will be to bound 
this function. The variables $\b d $ and $\b h $  may have an adversely large size for the subsequent analytic arguments, thus
we start by using  the (necessarily   small)  ``dual'' variables $\widetilde{\b d}, \widetilde{\b h }$.
In particular,   let  $  \sigma_{ij}=v_2( m_{ij} )$ and substitute 
  $m_{ij}=2^{\sigma_{ij} } h_{ij } \widetilde{h_{ij}}$ to
obtain
\begin{align*}
\l( \frac{2^{\sigma_2+\sigma_3}d_2 \widetilde{ d_2} d_3 \widetilde{ d_3}    m_{12}  m_{13}   }{d_1 h_{23} }\r) 
=& 
 \prod_{i=2,3} 
\l( \frac{ d_i   }{d_1   }\r) 
\l( \frac{ \widetilde{ d_i}    }{d_1   }\r) 
\l( \frac{ h_{1i }    }{d_1   }\r) 
\l( \frac{ \widetilde{ h_{1i }    }    }{d_1   }\r) 
\\
\times  
& \l( \frac{ 2^{\sigma_2+\sigma_3 + \sigma_{12}  +\sigma_{13}  } } {h_{23}    }\r)  
\prod_{i=2,3} 
\l( \frac{ d_i   }{h_{23}     }\r) 
\l( \frac{ \widetilde{ d_i}    }{h_{23}   }\r) 
\l( \frac{ h_{1i }    }{h_{23}    }\r) 
\l( \frac{ \widetilde{ h_{1i }    }    }{h_{23}    }\r) 
.\end{align*} We can similarly obtain analogous expressions for the other Jacobi symbols in the right hand side of the equation  
in Lemma~\ref{lem:detectlem4}.
Putting together these expressions will give various terms of the form $(\frac{n}{m})(\frac{m}{n})$
where $n, m$   are  odd positive integers. Using the  reciprocity law these terms can then be written as $(-1)^{\frac{(n-1)(m-1)  }{4}} $
and one will thus get the following expression for the 
product of the   Jacobi symbols in the definition of $\c E_{\b b , \b m }$:
\begin{align*}
(-1)^{\frac{G(\b h ) + G_{\b h }(\b d )  }{4}}
& \l(\frac{2}{d_1 h_{23} } \r)^{\sigma-\sigma_1-\sigma_{23} }
\l(\frac{2}{d_2 h_{13} } \r)^{\sigma-\sigma_2-\sigma_{13} }
\l(\frac{2}{d_3 h_{12} } \r)^{\sigma-\sigma_3-\sigma_{12} }
\\ 
&
\times \l( \frac{  \widetilde{d_2}   \widetilde{d_3}  \widetilde{ h_{12} }   \widetilde{  h_{13}  }  }{d_1 h_{23} }\r) 
\l( \frac{  \widetilde{d_1}   \widetilde{d_3}  \widetilde{ h_{12} }   \widetilde{  h_{23}  }  }{d_2 h_{13} }\r) 
\l( \frac{ - \widetilde{d_1}   \widetilde{d_2}  \widetilde{ h_{23} }   \widetilde{  h_{13}  }  }{d_3 h_{12} }\r) 
,\end{align*}    where 
$\sigma= \sigma_ 1+   \sigma_2 +   \sigma_3 +   \sigma_{12}  +   \sigma_ {13} +   \sigma_{23}   $ and 
\begin{align*}
G(\b h ) := &  (h_{12}-1) (h_{13}-1)  +   (h_{12}-1) (h_{23}-1)  +    (h_{13}-1) (h_{23}-1)  ,  
\\   
G_{\b h }(\b d )   := & 
 (d_1 -1 )(d_2 -1 ) +(d_1 -1 )(d_3 -1 ) +(d_2 -1 )(d_3 -1 )
\\+ &  (d_1 -1 )(h_{12}  -1 ) +(d_1 -1 )(h_{13}  -1 )  
\\+ &  (d_2 -1 )(h_{12}  -1 )   +(d_2 -1 )(h_{23}  -1 )
\\+ &  (d_3 -1 )(h_{13}  -1 )   +(d_3 -1 )(h_{23}  -1 ) . \end{align*}  
Injecting this into Lemma~\ref{lem:abmdbvl55}  and then into $\mathcal{E}_{  \b b, \b m}( \b X) $ gives the following result.
\begin{lemma}\label{lem:pirates} 
For  all $X _1,X_2,X_3\geq 2 $    and $\b b , \b m $ as in Theorem~\ref{thm:genguo},
the quantity
  $\mathcal{E}_{  \b b, \b m}( \b X) $ equals  \[   \sum_{\substack{(\sigma_i)\in \{0,1\}^3 \\ \eqref{eq:bwv228}}}
\hspace{-0,2cm}
  \sum_{\substack{ \b h , \widetilde{\b h} \in \N^3 \\ \forall i< j :h_{ij } \widetilde{h_{ij}}= (m_{ij})_{\rm{odd}} }} 
\hspace{-0,5cm}
  (-1)^{\frac{G(\b h )}{4} } \c {E}_{  \b b,\b h ,  \widetilde{\b h } }( \b X )    
\l(\frac{2}{  h_{23} } \r)^{\sigma-\sigma_1-\sigma_{23} } \l(\frac{2}{  h_{13} } \r)^{\sigma-\sigma_2-\sigma_{13} }
\l(\frac{2}{  h_{12} } \r)^{\sigma-\sigma_3-\sigma_{12} } \!\! ,\] where   
\begin{align*}\mathcal{E}_{  \b b, \b h ,\widetilde{\b h }  }( \b X )  =  \sum_{\substack{ \b d, \widetilde{\b d}  \in   \N^3 \\
  \eqref {eq:Blnbc93fdty}, 
 \eqref{eq:degh93vxz}
}}
&\frac
{
 (-1)^{\frac{G_{\b h }(\b d)  }{4} }
 }{
\tau(  d_1 \widetilde{d_1}   d_2  \widetilde{d_2} d_3 \widetilde{d_3}  )} 
 \l(\frac{2}{  d_1 } \r)^{\sigma-\sigma_1-\sigma_{23} } \l(\frac{2}{  d_2 } \r)^{\sigma-\sigma_2-\sigma_{13} }
\l(\frac{2}{  d_3 } \r)^{\sigma-\sigma_3-\sigma_{12} }
 \\  &\times
  \l(\frac{          \widetilde{d_2 }   \widetilde{d_3}  \widetilde{h_{12} }    \widetilde{ h_{13} }    }{d_1   h_{23}  }\r) 
    \l(\frac{             \widetilde{d_1}   \widetilde{d_3}  \widetilde{h_{12} }      \widetilde{h_{23} }   }{d_2   h_{13}  }\r) 
\l(\frac{    -     \widetilde{d_1}  \widetilde{d_2 }  \widetilde{h_{23} }  \widetilde{ h_{13} }    }{d_3  h_{12} }\r)  
,\end{align*}
 with    
  \begin{equation}\label {eq:Blnbc93fdty}
  d_1     \widetilde{d_1}     \leq  X_1/2^{\sigma_1}  ,
  d_2  \widetilde{d_2}    \leq   X_2/2^{\sigma_2}  ,
  d_3  \widetilde{d_3}    \leq   X_3/2^{\sigma_3}  ,\end{equation} and 
  \begin{equation}
\label{eq:degh93vxz}  
\begin{cases} 
d_1 d_2 d_3  h_{12}h_{13}h_{23} \neq 1, \,\, \widetilde d_1  \widetilde d_2 \widetilde d_3
\widetilde {h_{12}}\widetilde {h_{13}}\widetilde {h_{23}} \neq 1
\\
\mu^2( d_1 \widetilde{d_1}   d_2 \widetilde{d_2}     d_3 \widetilde{d_3}    )  =1,  \,\,
 2\nmid  d_1 \widetilde{d_1}   d_2 \widetilde{d_2}     d_3 \widetilde{d_3}   ,    \\
\gcd( d_1 \widetilde{d_1} d_2 \widetilde{d_2} d_3 \widetilde{d_3} , 
        h_{12}\widetilde{h_{12}} h_{13} \widetilde{h_{13}} h_{23} \widetilde{h_{23}})
=1,
\\   \gcd( d_1 \widetilde{d_1}  , b_2, b_3) 
= \gcd( 
 d_2 \widetilde{d_2}  , b_1, b_3)
= \gcd( 
  d_3 \widetilde{d_3}  
 , b_1, b_2) =1 . \end{cases}
 \end{equation} \end{lemma}\begin{remark}\label{remrk:strgh8}
 The sums $\c E_{\b b , \b h , \widetilde {\b h } } $ in Lemma~\ref{lem:pirates} 
may be thought as \[
\sum_{\substack{ \b d , \widetilde{\b d } \in \N^3   \\ \forall i:  d_i \widetilde {d_i } \leq X_i   }}
\l( \frac{\widetilde {d_2 }  \widetilde {d_3 } }{d_1 } \r) 
\l( \frac{\widetilde {d_1 } \widetilde {d_3 } }{d_2 } \r) 
\l( \frac{\widetilde {d_1 } \widetilde {d_2 }  }{d_3 } \r) , \] where our variables in the summation are square-free, odd and coprime in pairs.
If each  $d_i $ is  small, 
 one has  averages of quadratic characters  of small conductor. We will handle this       directly with Selberg--Delange estimates in \S\ref{s:smalcond};
the same approach applies if  each  $\widetilde{ d_i }$ is small instead.  The remaining cases occur when both 
$ d_i$ and $ \widetilde {d_j}  $  run over long intervals 
for some $i\neq j $. In this case the Jacobi symbol $(\frac{\widetilde {d_j} }{d_i})$
  in the sum above will give rise to   cancellation. To make this precise we use the following bound for bilinear sums
taken from the work of Friedlander--Iwaniec~\cite[Lemma 2]{MR2675875},
(see also  \cite{MR97362}, \cite[Corollary 4]{MR1347489} and \cite[\S 21]{MR1670065}): for all 
$M, N > 1 $ and   all sequences $a_m, b_n \in \mathbb C$ of modulus at most $1$   we have 
\beq{eq:bilinear}
{\sum_{\substack{ 1\leq m\leq M 
\\
1\leq n \leq N 
}} a_m b_n \mu(2 mn)^2 \left( \frac{m }{n } \right) 
\ll  (MN^{5/6}+M^{5/6} N) (\log  MN)^{2} ,}
where the implied constant is absolute.
\end{remark}

\begin{lemma}\label{lem:passacaglia_koopman2} Keep the setting of Lemma~\ref{lem:pirates} and  fix any $C> 0 $.
The contribution towards 
$\mathcal{E}_{\b b, \mathbf{h}, \widetilde{\mathbf{h}}}(\b X)$ of those $\b d,\widetilde{\b d}$ for which  
$ \min\{d_i, \widetilde{d_j} \}>\min_{i=1,2,3} (\log  X_i)^{100 C}$ holds  for  some    $i\neq j$   
is  bounded by  $O_C  ( X_1X_2X_3 (\log X_1)^{4-10C} ) $,
where the implied constant depends at most on $C$ and $\eta$.
\end{lemma}

\begin{proof}Let $Z=\min_{i=1,2,3} \log X_i$.
By symmetry we can assume $(i,j)=(2,1)$. Using 
$d_2>Z^{100 C}  $ and~\eqref {eq:Blnbc93fdty}
we see that 
$   \widetilde{d_2}   \leq X_2 Z^{-100 C} $. Similarly   one has 
$   d_1 \leq X_1 Z^{-100 C} $. Thus, the contribution is \[ 
 \ll    \sum_{\substack{     d_1   \leq X_1 Z^{-100 C}    \\ \widetilde{d_2}   \leq X_2 Z^{-100 C} \\  d_3    \widetilde{d_3} \leq X_3  }}  
   \left \vert  \sum_{\substack{      \widetilde{d_1}     \leq     X_1/(2^{\sigma_1  } d_1   ),   \\  
  d_2     \leq     X_2/( 2^{\sigma_2  }\widetilde{d_2}   ) }} 
 a(   \widetilde{d_1}  )  b(d_2)
\mu^2(  \widetilde{d_1}    )    \mu^2(2 \widetilde{d_1}   d_2     )
 \l(\frac{               \widetilde{d_1}     }{d_2     }\r) 
\r \vert 
, \] 
where $a, b $ are     products  of factors 
of the form  
  $(-1)^{G_{\b h }(\cdot ) /4 } , (2/\cdot )$, and $1/ \tau(\cdot ) $,
multiplied by  indicator functions  of the coprimality conditions inherited from~\eqref{eq:degh93vxz}.
By~\eqref{eq:bilinear} we get the bound 
\begin{align*} & \ll    \sum_{\substack{     d_1   \leq X_1 Z^{-100 C}    \\ \widetilde{d_2}   \leq X_2 Z^{-100 C} \\  d_3    \widetilde{d_3} \leq X_3  }}  
\l(  \frac{ X_1}{ d_1}      \l(  \frac{ X  _2   }{    \widetilde{d_2}  }      \r)^{\frac{5}{6} }  +\l( \frac{ X _1}{ d_1 }         \r)^{\frac{5}{6} }  \frac{ X_2}
{  \widetilde{d_2} }      \r) (\log  X_1X_2)^2\\&\ll X_1X_2X_3 (\log  X_1X_2X_3)^4 Z^{-100C/6}  ,\end{align*}
owing to  $\#\{d_3  ,  \widetilde{d_3}  \in \N:   d_3    \widetilde{d_3} \leq X_3 \} = \sum_{d \leq X_3} \tau(d) \ll  X_3 \log  X_3 $. 
This is sufficient due to~\eqref{eq:should23}.
\end{proof}

 \begin{lemma}
\label{lem:passacaglia bwv 582}Keep the setting of Lemma~\ref{lem:pirates} and  fix any $C> 0 $.
The contribution towards 
$\mathcal{E}_{\b b, \mathbf{h}, \widetilde{\mathbf{h}}}(\b X)$  
of those $\b d,\widetilde{\b d}$ 
for which $ \max\{ d_i, \widetilde{d_j}   \} \leq \min_{i=1,2,3} (\log  X_i)^{100 C}
$ holds  for  some    $i\neq j$    is  bounded by  
$O_C( X_1X_2X_3  (\log X_1 )^4 \min_{i=1,2,3}X_i^{-1/10})$,  
where the implied constant depends at most on $C$ and $\eta$. 
\end{lemma}
\begin{proof}
We can assume that $i=1$ and $j=2 $  by symmetry. Rearranging the contribution as $ \sum_{d_1, \widetilde{d_2}   , d_3, \widetilde{d_3}  }  | \sum_{   d_2, \widetilde{d_1}  } |$ and using arguments  identical to those in the proof of Lemma~\ref{lem:passacaglia_koopman2}     
is sufficient. \end{proof} 
We now give the outcome of the last two lemmas:
\begin{lemma}
\label{lem:K. 622 - Adagio } Keep the setting of Lemma~\ref{lem:pirates} and  fix any $C> 0 $.
We have  \[\mathcal{E}_{\b b, \mathbf{h}, \widetilde{\mathbf{h}}}(\b X)=\sum_{\lambda \in \{1,2\}}  \c {E}_{  \b b,\b h ,  \widetilde{\b h } }^{(\lambda)}( \b X ) 
+O_C \l ( X_1X_2X_3   (\log X_1 )^{4  -   10C  } \r ),\] where 
 \begin{align*}   \c {E}_{  \b b,\b h ,  \widetilde{\b h } }^{(\lambda)}( \b X ) &=
 \sum_{\substack{ \b d , \widetilde{\b d}   \in   \N^3  \\  \eqref {eq:Blnbc93fdty} ,   \eqref{eq:degh93vxz}, \eqref{bcvdfge} }} \frac {  (-1)^{\frac{G_{\b h }(\b d)  }{4} }  }{
\tau(  d_1 \widetilde{d_1}   d_2  \widetilde{d_2} d_3 \widetilde{d_3}  )} 
 \l(\frac{2}{  d_1 } \r)^{\sigma-\sigma_1-\sigma_{23} } \l(\frac{2}{  d_2 } \r)^{\sigma-\sigma_2-\sigma_{13} }
\l(\frac{2}{  d_3 } \r)^{\sigma-\sigma_3-\sigma_{12} }
 \\  &\times  \l(\frac{          \widetilde{d_2 }   \widetilde{d_3}  \widetilde{h_{12} }    \widetilde{ h_{13} }    }{d_1   h_{23}  }\r) 
    \l(\frac{             \widetilde{d_1}   \widetilde{d_3}  \widetilde{h_{12} }      \widetilde{h_{23} }   }{d_2   h_{13}  }\r) 
\l(\frac{    -     \widetilde{d_1}  \widetilde{d_2 }  \widetilde{h_{23} }  \widetilde{ h_{13} }    }{d_3  h_{12} }\r)  
,\end{align*} with \begin{equation}\label{bcvdfge}\begin{cases}   
\max\{d_i:1\leq i \leq 3\}   \leq \min_{i=1,2,3} (\log  X_i)^{100 C} < \min\{\widetilde{d_i }  :1\leq i \leq 3\} ,&\mbox{if } \lambda=1,\\  
\max\{\widetilde{d_i }  :1\leq i \leq 3\}   \leq \min_{i=1,2,3} (\log  X_i)^{100 C} < \min\{d_i:1\leq i \leq 3\} ,&\mbox{if } \lambda=2.
\end{cases}\end{equation}
The implied constant depends at most on $C$ and $\eta$. \end{lemma}  
\begin{proof}  Let $\c L=\min_{i=1,2,3} (\log  X_i)^{100 C}$.
The ranges covered by Lemmas    \ref{lem:passacaglia_koopman2}-\ref{lem:passacaglia bwv 582} are part of the error term of the present lemma. 
Out of the remaining terms we first consider the contribution of those with $d_1 \leq \c L$.
By
Lemma~\ref{lem:passacaglia bwv 582} one must have $\min\{ \widetilde{d_2 } , \widetilde{d_3 } \}>\c L$,
which, by Lemma~\ref{lem:passacaglia_koopman2} proves    $ \max\{d_2,d_3\}\leq \c L$ and therefore $ \widetilde{d_1 }   >\c L$ holds 
due to Lemma~\ref{lem:passacaglia bwv 582}. Hence, the condition $d_1\leq \c L $  is equivalent to the case $\lambda=1$ 
in~\eqref{bcvdfge}. A similar reasoning shows that the  condition  $d_1> \c L$  is equivalent to the case 
$\lambda=2$ in~\eqref{bcvdfge}. \end{proof}

\subsection{The error term in Theorem~\ref{thm:genguo}: small conductors}
\label{s:smalcond} We will now focus on bounding $\c E^{(\lambda)}$ that is defined in Lemma~\ref{lem:K. 622 - Adagio }.
As explained in Remark~\ref{remrk:strgh8}, this can be done by invoking  
 estimates about  averages of arithmetic functions twisted by 
a Dirichlet character of small conductor.
\begin{lemma}  \label{lem:friediwsiegwalfsiz} Fix any $C >0 $. Then for any non-zero $d, q,q_0,j \in \Z$ with $\gcd(q_0,jq)=1$,
any non-principal Dirichlet character $\chi \md{q}$, any function 
$f:\N\to \mathbb C$ with period $q_0$
 and any $X,Y\geq 2 $  we have 
 \[
\sum_{\substack{ Y\leq n \leq X  \\ \text{ \rm gcd}(n,d)=1   } }  \chi (n) f(n) \frac{  \mu(n)^2}{\tau(n) }
=O_C\l( 
\frac{ X }{(\log X)^C} \tau(d) q_0^2 q   \max_{n\in \N}\{|f(n)| \}  \r)  ,\] where the implied constant depends at most on $C$. \end{lemma}
\begin{proof} First note that it is sufficient to consider the case $Y=1$.
The periodicity of $f$ allows us to write the sum as 
\begin{align*}
&\sum_{j\md {q_0}} f(j) 
\sum_{\substack{ 1\leq n \leq X , n\equiv j \md{q_0}  \\ \gcd(n,d)=1   } }  \chi (n)  \frac{  \mu(n)^2}{\tau(n) } \\
&\ll  q_0 \max_{n\in \N}\{|f(n)| \}
\max_{j\md {q_0}}
\l|
\sum_{\substack{ 1\leq n \leq X , n\equiv j \md{q_0}  \\ \gcd(n,d)=1   } }  \chi (n)  \frac{  \mu(n)^2}{\tau(n) }
\r|.
\end{align*}
It therefore suffices to bound each sum in the modulus signs by $\ll_C X(\log X)^{-C} \tau(d) q_0 q $.
By orthogonality of Dirichlet  characters we can write each such  sum as 
\beq{eq:mgap}{ \frac{1}{\phi(q_0)}\sum_{\psi \md {q_0  } } \overline{\psi(j) }\sum_{\substack{ 1\leq n \leq X \\ \gcd(n,d)=1   } } (\psi \chi)(n)\frac{  \mu(n)^2 }{\tau(n) } .}
Let us now see why    $\psi \chi $  is a non-principal   character modulo $q_0 q$. 
Using that  $\chi \md q $ is non-principal we may  find  $ t \in \Z/q\Z$ with $\chi(t)  \notin \{0,1\} $.  Then by   $\gcd(q_0,q)=1$ 
we can   find an integer  $\ell  $ such that $\ell   \equiv 1 \md{ q_0 }$  and $\ell \equiv t \md q $. Since $\psi(1)=1$ we obtain 
 $(\psi \chi ) (\ell) = \chi(\ell  )=\chi(t)  \notin\{0,1\}$, thus $ \psi \chi $ is non-principal. Now let  $d_0$   be the square-free
product of all 
primes dividing $d$ and   coprime to $qq_0$ so that $\gcd(d_0,qq_0)=1$ and the condition 
$\gcd(n,d)=1$ in~\eqref{eq:mgap} can be replaced by $\gcd(n,d_0)=1$. Using~\cite[Lem. 1]{MR2675875}
then shows that   the sum in~\eqref{eq:mgap} is at most
  \[  \ll_C \frac{1}{\phi(q_0)} \sum_{\psi \md {q_0  } }     \frac{\tau(d_0) q_0 q  X }{(\log X)^C}   =  
   \frac{\tau(d_0) q_0 q  X }{(\log X)^C}    ,\]  which is sufficient due to $\tau(d_0 )\leq  \tau(d) $.  \end{proof}

\begin{lemma} \label     {lem:bvgt5iogln2}  
Keep the setting of Lemma~\ref{lem:K. 622 - Adagio } and  fix any $C'> 0 $.
We have 
 $$ \c {E}_{  \b b,\b h ,  \widetilde{\b h } }^{(1)}( \b X )
\ll
  b_1b_2b_3 \l(\prod_{i<j }h_{ij} \widetilde{   h_{ij}  }  \r)    \frac{X_1 X_2 X_3 }{ (\log X_1)^{C'}  }   ,$$
 where the implied constant depends at most on  $C, C'$ and $\eta$.  \end{lemma}
\begin{proof} By~\eqref{eq:degh93vxz}  
we know that 
$d_1 d_2 d_3 h_{12} h_{13} h_{23 } >1 $.
By symmetry we may assume that one of 
$d_1, d_2 ,h_{13} ,h_{23}  $ exceeds $1$.    
Thus, \[ \c {E}_{  \b b,\b h ,  \widetilde{\b h } }^{(1)}( \b X )\ll \sum_{\substack{ \widetilde{d_1} \leq X_1, \widetilde{d_2} \leq X _2\\ \forall i: d_i \leq (\log \min_i X_i)^{100C }}}  \mu^2(d_1   d_2   h_{13}  h_{23} ) 
\Bigg|
 \sum_{\substack{ (\log \min X_i)^{100C }  <   \widetilde{d_3}  \leq   X_3/b
  \\
\gcd(\widetilde{d_3} , d ) =1  }}
\frac{\mu^2(\widetilde{d_3} )}{\tau(   \widetilde{d_3}  )}  
  \l(\frac{        \widetilde{d_3}     }{q  }\r)   \Bigg|, \]where the outer  summation is subject to the condition $d_1d_2 h_{13} h_{23} >1$ and   
 $$q=d_1   d_2   h_{13}  h_{23}, \
\
d =2   d_3 \gcd(b_1,b_2)  \widetilde{d_1}\widetilde{d_2} \prod_{i<j }h_{ij} \widetilde{   h_{ij}  }   ,
\
\
b= d_3 2^{\sigma_3  }        .$$ 
The resulting sum over $\widetilde{d_3}  $  can   be estimated  via Lemma~\ref{lem:friediwsiegwalfsiz}
with $f(n)=1$ and  $q_0=1$. The character   $(\cdot/q)$  modulo $ q$  is    non-principal    due to  the fact that 
   $q$  is  an odd square-free integer  exceeding $1$. We obtain   the bound     \[ \ll_{C'}
 \sum_{\substack{ \widetilde{d_1} \leq X_1, \widetilde{d_2} \leq X _2\\ \forall i: d_i \leq (\log \min_i X_i)^{100C }}}  \hspace{-0,5cm}
 \frac{\tau(d) d_1d_2 h_{13} h_{23}   X_3 }{(\log X_3)^{C'}} 
\ll 
 \frac{  h_{13} h_{23}   X_1X_2 X_3 }{(\log X_3)^{C'}}
 \sum_{\substack{   \forall i: d_i \leq 
(\log \min_i X_i)^{100C }}} \tau(d) d_1d_2   \] for every fixed $C'>0$. 
Using  $\tau(mn) \leq \tau(m) \tau(n) $, which holds for all $m , n \in \N$,
we obtain  \[\tau(d) \leq    d_3 \tau( \widetilde{d_1})\tau( \widetilde{d_2} )b_1b_2\prod_{i<j }h_{ij} \widetilde{   h_{ij}  }   .\]
This gives  \[ \sum_{\substack{ \widetilde{d_1} \leq X_1, \widetilde{d_2} \leq X _2\\ \forall i: d_i \leq (\log \min_i X_i)^{100C }} } \hspace{-0,5cm}\tau(d)  
\ll     b_1b_2\l(\prod_{i<j }h_{ij} \widetilde{   h_{ij}  }  \r) X_1 X_2 (\log X_1 X_2 X_3 )^{1000 C +2 } .\] Using~\eqref{eq:should23}
and taking $C'$ large enough compared to $C$ concludes the proof.
 \end{proof}

\begin{lemma}
\label{lem:bvchhar29}  
We have 
 $$ \c {E}_{  \b b,\b h ,  \widetilde{\b h } }^{(2)}( \b X )
\ll
  b_1b_2b_3 \l(\prod_{i<j }h_{ij} \widetilde{   h_{ij}  }  \r)    \frac{X_1 X_2 X_3 }{ (\log X_1)^{C'}  }   ,$$
 where the implied constant depends at most on  $C, C'$ and $\eta$. 
\end{lemma} \begin{proof}  We  first focus on  
  the terms satisfying  $ \widetilde{d_1    }\widetilde{ d_2 }  \widetilde{  h_{13} }\widetilde{  h_{23 }} >1 $.
 We
rearrange their contribution as  $  \sum_{d_1, d_2,  \widetilde{d_1} , \widetilde{d_2}  ,  \widetilde{d_3}  }  | \sum_{ d_3   } |$.
Here the outer sum includes the condition that  $ 2 d_1  d_2   \widetilde{d_1}   \widetilde{d_2}     \widetilde{d_3}$  is square-free and the
  inner sum is given by \[  
 \sum_{\substack{ (\log \min_i X_i)^{100C }  <  d_3 \leq   X_3/b
  \\
\gcd( d_3 , d ) =1  }}
f_0(d_3)
\frac{\mu^2(d_3  )}{\tau( d_3 )}  
\l(\frac{  q  }{d_3  }\r) 
,\]  where $b=2^{\sigma_3} \widetilde{d_3}$,  $ d $ is similar to the analogous one  in the proof of Lemma~\ref {lem:bvgt5iogln2},
\[f_0(d_3):=
(-1)^{\frac{(d_3 -1 ) m  }{4}}   \l(\frac{2}{d_3 }  \r) ^ \beta 
 \l(\frac{-1}{d_3 }  \r)  
,\]     $\beta=\sigma-\sigma_3-\sigma_{12}\in \{0,1\}$,  
$q=  \widetilde{d_1    }\widetilde{ d_2 }  \widetilde{  h_{13} }\widetilde{  h_{23 }}   $ is a positive odd square-free integer with $q >1 $
and
 $m =d_1+d_2+h_{13}+h_{23} $ is an even integer. 
 By  quadratic reciprocity we can  write 
$$  \l (\frac{q}{d_3 } \r)   =  (-1)^{\frac{ (q-1)(d_3-1)}{4} } \l (\frac{d_3 }{q} \r) ,$$ hence, 
the function $f(d_3)=f_0(d_3) (-1)^{\frac{ (q-1)(d_3-1)}{4} } $ has period $q_0=8$ and the sum equals
 \[  
 \sum_{\substack{ (\log \min_i X_i)^{100C }  <  d_3 \leq   X_3/b
  \\
\gcd( d_3 , d ) =1  }}
f (d_3)
\frac{\mu^2(d_3  )}{\tau( d_3 )}  
\l(\frac{d_3  }{  q  }\r) 
.\] Since $q$ is square-free and satisfies $q>1$, 
one can employ  Lemma~\ref{lem:friediwsiegwalfsiz}  
to bound the inner sum.  Then  the proof can be completed as  the one of   Lemma~\ref {lem:bvgt5iogln2}.
 The remaining cases  $  \widetilde{  d_3  }\widetilde{ h_{12} } >1 $ can be dealt with similarly by 
rearranging the sums as $  \sum_{d_1, d_3,  \widetilde{d_1} , \widetilde{d_2}  ,  \widetilde{d_3}  }  | \sum_{ d_2   } |$.   \end{proof}

\subsection{Final steps in the proof of Theorem~\ref{thm:genguo}} \label{sec:finstep}
Injecting Lemmas~\ref{lem:bvgt5iogln2}-\ref{lem:bvchhar29} into Lemma~\ref{lem:K. 622 - Adagio }  shows that for every fixed 
$N>0$ one has  \[ \mathcal{E}_{\b b, \mathbf{h}, \widetilde{\mathbf{h}}}(\b X)
\ll   b_1b_2b_3  X_1 X_2 X_3  (\log X_1)^{-N} \prod_{i<j }h_{ij} \widetilde{   h_{ij}  }  ,\] where the implied constant depends at most on  $\eta$ and $N$.
Recalling that $h_{ij} \widetilde{   h_{ij}  }\leq m_{ij}$, using~\eqref{eq:should},
and taking $N$ large enough compared to $A$, shows that  for every fixed $M>0$ one has 
 \[ \mathcal{E}_{\b b, \mathbf{h}, \widetilde{\mathbf{h}}}(\b X) \ll   X_1 X_2 X_3  (\log X_1)^{-M} .\] When this is fed into Lemma~\ref{lem:pirates} 
it yields 
$ \mathcal{E}_{\b b, \b m }(\b X) \ll X_1 X_2 X_3 (\log X_1)^{-M}$.
Combining this with Lemma~\ref{lem:detectlem4} 
completes the proof by alluding to 
Lemma~\ref{maintermlemfg7}.
  
 \subsection{The proof of Theorem~\ref{thm:main}}  \label{sec:proofmain}
Recall the definition \eqref{def:nalphax} of $N(B)$. Our first step  is to introduce new variables which account for square factors and common factors of the coefficients $t_i$. Recall~\eqref{def:nbmx}.
 \begin{lemma} 
\label{lem:hool1}
Fix any $C>0$. Then for all $B\geq 2   $   
the quantity $N(B) $ equals \begin{align*}
6  \sum_{\substack{\b b \in  \N ^3, \, \gcd(b_1,b_2,b_3)=1   \\ \forall i: b_i \leq (\log B)^C }}
\ \   & \sum_{\substack{
 (m_{12},m_{13},m_{23})\in \N^3, \, \eqref{eq:delbd}
\\ \forall i\neq j:  m_{ij} \leq  (\log  B)^C   }}   \mu^2(m_{12} m_{13} m_{23})
\\ &\times
 \mathcal{N}_{ \b b, \b m}\l(  \frac{B}{ b_1^2 m_{12}m_{13}},
\frac{B}{ b_2^2  m_{12}m_{23} } , 
\frac{B}{ b_3^2  m_{13}m_{23} } \r)  +O_C\l(
\frac{B^3 }{(\log B)^C}
\r),\end{align*}
where  the implied constant depends only on $C$ 
and the conditions in the summation are
\begin{equation} \label{eq:delbd}   \gcd(m_{12},b_3  )= \gcd(m_{13},b_2   )= \gcd(m_{23},b_1   )=1.
 \end{equation}   
 \end{lemma} 
\begin{proof}Imposing the condition $(t_1-t_2) (t_1-t_3)(t_2-t_3)t_1t_2t_3 \neq 0 $ in $N(B)$ introduces an error of size $O(B^2)$.
 For the remaining 
$t_i$
having a point in $\R$ is equivalent to not all $t_i$ having the same sign. Hence, by permuting variables, we obtain 
 \[\frac{N(B)}{6} +O(B^2)=  
\#\bigg\{\b t \in (\N\cap[1,B])^3:  \gcd(t_0,t_1,t_2)=1, 
\sum_{i=0,1} t_i X_i^2=t_2X_2^2 \textrm{ has a }\Q\textrm{-point}
\bigg\}
.\]  Each integer $t_i$ can be written uniquely as $b_{i+1}^2 c_{i+1}$ for some $b_{i+1},c_{i+1} \in \N$ with  $\mu^2(c_{i+1})=1$.
Since the solubility of the equation is independent of square factors of the coefficients $t_i$, the quantity in the 
right-hand side becomes 
\[
 \sum_{\substack{\b b \in (\N\cap[1,\sqrt{B}]^3\\ \gcd(b_1,b_2,b_3)=1}}
\sum_{\substack{\b c \in \N^3, c_i \leq B/b_i^2 \\
 \gcd(c_1b_1,c_2b_2,c_3b_3) =1
}}
\mu^2(c_1)\mu^2(c_2)\mu^2(c_3) 
\begin{cases} 1, \text{ if } c_1 x_1^2+c_2x_2^2=c_3x_3^2  \text{ has a }\QQ\text{-point},\\ 0, \text{ otherwise.} \end{cases}
\]
Now we restrict the range of $\b b$. Ignoring solubility over $\Q$ and $\gcd$ conditions, we may
bound the   contribution of the terms with at least one of $b_j$ exceeding $(\log B)^C$ by
\[
\sum_{\substack{\b b \in \N^3\\ b_j >(\log B)^C}}
\prod_{i=1}^3 \frac{X}{b_i^2}
\ll \frac{B^3}{(\log B)^C}. 
\]
We now introduce new variables to keep track of common factors between the $c_i$. Define $m_{ij} = \gcd(c_i, c_j)$
so that there exists $\b n \in \N^3$ such that $c_1 = m_{12}m_{13}n_1, 
c_2 = m_{12}m_{23}n_2$ and $ c_3=m_{13}m_{23}n_3$.
By construction it is evident that any two of the $6$ variables $n_i, m_{ij}$  are coprime, thus the sum over $\b c $ is\[
\sum_{\substack{
 (m_{12},m_{13},m_{23})\in \N^3, \eqref{eq:delbd}
\\ \forall i\neq j:  m_{ij} \leq  B  }}   \mu^2(m_{12} m_{13} m_{23})
 \sum_{\substack{\b n\in \N^3,  \eqref{eq:sn4ws}   
\\ \{i,j,k\}=\{1,2,3\}\Rightarrow   m_{ij}m_{ik}  n_i \leq  B/b_i^2  }}   \mu^2(n_1n_2n_3)F(\b n )
,\] where $F(\b n )$ is the indicator function of the event that the following curve has a a $\Q$-point:
$$ m_{12}m_{13}n_1x_1^2+m_{12}m_{23}n_2x_2^2=m_{13}m_{23}n_3x_3^2.$$ The     change of variables 
$
Y_1=m_{12}  m_{13}x_1 ,
Y_2= m_{12} m_{23} x_2 , 
Y_3=  m_{13} m_{23} x_3 
$ is invertible over $\QQ$ (and all its completions) and it transforms  the curve
into  $$n_1       m_{23}   Y_1^2+ n_2     m_{13}    Y_2^2=n_3  m_{12}   Y_3^2.$$
In particular, the sum over $\b n $ equals $$ \mathcal{N}_{ \b b, \b m}\l(  
\frac{B}{ b_1^2 m_{12}m_{13}},
\frac{B}{ b_2^2  m_{12}m_{23} } , 
\frac{B}{ b_3^2  m_{13}m_{23} } \r).$$ Using the bound 
$\mathcal{N}_{ \b b, \b m}(\b B)\leq B_1B_2B_3$ we see that 
the contribution of the cases $m_{12} > (\log X)^C$ is  
\[ \ll  \sum_{\b b \in \N^3}
\sum_{\substack{ (m_{12} ,m_{13} ,m_{23} ) \in \N^3 \\  m_{12} > (\log X)^C }} 
\frac{X^3}{(b_1 b_2 b_3 )^2(m_{12} m_{13} m_{23})^3}
\ll \frac{X^3}{(\log X)^C} \]   and a similar argument deals with the terms satisfying $\max\{ m_{13},m_{23} \}> (\log X)^C$.  \end{proof}Recall~\eqref{defbetabm}
and note that ignoring the condition $p\nmid  m_{12} m_{13}  m_{23} \gcd(b_i,b_j) $ yields 
\beq{eq:bndbcm}{
\beta(\b b,\b m)
\leq
\prod_{p  } 
\l(1-\frac{1}{p} \r)^{3/2}
\l(1+  \frac{3}{2 p} \r)  
.}  
 \begin{lemma} 
\label{lem:hog97ro}
For $B\geq 2   $    we have 
$$N(B)= \frac{ 3  c_0}{ (2\pi )^{ 3/2}} \frac{B^3}{(\log B)^{3/2}} +O\l(\frac{B^3}{(\log B)^{5/2}}\r) ,$$ where 
$$
c_0=\sum_{\substack{\b b \in  \N ^3\\\gcd(b_1,b_2,b_3)=1    }}
  \sum_{\substack{\b m \in \N^3\\ \eqref{eq:delbd}
   }}    
  \frac{ \mu^2(m_{12} m_{13} m_{23}) }{ (b_1b_2b_3  m_{12}m_{13}m_{23})^2}
 \frac{     \beta(\b b , \b m ) c(\b b , \b m )  }{    \tau(( m_{12}    m_{13} m_{23})_{\rm{odd}})   }
 .$$
\end{lemma}
\begin{proof}We shall apply Theorem~\ref{thm:genguo} with 
$$
X_1=\frac{B}{ b_1^2 m_{12}m_{13}}, \quad
X_2=\frac{B}{ b_2^2  m_{12}m_{23} } , \quad
X_3=\frac{B}{ b_3^2  m_{13}m_{23} }
$$  and with $\max\{b_i,m_{ij}\}\leq (\log B)^C$, where $C$ is the constant from Lemma~\ref{lem:hool1}
that furthermore satisfies $C>5/2$. 
Therefore,~\eqref{eq:should} is satisfied with $A=C$ and~\eqref{eq:should23} holds with $\eta=1/2$
due to  
$ B^{1/2}\ll_C \min X_i \leq \max X_i \leq B $.
Thus,  
$$N(B) B^{-3} =6   N_1(B) + O(E(B)+(\log B)^{-5/2}),$$ 
where 
 \begin{align*}
N_1(B) :=
&\frac{ (2\pi )^{-3/2}}{2}
\Osum_{\b b,\b m }
  \frac{ \mu^2(m_{12} m_{13} m_{23}) }{ (  m_{12}m_{13}m_{23})^2}
 \frac{     \beta(\b b , \b m ) c(\b b , \b m )  }{    \tau(( m_{12}    m_{13} m_{23})_{\rm{odd}})   }
  \prod_{i=1}^3  \frac{b_i^{-2}}{ (\log X_i)^{1/2} }
,  \\
E(B)  :=
& \sum_{\substack{\b b , \b m \in  \N ^3  }}
  \frac{(\log X_1)^{-5/2}}{ (b_1b_2b_3 m_{12}m_{13}m_{23})^2}
\ll (\log X_1)^{-5/2}\ll (\log (B^{1/2}))^{-5/2}
\end{align*} 
and the conditions
in $\Osum$ are  as in 
Lemma~\ref{lem:hool1}.  
Furthermore, as in~\eqref{eq:tlrsqrtlog}
we can   write 
\[\prod_{i=1}^3 (\log X_i)^{-1/2}=
(\log B)^{-3/2} \l(1+O\l(\frac{\log (b_1b_2b_3 m_{12} m_{13} m_{23} )}{\log B}\r)\r)
.\] Using this along with~\eqref{eq:bndbcm} yields 
  \begin{align*}
N_1(B) =
&\frac{ (2\pi )^{-3/2}}{2(\log B)^{3/2}}
\Osum_{\b b,\b m }
  \frac{ \mu^2(m_{12} m_{13} m_{23}) }{ (  b_1b_2b_3m_{12}m_{13}m_{23})^2}
 \frac{     \beta(\b b , \b m ) c(\b b , \b m )  }{    \tau(( m_{12}    m_{13} m_{23})_{\rm{odd}})   }
   \\
+& \,O\l((\log B )^{-5/2}\sum_{\b b,\b m \in \N^3 }
  \frac{\log (b_1b_2b_3 m_{12} m_{13} m_{23} ) }{ ( b_1b_2b_3  m_{12}m_{13}m_{23})^2}
\r) .\end{align*}
The error term is satisfactory since  the sum is
$\sum_{n\in \N} \tau_6(n) (\log n) n^{-2} =O(1)$, where $\tau_6$ denotes the number of different decompositions 
as a product of $6$ positive integers. The main term sum $\Osum$ contains the condition $\max\{b_i,m_{ij}\} \leq (\log X)^C$
that can be safely ignored by using~\eqref{eq:bndbcm}.
\end{proof}
We now continue by factoring $c_0$ as an Euler product. This otherwise straightforward task is somewhat hampered by the fact that 
neither of   $ \beta(\b b , \b m ), c(\b b , \b m ) $ has standard  multiplicative properties.
 \begin{lemma} 
\label{lem:hofactorout2o}
We have 
$$
c_0=\frac{49}{3}
 \sum_{\substack{\b b \in  \N ^3, 2\nmid  b_1 b_2 b_3\\  \gcd(b_1,b_2,b_3)=1   }}
 \frac{1}{ (b_1b_2b_3)^2    } \hspace{-0,3cm}
  \sum_{\substack{\b m\in \N^3, 2\nmid m_{12} m_{13} m_{23} \\  
 \gcd(m_{12},b_3  )= \gcd(m_{13},b_2 )= \gcd(m_{23},b_1  )=1
   }} 
 \hspace{-0,3cm}
 \frac{ \mu^2(m_{12} 'm_{13}' m_{23}')}{  (m_{12}' m_{13}' m_{23}')^2  }
 \frac{     \beta(\b b' , \b m' )   }{    \tau(  m_{12}'    m_{13}' m_{23}' )   }
.$$
\end{lemma}
\begin{proof}
Write  $ b_i=2^{\beta_i} b_i'$ and $m_{ij}=2^{\mu_{ij} } m_{ij}'$, where each $b_i',  m_{ij}'$ is   odd. 
Note that $ \beta(\b b , \b m ) $ only depends on the odd parts of   $b_i,m_{ij}$
and that $ c(\b b , \b m )= c( (2^{\beta_i }) , (2^{ \mu_{ij} } ) )$.
Hence, $c_0=\gamma c_0'$, where  
\[\gamma=
 \sum_{\substack{\beta_1,\beta_2,\beta_3 \geq 0  \\  \min \beta_i=0   }}
 \frac{1}{ 4^{\beta_1+ \beta_2 + \beta_3 }   }
 \sum_{\substack{
0\leq 
 \mu_{12}+  \mu_{13} + \mu_{23}\leq 1 
\\
 \min(\mu_{12},\beta_3  )= \min(\mu_{13},\beta_2   )= \min(\mu_{23},\beta_1   )=0
   }}  
 \frac{    c( (2^{\beta_i }) , (2^{ \mu_{ij} } ) )   }{ 4^{\mu_{12} +\mu_{13} +\mu_{23} }   }
 \] and $c_0'$ is the sum over $\b b, \b m $ given in the present lemma. 
If one of  $ \mu_{ij}$ is  $1$, then $ c( (2^{\beta_i }) , (2^{ \mu_{ij} } ) )=2 $, hence, the contribution equals 
\[ 3
 \sum_{\substack{\beta_1,\beta_2  \geq 0  \\  \beta_3=0   }}
 4^{-\beta_1- \beta_2 - \beta_3 }   
\cdot   \frac{2}{4}   = \frac{8  }{  3 }.  \] 
If all    $ \mu_{ij}$ are  $0$ and all $\beta_i$ are $0$, then $ c( (2^{\beta_i }) , (2^{ \mu_{ij} } ) )=6 $, hence, the contribution equals  $6$.
If all    $ \mu_{ij}$ are  $0$ and exactly one of  $  \beta_i   $ is  $\geq 1 $, 
then $ c( (2^{\beta_i }) , (2^{ \mu_{ij} } ) )=6 $, hence, the contribution equals  
\[  3  \sum_{\substack{\beta_1  \geq 1    }}  \frac{1}{ 4^{\beta_1  }   }\cdot  6= 6 . \] 
If all    $ \mu_{ij}$ are  $0$ and exactly two  of  $  \beta_i   $  are  $\geq 1 $, 
then $ c( (2^{\beta_i }) , (2^{ \mu_{ij} } ) )=5 $, hence, the contribution equals  
\[3   \sum_{\substack{\beta_1\geq 1 , \beta_2\geq 1   }} \frac{1}{ 4^{\beta_1+ \beta_2   }   } \cdot  5=\frac{5}{3} .  \] 
Adding the various contributions shows that 
$\gamma=8/3+6+6+5/3=49/3$. 
\end{proof}
Recall~\eqref{defgambd} and let $$\kappa =\prod_{p\neq 2 } 
 \l(1-\frac{1}{p} \r)^{3/2}
\l(1+  \frac{ 3  }{2 p} \r) .$$
A function $f:\N^3\to\mathbb C$ is called multiplicative if 
$
f(a_1b_1 , a_2b_2, a_3b_3 )
=f(\b a) f(\b b )
$ whenever $ a_1a_2a_3$ and $b_1b_2b_3$ are coprime.
 \begin{lemma} 
\label{lem:hock4ba3} For $\b d \in \N^3$
the  function $f(   \b d )= \gamma(\b d )/\kappa$ is    multiplicative and given by 
$$
f(\b d ) =
  \prod_{\substack{ p\neq 2 \\ p\mid d_1d_2d_3 } } 
\l(1-
\frac
{ \#\{1\leq i  \leq 3 : p\mid d_i \}  }
{ 2p +  3  }
\r) 
.$$ 
\end{lemma}
\begin{proof} 
 If $p\nmid d_1d_2d_3 $ then the $p$-th terms of $\gamma(\b d )$  and $\kappa$ coincide, hence, 
$\gamma(\b d )/\kappa$  equals  
\[
 \prod_{\substack{ p\neq 2 \\ p\mid d_1d_2d_3 } } 
\frac
{\l(1+  \#\{1\leq i  \leq 3 : p\nmid d_i \}/(2 p) \r)  
}
{\l(1+ 3/(2p) \r)  
} 
=
  \prod_{\substack{ p\neq 2 \\ p\mid d_1d_2d_3 } } 
\frac
{\l(2p +  \#\{1\leq i  \leq 3 : p\nmid d_i \}  \r)  
}
{\l(2p +  3 \r)  
}.
 \]
  To prove multiplicativity, 
 note that  if  $\gcd(a_1a_2a_3,b_1b_2b_3)=1$ then $f(a_1b_1, a_2b_2, a_3b_3 )$ splits as  
 \[  \prod_{\substack{ p\neq 2 \\ p\mid a_1a_2a_3   } } 
\l(1-
\frac
{ \#\{1\leq i  \leq 3 : p\mid a_ib_i \}  }
{ 2p +  3  }
\r)   \prod_{\substack{ p\neq 2 \\ p\mid   b_1 b_2 b_3 } } 
\l(1-
\frac
{ \#\{1\leq i  \leq 3 : p\mid a_ib_i \}  }
{ 2p +  3  }
\r) 
.\] 
By coprimality the condition $p\mid a_ib_i$ in the first product is equivalent to $p\mid a_i$, thus, it equals $f(\b a )$.
Similarly, the second product equals  
  $  f(\b b )$. 
\end{proof}

 \begin{lemma} 
\label{lem:eflatbwv852}
The sum over $\b b , \b m $ in Lemma~\ref{lem:hofactorout2o}
equals 
$$
\prod_{p\neq 2 } 
\l(1-\frac{1}{p} \r)^{3/2} 
\frac{(p^2+p+1)(2p^2+p+2)}{2(p^2-1)^2}
.$$
 \end{lemma}
\begin{proof}
By Lemma~\ref{lem:hock4ba3} 
one directly sees that  
  the function $g:\N^6\to \mathbb C$
given by 
$$ g( \b b, \b m )= f(   m_{12} m_{13} m_{23} \gcd(b_2,b_3),
 m_{12} m_{13} m_{23} \gcd(b_1,b_3),
m_{12} m_{13} m_{23} \gcd(b_1,b_2)
)$$ is multiplicative.
Thus, the sum over $\b b , \b m $ in Lemma~\ref{lem:hofactorout2o}
is  $\kappa \prod_{p\neq 2 } \kappa_p'$, where  
 \[\kappa_p'=
 \sum_{\substack{\beta_1,\beta_2,\beta_3 \geq 0  \\  \min \beta_i =0   }}
 \sum_{\substack{ 
 0\leq \mu_{12}+\mu_{13} +\mu_{23}\leq 1 
\\
\mu_{12} \beta_3  =
\mu_{13} \beta_2  = 
\mu_{23} \beta_1  = 0   }}
 \frac{  
g(p^{\beta_1}, p^{\beta_2}, p^{\beta_3}, p^{\mu_{12}},p^{\mu_{13}},p^{\mu_{23}} )
  }{     p^{2(\beta_1+\beta_2+\beta_3+\mu_{12}+\mu_{13} +\mu_{23})} 
\tau( p^{ \mu_{12}+\mu_{13} +\mu_{23} } )   }
. \]  
The terms with $\mu_{12}+\mu_{13} +\mu_{23}= 1 $ contribute    \[3 
 \sum_{\substack{\beta_1,\beta_2  \geq 0  }}
 \frac{  
g(p^{\beta_1}, p^{\beta_2}, 1, p ,1,1 )
 }{     p^{2(\beta_1+\beta_2 +1)} 
\tau( p )   }=
\frac{3 f(  p ,p,p  )   }{2p^2}  
 \sum_{\substack{\beta_1,\beta_2  \geq 0  }}  \frac{ 1 }{     p^{2(\beta_1+\beta_2  )}  }
 =\frac{3 (1-\frac{3}{3+2p} ) }{2p^2(1-1/p^2)^2}  
.\]
The terms with $\mu_{12}+\mu_{13} +\mu_{23}= 0 $ contribute  
\[
 \sum_{\substack{\beta_1,\beta_2,\beta_3 \geq 0  \\  \min \beta_i =0
   }}
 \frac{  
g(p^{\beta_1}, p^{\beta_2}, p^{\beta_3}, 1,1,1)
  }{     p^{2(\beta_1+\beta_2+\beta_3 )}   }
= \sum_{\substack{\beta_1,\beta_2,\beta_3 \geq 0  \\  \min \beta_i =0 }}
 \frac{  f(p^{\min\{\beta_1,\beta_2\}},
p^{\min\{\beta_1,\beta_3\}},
p^{\min\{\beta_2,\beta_3\}}  ) }{     p^{2(\beta_1+\beta_2+\beta_3 )}   }
 .\] 
 The terms  where all $\beta_i$ are zero contribute  $ 1$.
The terms  where exactly  one $\beta_i$ is non-zero give  
\[3   \sum_{\substack{\beta  \geq 1  }}
 \frac{  1}{     p^{2\beta}    }=\frac{3}{p^2-1}
.\]
The terms  where exactly two  $\beta_i$ are non-zero contribute   \[ 3f(p,1,1) 
 \sum_{\substack{\beta_1,\beta_2  \geq 1    }}
 \frac{  1 }{     p^{2(\beta_1+\beta_2  )}    }
=\frac{3f(p,1,1)  }{(p^2-1)^2}
=
\frac{3(1-\frac{1}{2p+3} )  }{(p^2-1)^2}
.\] In total the sum over $\b b, \b m $ becomes 
 \begin{align*}
&
\kappa\prod_{p\neq 2 }  \l(
\frac{3 (1-\frac{3}{3+2p} ) }{2p^2(1-1/p^2)^2}  
+1+\frac{3}{p^2-1}
+ \frac{3(1-\frac{1}{2p+3} )  }{(p^2-1)^2}
 \r) 
\\
=&
\prod_{p\neq 2 } 
\l(1-\frac{1}{p} \r)^{3/2} 
\frac{(p^2+p+1)(2p^2+p+2)}{2(p^2-1)^2}. \qedhere
\end{align*}
\end{proof} 
By Lemmas~\ref{lem:hog97ro}, \ref{lem:hofactorout2o}, \ref{lem:eflatbwv852} we see that $N(B)$ equals   
 $$ 
  \frac{ 3   }{ (2\pi )^{ 3/2}} \frac{49}{3}  
\l(\prod_{p\neq 2 }  \l(1-\frac{1}{p} \r)^{3/2} 
\frac{(p^2+p+1)(2p^2+p+2)}{2(p^2-1)^2}
 \r)
\frac{B^3}{(\log B)^{3/2}} +O\l(\frac{B^3}{(\log B)^{5/2}}\r) .$$  
Using $48 = 6 \cdot 8$, this is easily checked to agree with the expression in Theorem~\ref{thm:main}.  \qed

\begin{remark} \label{rem:equiv}
	An alternative  to \eqref{def:nalphax}, also considered by Serre in \cite{Ser90},
	is 
	$$N_0(B) = \#\bigg\{\b t \in (\Z\setminus\{0\} )^3: 
	\max_i|t_i| \leq B,  \sum_{i=0}^3 t_i x_i^2=0 \textrm{ has a }\Q\textrm{-point}
	\bigg\},$$
	i.e.~counting all $\mathbf{t}$ and not just primitive $\mathbf{t}$. A simple
	application of M\"{o}bius inversion, Proposition \ref{prop:local_densities}, and 
	Lemma \ref{lem:oddprimeshaar} shows that Theorem \ref{thm:main} is equivalent to
	$$N_0(B)=\frac{2}{\pi^{3/2}   }\Bigg(
 \vartheta_\infty  
\prod_{\substack{ p \text{ \rm prime } } }
 \frac{   \vartheta_p}{      (1-1/p )^{3/2} }
\Bigg)
  \frac{B^3}{(\log B)^{3/2} }+O\left( \frac{B^3}{(\log B)^{5/2} } \right) ,$$
where
$  \vartheta_\infty $ is the Lebesgue measure of $\b t$ in $[-1,1]^3$ for which the conic $\sum_{i=0}^2t_ix_i^2$ has an $\RR$-point
and 
$ \vartheta_p$ is the  $p$-adic Haar measure of   
$\b t \in \Z_p^3$ for which $\sum_{i=0}^2t_ix_i^2$ 
 has a $\Q_p$-point. 
\end{remark}

\end{document}